\documentclass[11pt,reqno]{amsart}

\usepackage[all]{xy}
\usepackage{amsmath,amssymb,amscd,graphicx,psfrag}
\usepackage{mathrsfs}
\usepackage{cite}
\usepackage{amsfonts}
\usepackage[active]{srcltx}
\newcommand{\rmv}[1]{}

\newcommand{\eH}{\ensuremath{\mathscr{H}}}

\newcommand{\SL}{\script{SL}}
\newcommand{\RB}{\script{RB}}
\newcommand{\Ba}{\script{B}}
\newcommand{\LC}{\script{LC}}
\newcommand{\RC}{\script{RC}}
\newcommand{\Co}{\script{C}}
\newcommand{\RZ}{\script{RZ}}
\newcommand{\LZ}{\script{LZ}}

\newcommand{\U}{\script{U}}
\newcommand{\simple}{\smallcaps{Simple}}

\newcommand{\rank}{\smallcaps{Rank}}

\newcommand{\Triv}{{\ensuremath{\{\pv 1\}}}}
\newcommand{\Null}{{\script{NULL}}}
\newcommand{\Dk}{{\ensuremath{\script{D}_k}}}
\newcommand{\Nilk}{{\ensuremath{\script{N}_k}}}
\newcommand{\Commn}{{\ensuremath{\script{COM}_{(m,n)}}}}

\newcommand{\Zp}{{\ensuremath{\langle \mathbb{Z}_p \rangle}}}
\newcommand{\rRh}{{\ensuremath{\wedge \mathscr R}}}
\newcommand{\Rh}{{\ensuremath{\wedge \mathscr L}}}

\newcommand{\KR}{{\ensuremath{\wedge \ell\mathbf 1^{op}}}}
\newcommand{\rKR}{{\ensuremath{\wedge \ell\mathbf 1}}}

\newcommand{\BR}{{\ensuremath{BR}}}

\newcommand{\Mac}{{\ensuremath{M^c}}}

\newcommand{\V}{\script{V}}

\newcommand{\ex}{{\ensuremath{E}}}
\newcommand{\exa}{{\ensuremath{E}}}
\newcommand{\exb}{{\ensuremath{F}}}

\newcommand{\str}{\smallcaps{Str}}

\newcommand{\script}[1]{\ensuremath{\mathcal{#1}}}
\newcommand{\smallcaps}[1]{\mathsf {#1}}
\newcommand{\category}[1]{\ensuremath{\mathbf{#1}}}
\newcommand{\field}[1]{\ensuremath{\mathbb{#1}}}
\newcommand{\N}{\field{N}}
\newcommand{\Z}{\field{Z}}
\newcommand{\SG}{\category{Sgp}}
\newcommand{\FS}{\category{FSgp}}
\newcommand{\FJ}{\category{FJSgp}}
\newcommand{\SGA}{\category{Sgp_A}}
\newcommand{\FSA}{\category{FSgp_A}}
\newcommand{\FJA}{\category{FJSgp_A}}

\newcommand{\A}{\script{A}}
\newcommand{\burnside}{\script{B}}
\newcommand{\sch}{\mathsf{Sch}}
\newcommand{\cut}{\mathrm{cut}}
\newcommand{\lp}{\mathsf{lp}}
\newcommand{\slp}{\mathsf{slp}}
\newcommand{\Cay}{\mathsf {Cay}}
\newcommand{\red}{\mathrm{Red}}

\newcommand{\wh}{\widehat}

\newcommand{\pv}[1]{\ensuremath{{\bf#1}}}

\newcommand{\inv}{^{-1}}
\newcommand{\p}{\varphi}

\newcommand{\pinv}{{\p \inv}}
\newcommand{\J}{\mathrel{\mathscr J}} 
\newcommand{\R}{\mathrel{\mathscr R}} 
\newcommand{\eL}{\mathrel{\mathscr L}} 

\newcommand{\HH}{\mathrel{\mathscr H}}

\newcommand{\ov}[1]{\ensuremath{\overline {#1}}}

\newcommand{\til}[1]{\ensuremath{\widetilde {#1}}}

\newcommand{\NN}{\ensuremath{\mathbb N}}
\newcommand{\malce}{\mathbin{\hbox{$\bigcirc$\rlap{\kern-8.3pt\raise0,50pt\hbox{$\mathtt{m}$}}}}}

\newcommand{\Fr}[1]{\mathrm{Fr}(#1)}
\newcommand{\Simple}{\mathsf{Simple}}

\newcommand{\Thmname}{Theorem}
\newcommand{\Propname}{Proposition}
\newcommand{\Lemmaname}{Lemma}
\newcommand{\Definitionname}{Definition}
\newtheorem{Thm}{\Thmname}[section]
\newtheorem{Prop}[Thm]{\Propname}
\newtheorem{Lemma}[Thm]{\Lemmaname}
{\theoremstyle{definition}
}
{\theoremstyle{definition}
\newtheorem{Def}[Thm]{\Definitionname}}
{\theoremstyle{remark}
\newtheorem{Rmk}[Thm]{Remark}}
{\theoremstyle{remark}
\newtheorem{exmp}[Thm]{Example}}
\newtheorem{Cor}[Thm]{Corollary}
\newtheorem{Example}[Thm]{Example}

\newtheorem{Conjecture}{Conjecture}

\newtheorem{thm}[Thm]{Theorem}

\newtheorem{lem}[Thm]{Lemma}

\newtheorem{cor}[Thm]{Corollary}
{\theoremstyle{remark}
}
{\theoremstyle{remark}
}
{\theoremstyle{remark}
\newtheorem*{Claim*}{Claim}}

\numberwithin{equation}{section}
\numberwithin{figure}{section}

\title[Geometric Semigroup Theory]
    {Geometric Semigroup Theory}
\author[J.~M{c}Cammond]{Jon M{c}Cammond}
      \address{Dept. of Mathematics\\
               University of California \\
               Santa Barbara, CA 93106}
      \email{mccammon@math.ucsb.edu}
\author[J.~Rhodes]{John Rhodes}
      \address{Dept. of Mathematics\\
               University of California\\
               Berkeley, CA 94720}
      \email{rhodes@math.berkeley.edu}
\author[B.~Steinberg]{Benjamin Steinberg}
      \address{School of Mathematics and Statistics\\
               Carleton University\\
               Ottawa, ON K1S 5B6\\ Canada}
      \email{bsteinbg@math.carleton.ca}
\thanks{The first author was partially supported by the National Science Foundation.  The third author gratefully acknowledges the support of NSERC and the DFG}
\date{\today}
\begin{document}

\begin{abstract}
Geometric semigroup theory is the systematic investigation of
finitely-generated semigroups using the topology and geometry of their
associated automata.  In this article we show how a number of
easily-defined expansions on finite semigroups and automata lead to simplifications
of the graphs on which the corresponding finite semigroups act.  We
show in particular that every finite semigroup can be finitely
expanded so that the expansion acts on a labeled directed graph which
resembles the right Cayley graph of a free Burnside semigroup in many
respects.
\end{abstract}
\maketitle
\tableofcontents

\section{Introduction}
Geometric semigroup theory is the systematic investigation of
finitely-generated semigroups using the topology and geometry of their
associated automata. An early example of this approach is the article
\cite{Mc91} where the first author proved that the Burnside semigroups
\[\burnside(m,n) = \langle A\mid x^m=x^{m+n}\rangle\] (for fixed $m\geq
6$ and $n\geq 1$) are finite $\J$-above, have a decidable word
problem, and their maximal subgroups are cyclic.  In addition, and
perhaps most importantly, the Brzozowski conjecture --- that the
equivalence classes of elements form regular languages --- was verified
in this range.
Independent proofs of these results were obtained by A.~de Luca and
S.~Varricchio~\cite{deLuVa92} at about the same time, and shortly
thereafter additional cases were covered by A.~do Lago~\cite{La92} and
V.~Guba~\cite{Gu93}.  The techniques used in these other papers,
however, were predominantly combinatorial in nature.  In the present
article we wish to generalize the geometric nature of the arguments
used in~\cite{Mc91}.  In particular, we will
prove the following:

\renewcommand{\thethm}{\Alph{thm}}
\begin{thm}[Rough statement]\label{thm:main}
If $S$ is a finite $A$-semigroup, then there is a finite expansion of
$S$ which acts faithfully on a labeled directed graph which has many
of the nice geometric properties possessed by the right Cayley graph
of a Burnside semigroup $\burnside(m,n)$, $m\geq 6$.
\end{thm}
\renewcommand{\thethm}{\thesection.\arabic{thm}}

In the course of the article, we will explicate the particular
expansions involved and the precise nature of the resemblance.  

The proofs in~\cite{Mc91} involved a detailed examination of the
automata which recognize the equivalence classes of words in $A^+$
under the relations defining a Burnside semigroup.  In particular, for
each word $w \in A^+$, the equivalence class $[w]$ of words equal to
$w$ in $\burnside(m,n)$ was described as the language accepted by a
non-deterministic finite-state automaton with a fractal-like
structure.  The deterministic version of this automaton is a full
subautomaton of the right Cayley graph of $\burnside(m,n)$, and its
geometric properties were described in some detail by the second
author in~\cite{flows}.  The main theorem will follow immediately
once we have shown that every finite automata can be finitely
expanded, so that the expanded automata closely resembles these
``McCammond automata'' as described in~\cite{Mc91} and
\cite{flows}.

\section{The topology of directed graphs}\label{sec: graphs}
We view here automata from two angles: as labeled directed graphs and as universal algebras.  The former viewpoint is geometric, whereas the latter is algebraic.   This section establishes the essential properties of the topology of directed graphs that we shall need.  Probably nothing in the first two subsections is original although maybe our slant is different.  To some extent it follows~\cite{geoams}.

\subsection{Directed graphs}
We begin with the definition of a directed graph.

\begin{Def}[Graph]
A (directed) \emph{graph} $\Gamma$ consists of a set $V(\Gamma)$ of vertices, $E(\Gamma)$ of edges and two maps $\iota,\tau\colon E(\Gamma)\to V(\Gamma)$ selecting the initial, respectively, terminal vertices of an edge $e$. Often we write $e\colon v\to w$ to indicate $\iota(e)=v$ and $\tau(e)=w$.
\end{Def}

Directed and undirected paths in a graph are defined in the usual way.  If $p$ is an undirected path, then $\iota(p),\tau(p)$ will denote the initial and terminal vertices of $p$, respectively and we shall write $p\colon \iota(p)\to \tau(p)$.   We admit an empty path at each vertex.  When we say ``path'' without any modifier, we mean a directed path, although we may include the word ``directed'' for emphasis. A directed or undirected path is called (vertex) \emph{simple} if it visits no vertex twice; empty paths are considered simple.  By a \emph{circuit} we mean a non-empty closed path $p$ (i.e., $\iota(p)=\tau(p)$).  A circuit is called \emph{simple} if the only repetition in the vertices it visits is when it returns to its origin.  An undirected path is called \emph{reduced} if it contains no backtracking (i.e., no subpath of length $2$ using an edge first in one direction and then in the other).  The inverse of an undirected path is defined in the usual way.

A graph is \emph{connected} if there is an undirected path from any vertex to any other.  A connected graph is called a \emph{tree} if it contains no reduced undirected circuits.  By an \emph{induced} or \emph{full} subgraph we mean a subgraph obtained by considering some subset of vertices and \emph{all} edges between them.

There is a natural preorder on the vertices of any directed graph.

\begin{Def}[Accessibility order]
Let $\Gamma$ be a graph. Define a preorder on $V(\Gamma)$ by $v\prec w$ if there is a path from $w$ to $v$.  A graph is termed \emph{acyclic} if $\prec$ is a partial order (equivalently, there are no directed circuits in $\Gamma$). As usual, an equivalence relation can be obtained form $\prec$ by setting $v\sim w$ if $v\prec w$ and $w\prec v$.   A \emph{strong component} of a graph is an induced (or full) subgraph of $\Gamma$ obtained by considering all vertices in a  $\sim$-equivalence class.

A graph $\Gamma$ is \emph{strongly connected} if it has a unique strong component.  In general, the set of strong components of $\Gamma$ is partially ordered by putting $C\geq C'$ if there is a path from a vertex in $C$ to a vertex in $C'$.
\end{Def}

It is convenient to divide strong components into two sorts: trivial and non-trivial.

\begin{Def}[Trivial strong components]
A strong component can contain no edge.  In this case, it consists of a single vertex and we shall call it \emph{trivial}.  Hence by a \emph{non-trivial} strong component, we mean a strong component with at least one edge.  In particular, a strong component with a single vertex and one or more loops edges is considered non-trivial.
\end{Def}

\begin{Rmk}
Notice that if $v\prec w$, then there is a simple path from $w$ to $v$.  Indeed, let $p\colon w\to v$ be a minimum length path.  If $p$ is not simple, we may factor it $p=uts$ with $t$ a directed circuit (in particular $t$ is non-empty).  Then $us\colon w\to v$ is shorter than $p$, a contradiction.  So $p$ must be simple.
\end{Rmk}

Other authors order vertices via the opposite convention.  Our choice was made to be compatible with Green's relations in semigroups.   That is, the accesibility order on the right Cayley graph of a semigroup (to be defined shortly) corresponds to the $\leq_{\R}$ ordering on the semigroup.

\subsubsection{Preordered sets}
It is convenient here to introduce some terminology from the theory of preordered sets.  If $(P,\leq)$ is a preordered set, then a \emph{downset} is a subset $X\subseteq P$ such that $x\in X$ and $y\leq x$ implies $y\in X$.  Downsets are also called \emph{order ideals} by some authors.  If $Y\subseteq P$, then $Y^{\downarrow}$ denotes the downset generated by $Y$.  A downset of the form $p^{\downarrow}$ with $p\in P$ is called \emph{principal}.  One can define \emph{upsets} (called \emph{filters} by some authors) dually.  The upset generated by $Y$ will be denote $Y^{\uparrow}$.   A subset $X$ of a preordered set $P$ is said to be \emph{convex} if $x\prec y\prec z$ and $x,z\in X$ implies $y\in X$.  The \emph{equivalence classes} of $P$ are defined by $p\sim q$ if $p\leq q$ and $q\leq p$.  A preordered set $P$ is a \emph{chain} if any two elements of $P$ are comparable.

A \emph{principal series} for $P$  is an unrefinable chain
\begin{equation}\label{principalseries}
P=P_0\supset P_1\supset \cdots \supset P_n
\end{equation}
of principal downsets.  Every finite preordered set has a principal series.  A principal series for a poset amounts to a topological ordering of the poset.  In general, one can verify that $P_i\setminus P_{i+1}$ is always an equivalence class of $P$ and that every equivalence class must arise this way.

Suppose now that $V$ is the vertex set of a directed graph $\Gamma$ and order $V$ by the accessibility order.  Then a downset in $V$ is a subset $X$ of vertices with the property that if the initial vertex of an edge belongs $X$, then so does the terminal vertex.  The equivalence classes of the accessibility ordering are the strong components and a principal series amounts to the same thing as a topological ordering on the strong components.  Therefore, when we assign indices to the strong components, we use notation consistent with \eqref{principalseries}.

\subsubsection{Transition edges}

An important role in the theory is played by those edges that go between strongly connected components. One facet of geometric semigroup theory is to simplify the structure of these edges.

\begin{Def}[Transition edge]
An edge $e$ of a graph $\Gamma$ is called a \emph{transition edge} if $\tau(e)\nsim \iota(e)$ (or equivalently, there is no directed path from $\tau(e)$ to $\iota(e)$). The \emph{frame} $\Fr \Gamma$ of $\Gamma$ is the graph with vertex set the strong components of $\Gamma$ and edge set the transition edges of $\Gamma$ (i.e., we contract each strong component to a point).  If $e$ is a transition edge, then $e$ starts at the strong component of its initial vertex and ends at the strong component of its terminal component.  Evidently, $\Fr \Gamma$ is an acyclic graph.
\end{Def}

The edge set $E$ of a directed graph $\Gamma$ is also a preordered set.  One can define $e\prec f$ if either $e=f$, or there is a directed path in $\Gamma$ of the form $epf$ where $p$ is some path.  Notice that distinct edges $e,f$ are equivalent if and only if they belong to the same strong component.  Transition edges are precisely those edges belonging to a singleton equivalence class.  In particular, the transition edges form a poset.

It is particularly important in geometric semigroup theory when the strong components form a chain.

\begin{Def}[Quasilinearity]\label{def:quasilinear}
If $\Gamma$ is a directed graph we say that $\Gamma$ is
\emph{quasilinear} if the natural partial order on its strong
components is a total ordering.  Vertices in the top-most strong
component of a quasilinear directed graph (if one exists) will be called
\emph{top-most vertices}.  In a finite graph, when the strong components are linearly
ordered in this fashion we will number them starting with $0$ and
beginning with the top-most component to be consistent with \eqref{principalseries}.
\end{Def}

Of course, a quasilinear graph is connected and every strongly connected graph is quasilinear.
The class of quasilinear graphs for which the transition edges form a chain plays a salient role in geometric semigroup theory.  In this case $\Fr \Gamma$ looks like a line, whence the following terminology.

\begin{Def}[Linearity]
A quasilinear graph $\Gamma$ is said to be \emph{linear} if its transition edges form a chain.
\end{Def}

\begin{Rmk}
If $\Gamma$ is a finite linear graph with
exactly $k+1$ strong components (numbered $0,1,\ldots,k$), then it has exactly
$k$ edges which do not belong to strong components, (which then
necessarily connect the $(i-1)^{st}$ strong component to the $i^{th}$
strong component, $i=1,\ldots,k$).
\end{Rmk}

\begin{Def}[Entry and exit points]\label{def:exits}
If $\Gamma$ is a finite linear graph, the unique
transition edge connecting the $(i-1)^{st}$ strong component to the
$i^{th}$ will be called the \emph{$i^{th}$ transition edge}.  Its start
point will be denoted $q_{i-1}$ and its end point will be denoted
$p_i$.  Notice that the subscripts indicates the strong component
which contains the vertex.  Since the vertices $p_i$ and $q_i$ are the
places where transition edges enter and exit the $i^{th}$ strong
component, we will sometimes refer to these vertices as the
\emph{entry and exit points} of the $i^{th}$ component.  Note that $p_i=q_i$ is possible.

If $p=p_0$ is a specified vertex in the $0^{th}$ strong component and
$q=q_k$ is a specified vertex in the $k^{th}$ strong component, then
any simple directed path in $\Gamma$ from $p$ to $q$ will be called a
\emph{quasi-base} for $\Gamma$.  Notice that every quasi-base for $\Gamma$ will
contain the transition edges plus simple paths in each strong component
connecting $p_i$ to $q_i$, and that conversely, any choice of simple
paths connecting $p_i$ to $q_i$, $i=0,\ldots,k$ can be strung together
with the transition edges to form a quasi-base.  In the case that there is a unique simple path from $p$ to $q$ (i.e., there is a unique simple path from $p_i$ to $q_i$ for each $i$) we shall call the corresponding quasi-base a \emph{base}.
\end{Def}

\begin{exmp}
An example of a finite linear graph has been
schematically drawn in Figure~\ref{fig:quasi-base}.  The shaded areas
are meant to represent non-trivial strong components.  In this
example, there are $6$ strong components and there are $5$ transition
edges, and their numbering has been illustrated.  The first strong
component is trivial (i.e., has no edges).  The well-defined entry and
exit points for each strong component have also highlighted.  Notice
that in this example $p_3$ and $q_3$ are identical, so that any base
must use the trivial path to connect $p_3$ to $q_3$.
\end{exmp}

\begin{figure}[ht]
\psfrag{1}{$1$}
\psfrag{2}{$2$}
\psfrag{3}{$3$}
\psfrag{4}{$4$}
\psfrag{5}{$5$}
\includegraphics{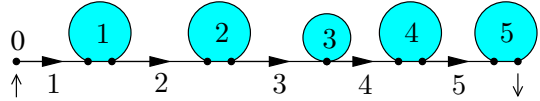}
\caption{A directed graph with a quasi-base.}\label{fig:quasi-base}
\end{figure}

\subsection{Morphisms of directed graphs}
A \emph{morphism} of graphs $\p\colon \Gamma\to \Gamma'$ consists of a pair $(\p_V,\p_E)$ of maps $\p_V\colon V(\Gamma)\to V(\Gamma')$, $\p_E\colon E(\Gamma)\to E(\Gamma')$ so that $\p_V(\iota(e)) = \iota(\p_E(e))$ and $\p_V(\tau(e)) = \tau(\p_E(e))$ for all $e\in E(\Gamma)$.  Normally we use $\p$ to denote both maps.  There is an obvious way to extend $\p$ from edges to paths.

Two especially important classes of morphisms are \emph{directed coverings} and \emph{directed immersions}.  If $v\in V(\Gamma)$, then the \emph{star} of $v$ is $St(v) = \iota^{-1}(v)$.

\begin{Def}[Directed coverings and immersions]
A graph morphism $\p\colon \Gamma\to \Gamma'$ is called a \emph{directed covering} if it is surjective on vertices and, for each vertex $v\in V(\Gamma)$, the induced map $\p\colon St(v)\to St(\p(v))$ is a bijection.  If $\p$ merely injective on stars, the $\p$ is called a \emph{directed immersion}.  We do note require directed immersions to be surjective on vertices.
\end{Def}

\begin{Rmk}\label{directedimmersion}
Notice that if $\p$ is a directed immersion, then $\p\psi$ is a directed immersion if and only if $\psi$ is a directed immersion.
\end{Rmk}

The following topological proposition is proved by straightforward induction on the length of a path.

\begin{Prop}[Path lifting]\label{pathlift}
Let $\p\colon \Gamma\to \Gamma'$ be a graph morphism.  Then $\p$ is a directed immersion if and only if, for each (directed) path $p$ at $v'\in V(\Gamma')$ and each $v\in \pinv(v')$, there is at most one path $q$ at $v$ with $\p(q)=p$.  The map $\p$ is a directed covering if and only if it is surjective on vertices and,
for each (directed) path $p$ at $v'\in V(\Gamma')$ and each $v\in \pinv(v')$, there is a unique path $q$ at $v$ with $\p(q)=p$.
\end{Prop}

Notice that there is a bijection between sets and graphs with a single vertex.  Hence we will frequently identify an alphabet $A$ with the ``bouquet'' graph $\mathscr B_A$ consisting of a single vertex with edge set $A$.  A \emph{labeling} of $\Gamma$ by an alphabet $A$ is then a graph morphism  $\ell\colon \Gamma\to A$. We can now define automata using our topological language cf.~\cite{geoams}.  See~\cite{EilenbergA,Eilenberg,Lawsonaut} for background on automata theory.

\begin{Def}[Automaton]
A non-deterministic \emph{automaton} over the alphabet $A$ is a pair $\mathscr A=(\Gamma,\ell)$ where $\Gamma$ is a graph and $\ell\colon \Gamma\to \mathscr B_A$ is a  labeling.  A morphism of automata $\p\colon (\Gamma_1,\ell_1)\to (\Gamma_2,\ell_2)$ is a graph morphism \mbox{$\p\colon \Gamma_1\to \Gamma_2$} so that
\[\xymatrix{\Gamma_1\ar[rr]^{\p}\ar[rd]_{\ell_1}&&\Gamma_2\ar[ld]^{\ell_2}\\ &\mathscr B_A&}\] commutes.

An automaton $(\Gamma,\ell)$ is called \emph{deterministic} if $\ell$ is a directed covering; it is termed a \emph{partial deterministic} automaton if $\ell$ is a directed immersion.  By an automaton or $A$-automaton, we shall mean a partial deterministic automaton.  In the context of automata, vertices are often called states and edges are termed transitions.  If $\mathscr A=(\Gamma,\ell)$ is an $A$-automaton with vertex set $Q$, we sometimes abusively write $\mathscr A=(Q,A)$.
\end{Def}

By Remark~\ref{directedimmersion} any morphism of $A$-automata is a directed immersion.  If $\mathscr A=(\Gamma,\ell)$ is an $A$-automaton, then we can associate to it its transition monoid $M(\mathscr A)$.  Namely, we can define a ``monodromy'' action of the free monoid $A^*$ generated by $A$ on $V(\Gamma)$ via path lifting.  If $q\in V(\Gamma)$ is a vertex and $w\in A^*$ then we can view $w$ as a path $p$ in $\mathscr B_A$.  This path has at most one lift $\til p$ with $\iota(\til p) = q$ by Proposition~\ref{pathlift}.  Define $qw = \tau(\til p)$ if $\til p$ exists, and leave it undefined otherwise.  One easily checks that this defines an action of $A^*$ on $V(\Gamma)$ by partial functions; the associated faithful partial transformation monoid is denoted $M(\mathscr A)$ and is called the \emph{transition monoid} of $A$.  Note that the action is by total functions if and only if $\mathscr A$ is deterministic.  We denote by $\eta_{\mathscr A}$ the \emph{transition morphism} $\eta_{\mathscr A}\colon A^*\to M(\mathscr A)$.  Notice that $p\prec q$ if and only if $p\in q\cdot M(\mathscr A)$.  Thus the accessibility order on $V(\Gamma)$ corresponds to the inclusion ordering on cyclic $M(\mathscr A)$-invariant subsets.  The subsemigroup of $M(\mathscr A)$ generated by $A$ is called the \emph{transition semigroup} and is denoted $S(\mathscr A)$.

\subsection{Semigroups and automata}
An important example of a deterministic automaton is the Cayley graph of an $A$-semigroup.

\begin{Def}[$A$-semigroup]
An \emph{$A$-semigroup} is a pair $(S,\p)$ where $\p\colon A^+\to S$ is a surjective homomorphism, where $A^+$ denotes the free semigroup on $A$.  To avoid reference to $\p$, we put $[w]_S=\p(w)$ for $w\in A^+$.
\end{Def}

If $S$ is a semigroup, then $S^I$ denotes $S$ with an adjoined identity $I$ (even if $S$ was already a monoid).  If $S$ is an $A$-semigroup, it is convenient to consider the empty word as mapping to $I$.

\begin{Def}[Cayley graph]
Let $S$ be an $A$-semigroup.  The \emph{right Cayley graph $\Cay(S,A)$ of $S$ with
  respect to $A$} is the deterministic automaton with vertex set $S^I$ and edge set $S^I\times A$.  Here $\iota(s,a)=s$ and $\tau(s,a) =sa$.  The edge $(s,a)$ is usually drawn $s\xrightarrow{\, a\,} sa$. The advantage of this ``monoid'' Cayley graph is that the non-trivial
paths from $I$ correspond exactly to the words in $A^+$ and two such
paths have the same endpoints if and only if the words corresponding
to these paths represent the same element in $S$.
\end{Def}

The strong components of these graphs are also of interest.

\begin{Def}[Sch\"utzenberger graphs]
If $S$ is an $A$-semigroup then the strong components of
$\Cay(S,A)$ are the \emph{Sch\"utzenberger graphs} of the
$\R$-classes of $S^I$.  In other words, the vertex set of a strong
component contains exactly those vertices which represent the elements
in an $\R$-class, and the strong component itself is the full subgraph
of $\Cay(S,A)$ on this vertex set.  For each word $w \in A^+$, we
will denote the strong component of $\Cay(S,A)$ containing the
vertex labeled $[w]_S$ by $\sch^S(w)$ (where $\sch$ stands for
Sch\"utzenberger).  Suppose $[w]_S = s$. Since $\sch^S(w)$ only depends
on the element $s$ and not on the word $w$, we sometimes write
$\sch^S(s)$ instead.  The study of the way in which $S$ acts on its
Sch\"utzenberger graphs has been dubbed the \emph{semilocal
  theory}. See Chapter~\cite[Chapter 8]{Arbib} and~\cite[Chapter 4]{qtheor}.
\end{Def}

Directed coverings correspond to transformation semigroup homomorphisms.  In this paper, a \emph{partial transformation semigroup} is a pair $(X,S)$ where $S$ is a semigroup acting faithfully on the right of $X$ by partial transformations.  Recall that if $(X,S)$ and $(Y,T)$ are partial transformation semigroups, then a \emph{morphism} is a pair $(\p,\psi)$ where $\p\colon X\to Y$ is a function and $\psi\colon S\to T$ is a homomorphism such that $\p(xs)=\p(x)\psi(s)$ for all $x\in X$ and $s\in S$, where equality means that either both sides are undefined or both are defined and equal. If $S$ and $T$ are $A$-generated, we shall call $(\p,\psi)$ a morphism of $A$-partial transformation semigroups if $\psi$ is a homomorphism of $A$-semigroups.

The following lemma is standard~\cite{Eilenberg}.

\begin{Lemma}\label{morphismoftransformation}
Let $(X,S)$ and $(Y,T)$ be partial transformation semigroups and $\p\colon X\to Y$ a surjective function so that, for all $s\in S$, there exists $\wh s\in T$ such that $\p(xs)=\p(x)\wh s$ for all $x\in X$ and $s\in S$ (interpreting equality as above).  Then there is a unique homomorphism $\psi\colon S\to T$ so that $(\p,\psi)$ is a morphism.
\end{Lemma}
\begin{proof}
First we show that $\wh s$ is unique.  Suppose $t,t'\in T$ satisfy $\p(xs)=\p(x)t=\p(x)t'$ for all $x\in X$, $s\in S$.  Let $y\in Y$  with $yt$ defined and choose $x\in \pinv(y)$.  Then $\p(x)t$ is defined and hence $\p(xs)$ is defined and so $\p(x)t'$ is defined.   Moreover, $yt=\p(xs)=yt'$.   Similarly, $yt'$ defined implies $yt$ is defined and $yt'=yt$.  Thus by faithfulness $t=t'$.  Hence the element $\wh s$ in the hypothesis is unique.  Define $\psi\colon S\to T$ by $\psi(s)=\wh s$.  Notice for $s,s'\in S$ that $xss'$ is defined if and only if $\p(xs)\psi(s')$ is defined, if and only if $\p(x)\psi(s)\psi(s')$ is defined and that $\p(xss')=\p(x)\psi(s)\psi(s')$.  The uniqueness then implies $\psi(ss') = \psi(s)\psi(s')$.  Finally the uniqueness of $\psi$ follows from the uniqueness of $\wh s$.
\end{proof}

We now verify that, for partial deterministic automata, a morphism is a directed covering if and only if it is surjective on vertices and induces a morphism of partial transformation semigroups.

\begin{Prop}\label{covershom}
Let $\mathscr A$ and $\mathscr A'$ be partial deterministic $A$-automata.  Then the following are equivalent:
\begin{enumerate}
\item A directed covering of $A$-automata $\p\colon \mathscr A\to \mathscr A'$;
\item A surjective morphism $(\p,\psi)\colon (V(\mathscr A),S(\mathscr A))\to (V(\mathscr A'),S(\mathscr A'))$ where $\psi$ is a morphism of $A$-semigroups.
\end{enumerate}
 \end{Prop}
\begin{proof}
Suppose first that $\p$ is a directed covering.  We claim that for all vertices $q$ of $\mathscr A$ and $w\in A^+$, one has $\p(qw) = \p(q)w$ (with the usual meaning).  Indeed, the image of a path labeled $w$ must be labeled $w$. So $qw$ defined means $\p(q)w$ is defined and $\p(qw)=\p(q)w$.  Conversely, if $\p(q)w$ is defined, then Proposition~\ref{pathlift} implies that there is a lift of $w$ starting at $q$, which must also have label $w$.  So $qw$ is defined and $\p(qw)=\p(q)w$.  Hence we can define $\wh{[w]}_{S(\mathscr A)} = [w]_{S(\mathscr A')}$ in Lemma~\ref{morphismoftransformation} to obtain $\psi$.

Conversely, suppose $(\p,\psi)$ is well defined. Define $\p\colon  \mathscr A\to \mathscr A'$ to agree with $\p$ on vertices.  If $e\colon p\to q$ is an edge of $\mathscr A$ with label $a$, then $pa=q$ and so $\p(q)=\p(pa)=\p(p)a$. Thus there is a (unique by determinism) edge labeled by $a$ from $\p(p)$ to $\p(q)$, which we define to be $\p(e)$.  Clearly $\p$ is a morphism of $A$-automata. Let us check that it is a directed covering.  Let $q$ be a vertex of $\mathscr A$ and let $e\in \iota^{-1}(\p(q))$ be an edge labeled by $a$.  Then $\p(q)a$ is defined, so $qa$ must be defined and $\p(qa)=\p(q)a$.  Since the automaton is deterministic, there is a unique edge labeled by $a$ emanating from $q$ and it must map to $e$ under $\p$ (again by determinism).
\end{proof}

\subsection{Rooted graphs}
In this article, we shall mostly be interested in rooted graphs.

\begin{Def}[Rooted graph]
A \emph{rooted graph} is a pair $(\Gamma,v)$ where $\Gamma$ is a graph and $v$ is a vertex of $\Gamma$ such that every vertex of $\Gamma$ can be reached from $v$ by a directed path (i.e., the strong component of $v$ is the unique maximum component in the ordering on strong components).
\end{Def}

There is an analogous definition for automata.

\begin{Def}[Pointed automaton]\label{pointedauto}
By a \emph{pointed} (or \emph{initial}) automaton, we mean an $A$-automaton $\mathscr A=(\Gamma,\ell)$ with a distinguished vertex $I$ so that $(\Gamma,I)$ is a rooted graph, that is, $I\cdot M(\mathscr A)=V(\Gamma)$.  We denote the pointed automaton by $(\mathscr A,I)$.  More generally, we say that a subset $\mathsf I$ is a \emph{generating} or \emph{initial} set for $\mathscr A$ if $\mathsf I\cdot M(\mathscr A) = V(\Gamma)$, that is, the downset generated by $\mathsf I$ is $V(\Gamma)$.  In this case, we indicate the generating set by writing $(\mathscr A,\mathsf I)$.
\end{Def}

Often it is important to endow an automaton with initial and terminal states in order to accept a language.

\begin{Def}[Acceptor]
A \emph{non-deterministic acceptor} over an alphabet $A$ is a non-deterministic $A$-automaton $\mathscr A$ equipped with a distinguished initial state $I$ and a set $T$ of \emph{terminal states}.   The language of the acceptor consists of all words in $A^+$ labeling a path from $I$ to an element of $T$.   If we use the word acceptor unmodified, then the underlying automaton is assumed partial deterministic.  By a deterministic acceptor, we mean one in which the underlying automaton is deterministic.  The languages accepted by finite acceptors are the so-called \emph{regular} or \emph{rational} languages.

An acceptor $(\mathscr A,I,T)$ is said to be \emph{trim}
if every vertex is contained in a directed path from the initial state
to some terminal state.  In particular, trim acceptors are rooted at $I$.
\end{Def}

For example, the Sch\"utzenberger graph $\sch^S(s)$ can be turned into
an acceptor by specifying the vertex labeled $s$ as both the
initial state and its only terminal state. Whenever we refer to $\sch^S(s)$
as an acceptor without specifying initial and terminal states, this is what
we intend.

\begin{Def}[Reading words]\label{def:reading}
Let $\mathscr A$ be an $A$-automaton.  Examples include the Cayley graph $\Cay(S,A)$ and $\sch^S(s)$.  A word $w\in A^+$ is said to be \emph{readable on $\mathscr A$}
if there exists a directed path in $\mathscr A$ where the concatenation of
labels is the word $w$.  Similarly, if $v$ is a vertex $\mathscr A$ then
being \emph{readable starting at $v$} or \emph{readable ending at $v$}
has the obvious meaning.
\end{Def}

Using this language, we can restate the advantage of the Cayley graph
$\Cay(S,A)$ as follows.  Every word $w\in A^+$ is readable starting
at $I$, and because $\Cay(S,A)$ is deterministic, $w$ is readable
starting at $I$ in exactly one way.

An important fact is that the Cayley graph of the transition semigroup of a complete deterministic pointed automaton is always a directed cover of the automaton.

\begin{Prop}\label{transitionsemigroupcovers}
Let $\mathscr A=(\Gamma,v)$ be a pointed deterministic $A$-automaton.  Then there is a directed covering $\rho\colon \Cay(S(\mathscr A),A)\to \mathscr A$ of $A$-automata given by $s\mapsto vs$ on vertices.
\end{Prop}
\begin{proof}
The map $\rho$ is clearly is surjective and satisfies $\rho(ts)=vts=\rho(t)s$.  Proposition~\ref{covershom} yields the desired result.
\end{proof}

A rooted graph $(\Gamma,v)$ is called a \emph{directed rooted tree} if $\Gamma$ is a tree.   Directed rooted trees are characterized by having a unique directed path from the root to any vertex.

\begin{Prop}\label{rootedtrees}
A rooted graph $(\Gamma,v)$ is a directed rooted tree if and only if, for each vertex $w\in V(\Gamma)$, there is a unique directed path from $v$ to $w$.  This path is necessarily simple.
\end{Prop}
\begin{proof}
Suppose first that $\Gamma$ is a rooted directed tree and that there are two directed paths $p,q\colon v\to w$.  Then by considering the longest common initial and terminal segments of $p,q$ we can write $p=urs$ and $q=uts$ so that $r,t$ begin and end with different edges.  Then $rt\inv$ is a reduced undirected circuit in $\Gamma$ and so $\Gamma$ is not a tree.  Suppose conversely, that $\Gamma$ is not a tree.  Consider a reduced undirected circuit $p$ in $\Gamma$. Let $w=\iota(p)$ and let $q\colon v\to w$ be a directed path. Replacing $p$ by a cyclic conjugate of its reverse circuit if necessary, we may assume the first edge of $p$ is traversed in the positive direction.   If $p$ is a directed circuit, then $q,qp$ are two directed paths from $v$ to $w$ and we are done. Otherwise, we may factor $p=ue^{-1}s$ where $u$ is the longest directed initial segment of $p$ and $e^{-1}$ indicates that $e$ is traversed backwards.
Let $e\colon a\to b$.  Then $b=\tau(u)$.  Let $r\colon v\to a$ be a directed path.   See Figure~\ref{fig:whataretrees}.
\begin{figure}
\[\SelectTips{cm}{}
\xymatrix{v\ar@/^1pc/[r]^q\ar@/_2pc/[rrr]_r&w\ar[r]^u &b&a\ar[l]_e\ar@/_2pc/[ll]_s}
\]
\caption{\label{fig:whataretrees}}
\end{figure}Then $re$ and $qu$ are two directed paths from $v$ to $b$.  We claim they are distinct. Indeed, if the last edge of $u$ were $e$, then $p=ue^{-1}s$ would not be reduced.  This completes the proof.

If $(\Gamma,v)$ is a directed rooted tree and $p\colon v\to w$ is the unique directed path, then obviously $p$ is simple.  For if we could write $p=qur$ where $u$ is a non-empty path with $\iota(u)=\tau(u)$, then $u$ is a reduced circuit in $\Gamma$, contradicting that $\Gamma$ is a tree.
\end{proof}

The unique directed path from $v$ to $w$ is denoted $[v,w]$ and called the \emph{geodesic} from $v$ to $w$.  More generally, if $T$ is a rooted directed tree and $w\leq u$, then there is a unique (directed) simple path from $u$ to $w$, which we denote $[u,w]$ and call the geodesic from $u$ to $w$.

Let $(\Gamma,v)$ be a rooted graph.  By a \emph{directed spanning tree} $T$ for $(\Gamma,v)$ we mean a directed rooted subtree $(T,v)$ containing all the vertices of $\Gamma$.  The next proposition shows that every rooted graph admits a directed spanning tree.  One should think of a directed spanning tree as a collection of normal forms for each vertex of $\Gamma$.

\begin{Prop}\label{hasspanningtree}
Let $(\Gamma,v)$ be a rooted graph. Then $(\Gamma,v)$ admits a directed spanning tree.
\end{Prop}
\begin{proof}
Let $\mathscr C$ be the collection of all directed rooted subtrees $(T,v)$ of $(\Gamma,v)$ ordered by inclusion.  It is non-empty since it contains $(\{v\},v)$.  Next observe that if $\{(T_{\alpha},v)\mid \alpha\in A\}$ is a chain in $\mathscr C$, then $(\bigcup_{\alpha\in A} T_{\alpha},v)\in \mathscr C$ since any reduced undirected circuit in this union must belong to some $T_{\alpha}$.  Thus by Zorn's Lemma, $\mathscr C$ contains a maximal element $(T,v)$.  We claim that this is the desired directed spanning tree.  For suppose $V(T)\subsetneq V(\Gamma)$.  We claim that there is an edge $e\in E(\Gamma)$ with $\iota(e)\in V(T)$ and $\tau(e)\notin V(T)$.  For if this is not the case, then $V(T)$ is a downset in the preorder $\prec$ on $V(\Gamma)$.  But $v\in V(T)$ and $\Gamma$ is rooted at $v$.  Thus $V(T)=V(\Gamma)$.  So let $e$ be such an edge.  One easily verifies that the graph $T'$ obtained by adjoining $e$ to $T$, so  $V(T') = V(T)\cup \{\tau(e)\}$ and $E(T') = E(T)\cup \{e\}$, is a tree. Indeed, if $w\in V(T)$, then the unique directed path in $T'$ from $v$ to $w$ is the geodesic in $T$, whereas the unique directed path from $v$ to $\tau(e)$ is $pe$ where $p$ is the geodesic $[v,\iota(e)]$ in $T$.  Thus $T'$ is a directed rooted tree by Proposition~\ref{rootedtrees}.
Clearly $(T',v)$ is a larger element of $\mathscr C$ than $T$.  This contradiction completes the proof.
\end{proof}

\begin{Rmk}\label{treecompletion}
The proof of Proposition~\ref{hasspanningtree} can easily be adapted to prove that if $(T_0,v)$ is any directed rooted subtree of a rooted graph $(\Gamma,v)$, then there is a directed spanning tree $(T,v)$ containing $(T_0,v)$.
\end{Rmk}

It turns out that every rooted directed graph $(\Gamma,v)$ has a unique directed cover that is a tree.  Moreover, this tree is a directed cover of all directed covers of $(\Gamma,v)$ and hence is called the \emph{universal directed cover} of $(\Gamma,v)$.

\begin{Thm}
Let $(\Gamma,v)$ be a rooted graph.  Then there is a rooted tree $(\til \Gamma,\til v)$ and a directed covering $\pi\colon (\til \Gamma,\til v)\to (\Gamma,v)$.  Moreover, given a directed covering $\p\colon (\Gamma',v')\to (\Gamma,v)$, there is a unique morphism (necessarily a directed covering) $\psi\colon (\til \Gamma,\til v)\to (\Gamma',v')$ so that \[\xymatrix{(\til \Gamma,\til v)\ar[rr]^{\psi}\ar[dr]_{\pi} & & (\Gamma',v')\ar[dl]^{\p}\\ & (\Gamma,v) &}\] commutes.
\end{Thm}
\begin{proof}
The proof is so similar to the classical undirected case that we just give the construction of $\til \Gamma$ and leave the remaining details to the reader.  The vertices of $\til \Gamma$ are the directed path starting at $v$.  One takes $\til v$ to be the empty path at $v$.  The edges of $\til \Gamma$ consist of pairs $(p,e)$ where $p$ is a path from $v$ and $\tau(p)=\iota(e)$.  The incidence functions are given by $\iota(p,e)=p$ and $\tau(p,e)=pe$.  The directed covering $\pi$ is given by $\pi(p)=\tau(p)$ on vertices and $\pi(p,e) = e$ on edges.
\end{proof}

For example, they Cayley graph of $A^*$ is the universal directed cover of the bouquet $\mathscr B_A$, as well as of any other deterministic $A$-automaton.

\subsubsection{The fundamental monoid of a rooted graph}
If $(\Gamma,v)$ is a rooted graph and $v\in V(\Gamma)$, then one can define the \emph{fundamental monoid} $\Gamma^*(v)$ of $(\Gamma,v)$ to be the monoid of all loops at $v$ with the concatenation product.    For instance, the fundamental monoid of the bouquet $\mathscr B_A$ at its unique vertex is the free monoid $A^*$.  Notice that the fundamental monoid ignores all of $\Gamma$ except the strong component of $v$ and moreover, it depends on the root vertex even for strongly connected graphs. Thus  one should really work with the free category $\Gamma^*$ on the graph $\Gamma$~\cite{Mac-CWM}. Analogously to the case of fundamental groups, the fundamental monoid of a rooted graph is free.

\begin{Prop}\label{fundisfreemonoid}
 Let $(\Gamma,v)$ be a rooted graph.  Then $\Gamma^*(v)$ is free on the set $P_v$ of non-empty paths $p\colon v\to v$ that do not visit $v$ except for at the beginning and the end.
\end{Prop}
\begin{proof}
Indeed, if $q\colon v\to v$ is any non-empty loop, then it has a unique factorization $q=p_1\cdots p_n$ with $p_1,\ldots, p_n\in P$ by partitioning $q$ according to each time it visits $v$.  Thus $\Gamma^*(v)$ is free on $P$.
\end{proof}

The next example shows that the generating set $P_v$ need not be finite even when the graph $\Gamma$ is finite and also exhibits the dependence on the generating set.

\begin{exmp}
Consider the strongly connected graph $\Gamma$
in Figure~\ref{infinitebasis}.
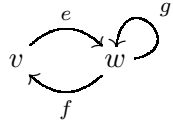
\begin{figure}
\[\SelectTips{cm}{}
\xymatrix{v\ar@/^1pc/[r]^e &w\ar@/^1pc/[l]^f\ar@(r,u)_g}
\]
\caption{An infinitely generated fundamental monoid\label{infinitebasis}}
\end{figure}
Then $\Gamma^*(v)$ is freely generated by the infinite set $\{eg^nf\mid n\geq 0\}$, whereas $(\Gamma,w)^*$ is freely generated by $fe,g$.
\end{exmp}

The next result generalizes the situation for free monoids that can be found, for instance, in~\cite{berstelperrinreutenauer}.
A submonoid $N$ of a monoid $M$ is called \emph{right unitary} if $u,uv\in N$ implies $v\in N$ for $u,v\in M$.

\begin{Prop}\label{coveringtheory}
Let $\p\colon (\til \Gamma,\til v)\to (\Gamma,v)$ be a directed covering of rooted graphs.  Then the induced map $\p\colon \til \Gamma^*(\til v)\to \Gamma^*(v)$ is injective and the image is right unitary.  Conversely, every right unitary submonoid of $\Gamma^*(v)$ is of this form.
\end{Prop}
\begin{proof}
Since $\p$ is a directed covering, any loop at $v$ has at most one lift starting at $\til v$ and so $\p$ is an injective homomorphism.  Suppose that $p,pq\in \p(\til \Gamma^*(\til v))$ with $q\in \Gamma^*(v)$.  Choose lifts $\til p,\til q$ of $p$ and $q$ respectively starting at $\til v$.  Then $\til p$ must be a loop at $\til v$ by uniqueness of lifts.  Also, since $\til p\til q$ lifts $pq$, it too must be a loop at $\til v$.  It follows now that $\til q$ is a loop at $\til v$ and so $q=\p(\til q)\in \p(\til \Gamma^*(\til v))$, as required.

Now suppose that $N$ is a right unitary submonoid of $\Gamma^*(v)$.   Let $X$ be the set of directed paths in $\Gamma$ starting at $v$.  Define an equivalence relation on $X$ by $p\equiv q$ if $\tau(p)=\tau(q)$ and $pu\in N$ if and only if $qu\in N$ for all $u\colon \tau(p)\to v$.  Let $V$ be the quotient of $X$ by $\equiv$ and let $E$ consist of all pairs $([p],e)$ with $e\in E(\Gamma)$ and $\tau(p)=\iota(e)$.  We define $\til \Gamma$ to have vertex set $V$ and edge set $E$ where $([p],e)$ goes from $[p]$ to $[pe]$. Notice that $[p]=[q]$ implies $peu\in N$ if and only if $qeu\in N$ and so the incidence functions are well defined.  The reader can easily verify that the maps $[p]\mapsto \tau(p)$ and $([p],e)\mapsto e$ yield a directed covering.  Surjectivity on vertices requires that $(\Gamma,v)$ is rooted.  Let $\til v = [1_v]$.   One can verify directly that if $q$ is a loop at $v$, then the lift of $q$ to $\til v$ ends at $[q]$.  Now $[q]=\til v$, if and only if $q\equiv 1_v$.  Let us show that this is equivalent to $q\in N$.  If $q\equiv 1_v$, then since $1_v\in N$, it follows that $q=q1_v\in N$.  Conversely, if $q\in N$, then the very definition of right unitary implies $qu\in N$ if and only if $u\in N$, if and only if $1_vu\in N$.  Thus $q\equiv 1_v$, completing the proof that $\til \Gamma^*(\til v)$ maps onto $N$.
\end{proof}

There can be multiple directed coverings corresponding to a given right unitary submonoid of $\Gamma^*(v)$, but the construction given in Proposition~\ref{coveringtheory} is the unique minimal one in the sense that all others cover it.  The above proof also shows that directed immersions induce injective maps on fundamental monoids.

Notice that if $\p\colon (\til \Gamma,\til v)\to (\Gamma,v)$ is a directed covering of rooted graphs, then $\Gamma^*(v)$ has a monodromy action on $\p\inv(v)$ given on $w\in \pinv(v)$ by $wp =w'$ where $w'$ is the end point of the unique lift of $p$ starting at $w$.  This generalizes the action of the transition monoid of a deterministic automaton.

\subsection{The unique simple path property}

Perhaps the most important notion in geometric theory is that of a rooted graph with the unique simple path property.  We give several equivalent properties that will define this notion.  Recall that path unmodified means directed path.

\begin{Prop}\label{defineuspp}
Let $(\Gamma,v)$ be a rooted graph.  Then the following are equivalent:
\begin{enumerate}
\item For each vertex $w$, there is a unique simple path from $v$ to $w$;
\item $(\Gamma,v)$ admits a unique directed spanning tree;
\item $(\Gamma,v)$ admits a directed spanning tree $(T,v)$ such that, for each edge $e\in E(\Gamma)\setminus E(T)$ one has $[v,\tau(e)]$ is an initial segment of $[v,\iota(e)]$ (i.e., $\tau(e)$ is visited by the geodesic $[v,\iota(e)]$).
\end{enumerate}
\end{Prop}
\begin{proof}
To see that (1) implies (2), suppose that $(T,v)$ and $(T',v)$ are distinct spanning trees.  Then there is an edge $e$ that belongs to, say, $T$ and not $T'$.  Set $w=\tau(e)$. Let $p\colon v\to \iota(e)$ and $q\colon v\to w$ be the geodesics in $T$ and $T'$ respectively.  Then $pe\colon v\to w$ is a directed path in $T$ and hence simple by Proposition~\ref{rootedtrees}.  Since $e\notin T'$, clearly $q\neq pe$.  This contradicts (1).

For (2) implies (3), let $T$ be the unique directed spanning tree for $\Gamma$ and suppose that $e\in E(\Gamma)\setminus E(T)$.  Suppose that $\tau(e)\notin [v,\iota(e)]$. The path $[v,\iota(e)]e$ is then simple and so its support $T'$ is a directed tree rooted from $v$.  By Remark~\ref{treecompletion} we may complete $T'$ to a directed spanning tree that evidently is different from $T$.  This contradicts the uniqueness of $T$.

Finally, we prove (3) implies (1).  Let $T$ be the spanning tree provided by (3). Let $p=[v,w]$ be the geodesic in $T$. Then $p$ is a simple path from $v$ to $w$.  Suppose that $q\colon v\to w$ is another simple path; it cannot be contained in $T$ so let $e$ be the first edge used by $q$ that does not belong to $T$.  Then we can factor $q=ret$ where $e\in E(\Gamma)\setminus T$ and $r=[v,\iota(e)]$.  But then by assumption $\tau(e)\in [v,\iota(e)]$.  This contradicts that $q$ is simple.  We conclude that $p$ is the unique simple path from $v$ to $w$.
\end{proof}

The third condition says that $\Gamma$ is a ``rooted tree falling back on itself.''
Here is the key definition in geometric semigroup theory.

\begin{Def}[Unique simple path property]
A rooted graph $(\Gamma,v)$ is said to have the \emph{unique simple path property} if it satisfies the equivalent conditions of Proposition~\ref{defineuspp}.  If $(\Gamma,v)$ has the unique simple path property, then we denote the unique simple path from $v$ to $w$ by $[v,w]$ and call it the \emph{geodesic}, just as we did for directed rooted trees.
\end{Def}

\begin{Rmk}
Having the unique simple path property depends not just on the graph $\Gamma$, but also on the choice of the root.  For example the graph in Figure~\ref{nosimplepathfromall}
\begin{figure}
\[\SelectTips{cm}{}
\xymatrix{v\ar[rr] &&\ar@/^2pc/[ll] \ar@/_2pc/[ll]w}
\]
\caption{A graph with the unique simple path property at $v$ but not $w$.\label{nosimplepathfromall}}
\end{figure}
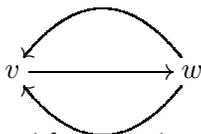
has the unique simple path property from $v$ but not from $w$.
\end{Rmk}

\begin{Rmk}[Non-planar]\label{rem:non-planar}
Also, having the unique simple path property does not force the graph
to be planar.  For example, the complete bipartite graph $K_{3,3}$ is
non-planar but as shown in Figure~\ref{fig:k33} it can be oriented to
have the unique simple path property from one of its vertices ($p_1$
in this case).
\end{Rmk}

\psfrag{r1}{$p_1$}
\psfrag{r2}{$p_2$}
\psfrag{r3}{$p_3$}
\psfrag{s1}{$q_1$}
\psfrag{s2}{$q_2$}
\psfrag{s3}{$q_3$}
\begin{figure}[ht]
\includegraphics{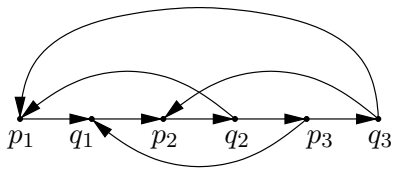}
\caption{A non-planar graph with the unique simple path
property.\label{fig:k33}}
\end{figure}

\begin{Rmk}
 Notice that if $(\Gamma,v)$ has the unique simple path property, then each transition edge of $\Gamma$ belongs to the directed spanning tree by the third item of Proposition~\ref{defineuspp}.
\end{Rmk}

It is easy to see that if $(\Gamma,v)$ has the unique simple path property and $X$ is a convex set of vertices of $\Gamma$ containing $v$, then $(\Delta(X),v)$ has the unique simple path property where $\Delta(X)$ is the subgraph induced by $X$.

The frame of a rooted graph with the unique simple path property is a directed rooted tree.  Moreover, each strong component has a unique transition edge entering it, called its \emph{entry edge}.  The endpoint of the entry edge is termed the \emph{entrance} of the strong component.  Most importantly, if $C$ is a strong component with entrance $w$, then $(C,w)$ has the unique simple path property as does $(C^{\downarrow},w)$ where $C^{\downarrow}$ is the subgraph of $\Gamma$ induced by the downset in $\prec$ generated by the vertices of $C$ (or equivalently by $w$).  This is summarized in the following proposition.

\begin{Prop}\label{entrances}
Let $(\Gamma,v)$ have the unique simple path property.  Then $(\Fr \Gamma,C_v)$ is a directed rooted tree where $C_w$ denotes the strong component of a vertex $w$.   Each strong component has a unique transition edge entering it, called its entry edge.  If $C$ is a strong component with entry edge $e$ and entrance $w=\tau(e)$, then $(C^{\downarrow},w)$ has the unique simple path property and hence $(C,w)$ has the unique simple path property (being convex in $(C^{\downarrow},w)$).
\end{Prop}
\begin{proof}
We show that there is a unique directed path from $C_v$ to $C$ in $\Fr \Gamma$ for any strong component $C$.  Proposition~\ref{rootedtrees} then implies that $(\Fr \Gamma,C_v)$ is a directed rooted tree.   It is immediate from the definition of $\Fr \Gamma$ that there is a path $C_x\to C_y$ if and only if there is a path from $x$ to $y$.  Thus there is some path from $C_v$ to $C$.  Suppose that $p=e_1\cdots e_n$ and $q=f_1\cdots f_m$ are two distinct such paths; they are necessarily simple since $\Fr \Gamma$ is acyclic.  The $e_i$ and $f_j$ are transition edges of $\Gamma$.  Since $e_n$ and $f_m$ end in $C$, we can find a simple path in $r\colon \tau(e_n)\to \tau (f_m)$ in $C$.  Also, for each $1\leq i\leq n-1$ we can find a simple path $p_i$ contained in the strong component of $\tau(e_i)$ from $\tau(e_i)$ to $\iota(e_{i+1})$.  Let $t\colon v\to \iota(e_1)$ be a simple path contained in $C_v$.  Then $p'=te_1p_2\cdots e_{n-1}p_{n-1}e_nr$ is a simple path from $v$ to $\tau(f_m)$ whose transition edges are precisely $e_1,\ldots,e_n$.  Similarly, we can construct a simple path $q'=t'f_1q_1\cdots f_{m-1}q_{m-1}f_m$ from $v$ to $\tau(f_m)$ whose transition edges are precisely $f_1,\ldots, f_m$.  Since $p'\neq q'$, this contradicts the unique simple path property.  We conclude $(\Fr \Gamma,C_v)$ is a directed rooted tree.

Suppose that $e,e'$ are distinct transition edges entering $C$.  Say $e\colon C_1\to C$ and $e'\colon C_2\to C$ in $\Fr \Gamma$.  Let $p,q$ be the geodesics in $\Fr \Gamma$ from $C_v$ to $C_1,C_2$ respectively.  Then $pe,qe'\colon C_v\to C$ are distinct directed paths, a contradiction to $\Fr \Gamma$ being a rooted tree.

Finally, let $C$ be a strong component with entry edge $e$ and entrance $w=\tau(e)$.  Let $z$ be any vertex with $z\prec w$.  Suppose $p,q\colon w\to z$ are simple paths.   Since $e$ is a transition edge, $p$ and $q$ cannot use the edge $e$. It follows that $[v,\iota(e)]ep$ and $[v,\iota(e)]eq$ are simple paths from $v$ to $z$.  We conclude $p=q$, as required.
\end{proof}

The following corollary is immediate.

\begin{Cor}\label{uniquepathquasi}
If $(\Gamma,v)$ has the unique simple path property and is quasilinear, then $\Gamma$ is linear.
\end{Cor}

Let us fix for the moment a rooted graph $(\Gamma,I)$ with the unique simple path property and denote by $T$ its unique directed spanning tree.  Let us use $\leq_T$ to denote the accessibility order in $T$.  Notice that $u\leq_T v$ implies $u\leq v$, but not conversely.   For example, $\leq$ is trivial on strongly connected components of $\Gamma$, but $\leq_T$ is still non-trivial.  If $v$ is a vertex of $\Gamma$, denote by $v^{\Downarrow}$ the set of vertices $u\in V(\Gamma)$ with $u\leq_T v$.  Abusively, we shall also denote the induced subgraph with vertex set $v^{\Downarrow}$ by the same notation.  Let us denote by $C_{v^{\Downarrow}}$ the strong component of $v$ in $v^{\Downarrow}$.

\begin{Prop}\label{treedown}
Let $(\Gamma,I)$ be a graph with the unique simple path property and suppose $v\in V(\Gamma)$.  The rooted graph $(v^{\Downarrow},v)$ has the unique simple path property.  Consequently, $(C_{v^{\Downarrow}},v)$ has the unique simple path property.
\end{Prop}
\begin{proof}
If $u\in v^{\Downarrow}$, then there is a simple path in $T$ from $v$ to $u$ by definition, which we denote $[v,u]$.  Then $[I,u]=[I,v][v,u]$.   Suppose that $p\colon v\to u$ is a simple path in $v^{\Downarrow}$ and consider $[I,v]p$.  If this path is simple, $[v,u]=p$ and we are done.  If not, then there is a vertex $r$ visited twice by $[I,v]p$.  Since $[I,v]$ and $p$ are simple, it follows that $r\neq v$ and is visited by both $[I,v]$ and $p$.  But this is impossible since this means that $v<_T r$ and $r\leq_T v$.
\end{proof}

Sometimes one wants to extract a quasilinear subgraph from a graph with the unique simple path property. If $w$ is a vertex of $\Gamma$, then by abuse of notation, $w^{\uparrow}$ will also denote the full subgraph of $\Gamma$ with vertex set $w^{\uparrow} = \{q\in V(\Gamma)\mid q\succ w\}$.
\begin{Prop}\label{straightlineinuspp}
Suppose that $(\Gamma,v)$ has the unique simple path property and $w\in V(\Gamma)$.   Then $(w^\uparrow,v)$ has the unique simple path property and is quasilinear.  In particular, $w^\uparrow$ is linear.
\end{Prop}
\begin{proof}
Since $w^\uparrow$ is convex, it has the unique simple path property from $v$.  To see that it is quasilinear, let $C,C'$ be strong components of $w^\uparrow$.  Then since $\Fr \Gamma$ is a tree and $C,C'$ are between the strong component of $v$ and the strong component of $w$, it follows that $C$ and $C'$ are comparable in the accessibility order.  Thus $w^\uparrow$ is quasilinear and hence is linear by Corollary~\ref{uniquepathquasi}.
\end{proof}

We now define the ``bold'' arrows of a rooted graph with the unique simple path property.

\begin{Def}[Bold arrows]
Let $(\Gamma,I)$ have the unique simple path property with directed spanning tree $T$.  Then $\pv E(\Gamma) = E(\Gamma)\setminus E(T)$ is called the set of \emph{bold arrows} of $\Gamma$. (Notice the bold font is used to denote the set of bold arrows.)
\end{Def}

The set $\pv E(\Gamma)$ of bold arrows is in bijection with a generating set for $\pi_1(\Gamma,I)$ (or equivalent $H_1(\Gamma)$) and so the number of bold arrows is just the first Betti number of $\Gamma$.  Notice that if $\Gamma'$ is any subgraph of $\Gamma$ containing the spanning tree $T$, then $(\Gamma',I)$ still has the unique simple path property.  Thus removing bold arrows does not cause one to lose the unique simple path property.  Induction on the number of bold arrows is a key idea in geometric semigroup theory.

An important notion in geometric semigroup theory is that of a sloop.

\begin{Def}[Sloop]\label{definesloop}
If $e\in \pv E(\Gamma)$ is a bold arrows, then the sloop (think ``simple loop'') is the path $\slp(e)=[I,\iota(e)]e$.  In pictures, we have
\[\SelectTips{cm}{}
\xymatrix{ & &\iota(e)\ar@/_1pc/@2[dl]_e &\\ I\ar[r]_{[I,\tau(e)]}&\tau(e)\ar `^u[rr]_{[\tau(e),\iota(e)]} `^l[rru] [ru] &&}\]
where the bold arrow is drawn doubled.
\end{Def}

If $e$ is a bold arrow, then the corresponding generator of $\pi_1(\Gamma,I)$ is $\slp(e)[I,\tau(e)]\inv$.  Next we define the loop of a sloop.

\begin{Def}[Loop of a sloop]
If $e\in \pv E(\Gamma)$ is a bold arrow, define the loop $\lp(e)$ of the sloop $\slp(e)$ to be the simple circuit $[\tau(e),\iota(e)]e$.  So the loop $\lp(e)$ of the sloop $\slp(e)$ in the picture from Definition~\ref{definesloop} is
\[\SelectTips{cm}{}
\xymatrix{  &\iota(e)\ar@/_1pc/@2[dl]_e &\\ \tau(e)\ar `^u[rr]_{[\tau(e),\iota(e)]} `^l[rru] [ru] &&}\]
Note that $\slp(e)=[I,\tau(e)]\lp(e)$.
\end{Def}

In a graph with the unique simple path property, each non-trivial strong component is a union of loops of sloops.

\begin{Prop}\label{describestronglyconnected}
Let $(\Gamma,v)$ have the unique simple path property and let $C$ be a strong component of $\Gamma$ containing at least one edge.  Then $C=\bigcup_{e\in \pv E(\Gamma)\cap C} \lp(e)$.
\end{Prop}
\begin{proof}
Obviously, $\lp(e)\subseteq C$ for any $e\in \pv E(\Gamma)\cap C$ since $\lp(e)$ is strongly connected. For the converse, it suffices to show that each edge of $C$ belongs to $\lp(e)$ for some $e\in \pv E(\Gamma)$ since every vertex of $C$ must be on some edge.  Let $u$ be the entrance of $C$ and suppose $f\in E(C)$. If $f\in \pv E(\Gamma)$, there is nothing to prove.  So suppose that $f\in E(T)$; say $f\colon v\to w$.   Choose a simple path $p\colon w\to u$.  Then $q=[u,v]fp$ is not simple and so we can factor $q=[u,\iota(e)]eq'$ where $e$ is the first bold arrow used by $q$. Evidentally, $[u,v]f$ is an initial segment of $[u,\iota(e)]$ and so $\iota(e)\leq_T w$.  On the other hand, since $p$ is simple, we must have $v\leq_T \tau(e)$ and so $\iota(e)\leq_T w\leq_T v\leq _T\tau(e)$.  See Figure~\ref{fig:describestronglyconnected}.
\begin{figure}
\[\SelectTips{cm}{}
\xymatrix{&& & \iota(e)\ar@/_1pc/[ld]_e& &\\ u\ar@/_2pc/[rr]_{[u,\tau(e)]}&&\tau(e)\ar@/_2pc/[ll]_{q'}\ar@/_1pc/[rd]_{[\tau(e),v]}&&w\ar@/_1pc/[lu]_{[w,\iota(e)]}\\ & & & v\ar@/_1pc/[ru]_f & &}\]
\caption{\label{fig:describestronglyconnected}}
\end{figure}
Thus $f$ is an edge of $\lp(e) = [\tau(e),\iota(e)]e$.
\end{proof}

Next we wish to define the notion of a geometric rank function.  Later on, when we deal with automata, we shall impose some algebraic restrictions on our rank functions, but here we define things in complete generality.  From now on we assume that all graphs have finite out-degree.

\begin{Def}[Geometric rank function]
Let $(\Gamma,I)$ have the unique simple path property.  Then a \emph{geometric rank function} is a mapping $r\colon \pv E(\Gamma)\to \NN$ so that, for each vertex $v$, $r$ maps $\tau^{-1}(v)\cap \pv E(\Gamma)$ bijectively to the interval $[0,|\tau^{-1}(v)\cap \pv E(\Gamma)|-1]$.  We then totally order  $\tau^{-1}(v)\cap \pv E(\Gamma)$ by putting $e\leq_r f$ if and only if $r(e)\leq r(f)$.
\end{Def}

Said differently, a geometric rank function is a way of totally ordering the bold arrows at each vertex.  Given a geometric rank function, we can now define a partial order on the bold arrows as follows.  Recall that $C_v$ denotes the strong component of a vertex $v$.

\begin{Def}[Order on bold arrows]
Fix a geometric rank function $r$.  Then define a partial order on $\pv E(\Gamma)$ by putting $e\ll f$ for $e\neq f\in \pv E(\Gamma)$ if:
\begin{enumerate}
\item $C_{\tau(f)}<C_{\tau(e)}$; or
\item $C_{\tau(e)}=C_{\tau(f)}$ and $\tau(e)<_T \tau(f)$; or
\item $\tau(e)=\tau(f)$ and $r(e)< r(f)$.
\end{enumerate}
\end{Def}

\begin{Example}
Consider the graph
\begin{center}
\[\SelectTips{cm}{}
\xymatrix{        &                               &           & \bullet\ar@/_1pc/@2[dl]_{e_1}\ar[r] &\bullet\ar@/_1pc/[r]&\bullet\ar@/_1pc/@2[l]_{e_3}\ar@(rd,ru)@2_{e_4}\\
            I\ar[r] & \bullet\ar@/_1pc/[rrd]\ar[r]  & \bullet\ar@/_1pc/[ru]   &               &\bullet\ar@/_1pc/@2[ld]_{e_2}&\\
                    &                               &  &\bullet\ar@/_1pc/[ru]  &&}\]
\end{center}
The bold arrows are $e_1,e_2,e_3, e_4$.  One has $e_1\ll e_4\ll e_3$, whereas $e_2$ is incomparable to the other arrows.
\end{Example}

\subsection{Cutting sloops}
A major idea in geometric semigroup is to work by induction on the number of bold arrows.  To do this, we need to cut sloops.  First we talk about cutting arbitrary rooted graphs.

\begin{Def}[Cutting a graph]
Let $(\Delta,v)$ be a rooted graph.  Define $\cut(\Delta,v)$ to be the graph with vertex set $V(\Delta)\cup (\{v\}\times \tau^{-1}(v))$ and edge set $E(\Delta)$.  The initial vertex function $\iota$ is as before.  The new terminal vertex function $\tau'$ is given by \[\tau'(e) = \begin{cases} \tau(e) & \tau(e)\neq v \\ (v,e) &\tau(e)=v.\end{cases}\]
\end{Def}

Geometrically, $\cut(\Delta,v)$ is obtained from $\Delta$ by taking each edge ending at $v$ and moving its end away from $v$.  Notice the map $\cut(\Delta,v)\to \Delta$ which is the inclusion on $V(\Delta)\cup E(\Delta)$ and which sends $\{v\}\times \tau\inv (v)$ to $v$ is a directed graph immersion.

\begin{Example}
Suppose $(\Delta,v)$ is the rooted graph in Figure~\ref{fi:tocut}.
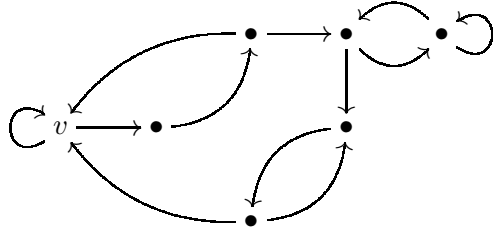
\begin{figure}[tbhp]
\begin{center}
\[\SelectTips{cm}{}
\xymatrix{                                       &           & \bullet\ar@/_1pc/[dll]\ar[r] &\bullet\ar[d]\ar@/_1pc/[r]&\bullet\ar@/_1pc/[l]\ar@(rd,ru)\\
            v\ar[r]\ar@(ld,lu)  & \bullet\ar@/_1pc/[ru]   &               &\bullet\ar@/_1pc/[ld]&\\
                               &  &\bullet\ar@/_1pc/[ru]\ar@/^1pc/[llu]  &&}\]
\end{center}
\caption{The graph $(\Delta,v)$ be\label{fi:tocut}}
\end{figure}
Then $\cut(\Delta,v)$ is the graph in Figure~\ref{fig:cut}.
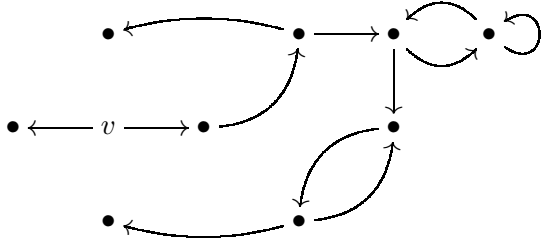
\begin{figure}[bthp]
\begin{center}
\[\SelectTips{cm}{}
\xymatrix{            &              \bullet            &           & \bullet\ar@/_/[ll]\ar[r] &\bullet\ar[d]\ar@/_1pc/[r]&\bullet\ar@/_1pc/[l]\ar@(rd,ru)\\
        \bullet &    v\ar[r]\ar[l]  & \bullet\ar@/_1pc/[ru]   &               &\bullet\ar@/_1pc/[ld]&\\ &
                        \bullet       &  &\bullet\ar@/_1pc/[ru]\ar@/^/[ll]  &&}\]
\end{center}
\caption{The graph $\cut(\Delta,v)$\label{fig:cut}}
\end{figure}
\end{Example}

The next proposition establishes some basic properties of cut graphs.

\begin{Prop}\label{cutkeepsuspp}
The graph $\cut(\Delta,v)$ is rooted at $v$.  Moreover,  the rooted graph $(\cut(\Delta,v),v)$ has the unique simple path property if and only if $(\Delta,v)$ does. If $(\Delta,v)$ has the unique simple path property, then the bold arrows of $(\cut(\Delta,v),v)$ are the bold arrows of $\Delta$ that do not end at $v$.
\end{Prop}
\begin{proof}
We shall repeatedly use the following fundamental observation:
if $w\in V(\Delta)$ with $w\neq v$, then the simple paths between $v$ and $w$ in $\Delta$ and $\cut(\Delta,v)$ are one in the same.  This is because on the one hand, a simple path from $v$ never uses an edge of $\tau\inv (v)$; on the other hand, a path in $\cut(\Delta,v)$ using an edge $e\in \tau\inv(v)$ can only use that edge as its last edge and therefore does not end in $V(\Delta)$.

It follows that if $w\in V(\Delta)$, then there is a simple path from $v$ to $w$ in $\cut(\Delta,v)$.
On the other hand, if $e\in \tau\inv(v)$ and $p$ is a simple path from $v$ to $\iota(e)$ in $\Delta$, then the fundamental observation implies $pe$ is a simple path in $\cut(\Delta,v)$ from $v$ to $(v,e)$.

From the fundamental observation, it is immediate that if $(\cut(\Delta,v),v)$ has the unique simple path property, then so does $(\Delta,v)$.  Suppose now that $(\Delta,v)$ has the unique simple path property.  If $w\in V(\Delta)$, then the fundamental observation shows that there is a unique simple path form $v$ to $w$ in $\cut(\Delta,v)$.  Now any simple path from $v$ to $(v,e)$ for $e\in \tau\inv(v)$ must end in the edge $e$ since it is the only edge ending at $(v,e)$.  Thus the unique simple path in $\cut(\Delta,v)$ from $v$ to $(v,e)$ is $pe$ where $p$ is the unique simple path in $\Delta$ from $v$ to $\iota(e)$.  It follows from this discussion that the bold arrows of $(\cut(\Delta,v),v)$ are those bold arrows of $\Delta$ that do not end at $v$.
\end{proof}

Next we need to define a notion of a closed subgraph with respect to a base point.

\begin{Def}[Closed subgraph with respect to a point]\label{closedsub}
Let $(\Gamma,I)$ have the unique simple path property and let $T$ be its unique directed spanning tree.  Fix $u\in V(\Gamma)$.  A subgraph $\Delta$ of $\Gamma$ is said to be \emph{closed with respect to $u$}, or the pair $(\Delta,u)$ is said to be \emph{closed} if:
\begin{enumerate}
\item $u\in\Delta\subseteq u^{\Downarrow}$;
\item For all $v\in V(\Delta)$, one has $[u,v]\subseteq \Delta$;
\item If $e\in \mathbf E(\Gamma)$ and $u\neq \tau(e)\in \Delta$, then $e\in \Delta$.
\end{enumerate}
\end{Def}

Let us show that the set of all closed subgraphs with respect to $u$ is a complete lattice.  Throughout $(\Gamma,I)$ is a fixed graph with the unique simple path property.

\begin{Prop}\label{closediscompletelattice}
The collection $L_u$ of all closed subgraphs of $\Gamma$ with respect to $u$ has maximum element $u^{\Downarrow}$  and is closed under non-empty intersections.  Thus it is a complete lattice.  The bottom of $L_u$ is $\{u\}$ and its join is determined by the meet.
\end{Prop}
\begin{proof}
First note that the graph $u^{\Downarrow}$ is closed with respect to $u$.  The first two axioms are obvious.  The last one follows since in a graph with the unique simple path property $\iota(e)\leq_T \tau(e)$ for any bold arrow $e$.  Clearly it is the largest element of $L_u$ by (1).  It is trivial that the set of elements satisfying (1)--(3) is closed under non-empty intersections.
\end{proof}

As a consequence we can define the closure of a subgraph.

\begin{Def}[Closure of a pointed subgraph]
If $\Delta\subseteq u^{\Downarrow}$, define $\ov{(\Delta,u)}=(\ov \Delta,u)$ where $\ov \Delta$ is the meet of all elements of $L_u$ containing $\Delta$.  This definition makes sense since the top of $L_u$ is $u^{\Downarrow}$.
\end{Def}

Let us establish some basic properties of $\ov{(\Delta,u)}$.

\begin{Prop}\label{closedhasuspp}
Let $(\Delta,u)$ be closed.  Then $(\Delta,u)$ has the unique simple path property.
\end{Prop}
\begin{proof}
Proposition~\ref{treedown} shows that $(u^{\Downarrow},u)$ has the unique simple path property.  Since $u\in \Delta\subseteq u^{\Downarrow}$, it suffices to show that $(\Delta,u)$ is rooted at $u$.  But this is immediate from Definition~\ref{closedsub}(2).
\end{proof}

Consequently, $\ov{(\Delta,u)}$ has the unique simple path property for any $\Delta\subseteq u^{\Downarrow}$.

\begin{Prop}\label{loopin}
Suppose that $(\Delta,u)$ is closed and $e\in \pv E(\Gamma)$ with $\tau(e)\in \Delta$ and $u\neq \tau(e)$.  Then $\lp(e)\in \Delta$.
\end{Prop}
\begin{proof}
It follows from the third axiom in the definition of closed that $e\in \Delta$.  Therefore, $[u,\iota(e)]\subseteq \Delta$ by the second axiom and hence $\lp(e) = [\tau(e),\iota(e)]e\subseteq \Delta$ as $\iota(e)\leq_T \tau(e)\leq_T u$.
\end{proof}

Next we show that taking the closure preserves strong connectivity.

\begin{Prop}\label{closurepreservessc}
Suppose that $\Delta\subseteq u^{\Downarrow}$ is strongly connected and $u\in \Delta$.  Then $\ov{(\Delta,u)}$ is strongly connected.
\end{Prop}
\begin{proof}
Let $C$ be the strong component of $u$ in $\ov{(\Delta,u)}$.  Then $\Delta\subseteq C$ since $u\in \Delta$ and $\Delta$ is strongly connected.  Thus it suffices to show that $(C,u)$ is closed.  The first axiom is clear.  If $v\in C$, choose a path $p\colon v\to u$ in $C$. Then $[u,v]\subseteq \ov{(\Delta,u)}$  and so the existence of the circuit $[u,v]p$ shows that $[u,v]\subseteq C$.  Suppose that $e\in \pv E(\Gamma)$ with $\tau(e)\in C$ and $u\neq\tau(e)$.  Then $e\in \ov{(\Delta,u)}$.  Therefore, $\lp(e)\subseteq \ov{(\Delta,u)}$ by Proposition~\ref{loopin} and hence $\lp(e)\subseteq C$.  We conclude $e\in C$.  This completes the proof that $(C,u)$ is closed.  It now follows that $\ov{\Delta}=C$.
\end{proof}

In order to show that certain automata are trim, we need to show that closure preserves the property that some vertex is reachable from all vertices.

\begin{Prop}\label{closurepreservesfilters}
Suppose that $(\Delta,u)\subseteq u^{\Downarrow}$ and $w$ is a vertex of $\Delta$ which is accessible from every vertex of $\Delta$ by a path in $\Delta$.  Then there is a directed path in $\ov \Delta$ from every vertex of $\ov \Delta$ to $w$.
\end{Prop}
\begin{proof}
Let $\Lambda$ be the full subgraph of $\ov{\Delta}$ containing all vertices $v$ of $\ov{\Delta}$ such that there is a path from $v$ to $w$ in $\ov{\Delta}$.  We claim that $(\Lambda,u)$ is a closed subgraph containing $\Delta$.  It will then follow that $\Lambda=\ov \Delta$, as required.  Indeed, by hypothesis $u\in \Delta\subseteq \Lambda\subseteq \ov \Delta\subseteq u^{\Downarrow}$.  In particular, the first condition in the definition of a closed subgraph is satisfied.  If $v\in \Lambda$, then there is a directed path from $v$ to $w$ in $\ov{\Delta}$.  Since $[u,v]\subseteq \ov{\Delta}$, it follows immediately that $[u,v]\subseteq \Lambda$.  Finally if $e\in \mathbf E(\Gamma)$ and $u\neq \tau(e)\in \Lambda$, then $e\in \ov{\Delta}$ since $\ov{\Delta}$ is closed.  Hence if $p$ is a path in $\ov \Delta$ from $\tau(e)$ to $w$, then $ep$ is a directed path in $\ov{\Delta}$ from $\iota(e)$ to $w$.  Thus $\iota(e)\in \Lambda$ and so $e\in \Lambda$.  This concludes the proof that $(\Lambda,u)$ is closed.
\end{proof}

We can now define the cut sloop of a sloop.

\begin{Def}[Cut sloop]
If $e\in \pv E(\Gamma)$, then we define the \emph{cut sloop} of $e$, denoted $\cut(e)$, to be $\cut(\ov{(\lp(e),\tau(e))})$.
\end{Def}

This definition makes sense because $\lp(e)$ is a subgraph of $\tau(e)^{\Downarrow}$.  Putting together Propositions~\ref{cutkeepsuspp} and~\ref{closedhasuspp}, we conclude that $(\cut(e),\tau(e))$ has the unique simple path property.

\begin{Prop}\label{cuthasuspp}
Let $e\in \pv E(\Gamma)$.  Then $(\cut(e),\tau(e))$ has the unique simple path property.
\end{Prop}

It turns out that $\cut(e)$ can be described more explicitly.

\begin{Prop}\label{allboldsarebigger}
Let $e\in \pv E(\Gamma)$ and set $u=\tau(e)$.  Let $K_u$ be the strong component of $u$ in the graph $\Gamma'$ obtained from $u^{\Downarrow}$ by removing the bold arrows ending at $u$ other than $e$ (so $K_u\subseteq C_{u^\Downarrow}$).  Then:
\begin{enumerate}
\item $\ov{(\lp(e),u)}\subseteq K_u$;
\item $e$ is the unique edge of $\ov{(\lp(e),u)}$ ending at $u$;
\item If $f\in \pv E(\Gamma)$ and $f$ belongs to $\ov{(\lp(e),u)}$, then $f\ll e$;
\item If $f\in \pv E(\Gamma)$ and $\tau(f)\in \ov{(\lp(e),u)}\setminus \{u\}$, then $\lp(f)\subseteq \ov{(\lp(e),u)}$;
\item Suppose that $f\in \pv E(\Gamma)\cap K_u$. 
Then $f\in \ov{(\lp(e),u)}$;
\item $K_u=\ov{(\lp(e),u)}$.
\end{enumerate}
\end{Prop}
\begin{proof}
To prove (1), clearly $\lp(e)\subseteq K_u$.  It therefore suffices to show that $(K_u,u)$ is closed.  The first two axioms are clear.  Suppose that $f\in \pv E(\Gamma)$ and $u\neq \tau(f)\in K_u$.  Then $\iota(f)\leq_T \tau(f)<_T u$ where the first inequality holds since $(\Gamma,v)$ has the unique simple path property, whereas the second follows by definition of $u^{\Downarrow}$.  Thus $f\in \Gamma'$.  Also $f\in K_u$ since bold arrows are never transition edges in $\Gamma'$ and $\tau(f)\in K_u$.  This completes the proof of (1).

We immediately deduce (2) from (1) since $\Gamma'$ contains only the edge $e$ coming into $u$ as any edge of $u^{\Downarrow}$ going into $u$ must necessarily be a bold arrow.

By (1) and (2), necessary conditions for a bold arrow $f$ to belong to $\ov{(\lp(e),u)}$ are that $C_{\tau(f)}=C_u$ and $\tau(f)\neq u$.  Also by definition of a closed set, we must have that $\tau(f)<_T u$.  But this implies $f\ll e$.

Statement (4) is just a restatement of Proposition~\ref{loopin} in this special case.

For (5), if $f=e$, there is nothing to prove. So suppose $f\neq e$.  By definition of $K_u$ we must then have $\tau(f)<_T u$.  We prove that $f\in \ov{(\lp(e),u)}$ by induction on the number of vertices of the form $\tau(f')$ with $f'\in \pv E(\Gamma)$ and $\tau(f)<_T \tau(f')<_T u$.  Suppose first that there is no vertex $w$ with $\tau(f)<_T w<_T u$ that is the tip of a bold arrow. We claim that $\tau(f)\in \lp(e)$.  Indeed, choose a simple path $p$ from $\tau(f)$ to $\iota(e)$ in $\Gamma'$ (we can do this since $f\in K_u$).  Consider the path $[u,\tau(f)]p$.  If this path is contained in $T$, then it equals $[u,\iota(e)]$ and so $\tau(f)\in \lp(e)$.  Otherwise, there is a first bold arrow $f'$ used by $p$.  Since $p$ is simple, we must have $\tau(f)<_T\tau(f')\leq_T u$.  See Figure~\ref{Fig:allboldsarebigger}.
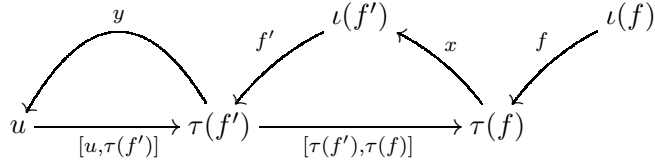
\begin{figure}
\[\SelectTips{cm}{}
\xymatrix{&&&\iota(f')\ar@/_/[ld]_{f'}  &&\iota(f)\ar@/_/[ld]_{f}\\
          u\ar[rr]_{[u,\tau(f')]} && \tau(f')\ar@/_3pc/[ll]_y\ar[rr]_{[\tau(f'),\tau(f)]}&& \tau(f)\ar@/_/[ul]_x &}\]
\caption{Factoring $p=xf'y$\label{Fig:allboldsarebigger}}
\end{figure}
Now $u\neq \tau(f')$ by definition of $\Gamma'$.  This contradicts our assumption on $f$, so $\tau(f)\in \lp(e)$ and hence $f\in \ov{(\lp(e),u)}$ by definition of a closed subgraph with respect to $u$.

Now we proceed by induction.  Choose a simple path $p$ from $\tau(f)$ to $\iota(e)$ in $\Gamma'$.  If the path $[u,\tau(f)]p$ is contained in $T$, then it equals $[u,\iota(e)]$ and so $\tau(f)\in \lp(e)$ and so $f\in \ov{(\lp(e),u)}$ by definition of a closed subgraph with respect to $u$.  Otherwise, there is a first bold arrow $f'$ used by $p$.  Again, because $p$ is simple, we must have $\tau(f)<_T \tau(f')\leq_T u$; see Figure~\ref{Fig:allboldsarebigger}.   Also the definition of $\Gamma'$ implies $u\neq \tau(f')$.  Thus $\tau(f)<_T\tau(f')<_T u$ and so by induction $f'\in \ov{(\lp(e),u)}$.  Therefore $\lp(f')\subseteq \ov{(\lp(e),u)}$ by (4).  Suppose that $p=xf'y$.  Then $\lp(f')=[\tau(f'),\tau(f)]xf'$ and so $\tau(f)\in \lp(f')\subseteq \ov{(\lp(e),u)}$.  We conclude that $f\in \ov{(\lp(e),u)}$ by the definition of a closed subgraph with respect to $u$.

To establish the final statement, (1) implies $\ov{(\lp(e),u)}\subseteq K_u$.  On the other hand,
Proposition~\ref{describestronglyconnected} shows that $K_u=\bigcup_{f\in \pv E(\Gamma')\cap K_u}\lp(f)$.  But (4) and (5) imply that $\lp(f)\subseteq \ov{(\lp(e),u)}$ for all $f\in \pv E(\Gamma')\cap K_u$.  This completes the proof.
\end{proof}

It follows from the proposition that when forming $\cut(\ov{(\lp(e),u)})$ we adjoin a single new vertex $(u,e)$, which we normally denote simply by $u'$.

\begin{Prop}\label{cutislinear}
Every edge of $\cut(e)$ is on a path from $u$ to $u'$.  Consequently, $(\cut(e),u)$ is linear.
\end{Prop}
\begin{proof}
The previous proposition implies $\ov{(\lp(e),u)}$ is strongly connected and hence every edge $f$ of $\ov{(\lp(e),u)}$ is on a circuit $q$ starting at $u$.  Since $e$ is the only edge entering $u$, the last edge of $q$ must be $e$.  We may assume without loss of generality that $q$ uses the edge $e$ only once.  Then $q$ corresponds to a path in $\cut(\ov{(\lp(e),u)})$ from $u$ to $u'$ using $f$.

By Proposition~\ref{cuthasuspp} $(\cut(e),\tau(e))$ has the unique simple path property.  It now follows from the first statement and Proposition~\ref{straightlineinuspp} that $\cut(e)$ is linear.
\end{proof}

A key property of $(\cut(e),u)$ is that it has strictly fewer bold arrows than $(\Gamma,I)$ since $e$ no longer is a bold arrow.  This is important for inductive constructions.

If $u,v$ are vertices with $v\leq_T u$, then $[u,v]\subseteq u^{\Downarrow}$ and so we can define $\ov{([u,v],u)}$.  This will be used later.

\subsection{The McCammond cover of a graph}
In this subsection, we present a construction, due to the first author, which the second author terms the ``McCammond'' expansion in the context of semigroups.   In the setting of graphs, the term ``cover'' seems more appropriate and so we will adhere to this terminology in this section, whereas in the context of semigroups and automata we shall use the term ``expansion.''

If $(\Gamma,I)$ is a rooted graph, let $\Simple(\Gamma,I)$ be the set of simple paths of $\Gamma$ starting at $I$ (including the empty path).  If $p$ is a simple path in a graph and $v$ is a vertex visited by $p$, denote by $p[v]$ the unique initial segment of $p$ ending at $v$.  Suppose that $\p\colon (\Gamma',I')\to (\Gamma,I)$ is a directed covering of rooted graphs.  Then we obtain an injective map $\til \p\colon \Simple(\Gamma,I)\to \Simple(\Gamma',I')$ by $\til \p(p) = \til p$ where $\til p$ is the unique lift of $p$ starting at $I'$.  Indeed, the lift of a simple path must be simple.  It is then natural to focus our attention on directed covers $\p$ where $\til \p$ is bijection.

\begin{Prop}\label{bijectsimple}
Let $\p\colon (\Gamma',I')\to (\Gamma,I)$ be a directed cover.  Then $\til \p\colon \Simple(\Gamma,I)\to \Simple(\Gamma',I')$ is a bijection if and only if $\p$ takes simple paths at $I'$ to simple paths at $I$.
\end{Prop}
\begin{proof}
Suppose first that $\til \p$ is a bijection and let $q$ be a simple path at $I'$.  Then $q=\til \p(p)$ for some path $p\in \Simple(\Gamma,I)$.  But then $\p(q)=p$ is simple.  Conversely, suppose $\p$ takes simple paths at $I'$ to simple paths at $I$.  Let $q\in \Simple(\Gamma',I')$; then $\p(q)\in \Simple(\Gamma,I)$ and $q=\til \p(\p(q))$.  This completes the proof.
\end{proof}

\begin{Def}[Simple covering]
A morphism $\p\colon (\Gamma',I')\to (\Gamma,I)$ of rooted graphs will be called a \emph{simple covering} if it is a directed covering satisfying the equivalent conditions of Proposition~\ref{bijectsimple}.
\end{Def}

It turns out that each rooted graph $(\Gamma,I)$ has a universal simple cover, which is characterized by having the unique simple path property.  This cover is called the \emph{McCammond cover} of $(\Gamma,I)$.  First we establish its universal property (and hence uniqueness).  Then we provide its construction.

\begin{Prop}\label{universalsimplecover}
Let $\alpha\colon (\til \Gamma,\til I)\to (\Gamma,I)$ be a simple covering such that $(\til \Gamma,\til I)$ has the unique simple path property.  Then given any simple covering $\p\colon (\Gamma',I')\to (\Gamma,I)$, there is a unique morphism (necessarily a simple covering) $\psi \colon (\til \Gamma,\til I)\to (\Gamma',I')$ so that the diagram
\[\xymatrix{(\til \Gamma,\til I)\ar[rr]^{\psi}\ar[dr]_{\alpha} & & (\Gamma',I')\ar[dl]^{\p}\\ & (\Gamma,I) &}\]
commutes.
\end{Prop}
\begin{proof}
Define $\psi$ as follows.  On vertices, set $\psi(v) = \tau(\til \p(\alpha([\til I,v])))$; this makes sense since $\alpha$ is a simple covering.  For $e\in E(\til \Gamma)$ with $e\colon v\to w$, define $\psi(e)$ to be the unique lift $\til e$ of $\alpha(e)$ starting at $\psi(v)$.  The uniqueness and existence of such an edge follows since $\p$ is a directed cover and $\p(\tau(\til \p(\alpha([\til I,v]))) = \tau(\alpha ([\til I,v])) = \alpha(v)$.  We need to check that $\p$ is a graph morphism.  By construction $\psi(v) = \iota(\psi(e))$.  The issue is to prove $\psi(w)=\tau(\psi(e)) = \tau(\til e)$.

The proof divides into two cases.  Suppose first that $[\til I,v]e$ is a simple path (and hence is $[\til I,w]$).  Then $\psi(w) = \tau(\til \p(\alpha([\til I,w]))) = \tau(\til \p(\alpha([\til I,v])\til e)) = \tau(\psi(e))$ since $\til \p(\alpha([\til I,v])\til e$ lifts $\alpha([\til I,v]e) = \alpha([\til I,w])$.

Next suppose that $[\til I,v]e$ is not simple. Then $[\til I,w]$ is an initial segment of $[\til I,v]$.  It follows that $\til \p(\alpha([\til I,w]))$ is an initial segment of $\til \p(\alpha([\til I,v]))$ since path lifting preserves initial segments.
Next observe that $\til \p(\alpha[\til I,v])\til e$ is not simple as $\p$ is a simple covering and $\p(\til \p(\alpha[\til I,v])\til e) = \alpha([\til I,v]e)$, which is not simple.  Hence $\psi(w) = \tau(\til e)$ is a vertex $w'$ of $\til \p(\alpha([\til I,v]))$ with $\p(w')=\alpha(w)$.  But since $\til \p(\alpha([\til I,v]))$ is simple and $\p$ is a simple covering, there is exactly one vertex $w'$ of $\til \p(\alpha[\til I,v])$ lying over $\alpha(w)$.  Since $\til \p(\alpha([\til I,w]))$ is an initial segment of $\til \p(\alpha[\til I,v])$, this $w'$ must be $\tau(\til \p(\alpha([\til I,w])))=\psi(w)$.  This completes the proof that $\psi$ is a graph morphism.  Since the above diagram commutes, $\psi$ must automatically be a simple covering.

If $\beta$ is another morphism making the diagram commute, then it to must be a simple covering.  Thus $\beta([\til I,v])$ is a simple path from $I'$ to $\beta(v)$ and its image under $\p$ is the simple path $\alpha([\til I,v])$.  Thus $\til \p(\alpha([\til I,v])) = \beta([\til I,v])$ and so $\beta(v)=\psi(v)$.  If $e$ is an edge at $v$, then $\beta(e)$ must be the unique lift of $\alpha(e)$ at $\beta(v)=\psi(v)$.  This proves uniqueness.
\end{proof}

The existence of the McCammond cover (or universal simple cover) is established via the same schema as the construction of the universal covering space in algebraic topology.  Abusively, in what follows we shall also use $I$ to denote the empty path at $I$.

\begin{Def}[McCammond cover]
Let $(\Gamma,I)$ be a rooted graph. Define $(\Gamma,I)^{\Mac} = (\Gamma^{\Mac},I)$ to be the graph given by:
\begin{align*}
V(\Gamma^{\Mac}) &= \Simple(\Gamma,I);\\
E(\gamma^{\Mac}) &= \{(p,e)\mid p\in \Simple(\Gamma,I), e\in E(\Gamma), \tau(p)=\iota(e)\}
\end{align*}
where the incidence functions are defined by $\iota(p,e)=p$ and \[\tau(p,e) = \begin{cases} pe & pe\in \Simple(\Gamma,I)\\ p[\tau(e)] & \text{else.}\end{cases}\]
Define $\eta\colon (\Gamma^{\Mac},I)\to (\Gamma,I)$ by $\eta(p)=\tau(p)$ on vertices and $\eta(p,e)=e$ on edges.  The map $\eta$ is called the \emph{McCammond covering} of $\Gamma$.  The graph $\Gamma^{\Mac}$ is termed the \emph{McCammond expansion} or \emph{McCammond cover} of $\Gamma$.
\end{Def}

Since any directed cover of an $A$-automaton is an $A$-automaton, it follows that if $(\mathscr A,I)$ is a pointed $A$-automaton, then $(\mathscr A,I)^{\Mac}$ is a pointed $A$-automaton.

\begin{Prop}\label{Mcisuniversal}
The map $\eta\colon (\Gamma^{\Mac},I)\to (\Gamma,I)$ is a simple covering and $(\Gamma^{\Mac},I)$ has the unique simple path property.  Hence it is the universal simple covering of $(\Gamma,I)$.
\end{Prop}
\begin{proof}
Let us verify that the map $\eta$ is a graph morphism. First we check that $\iota(\eta(p,e)) = \iota(e) = \tau(p) = \eta(\iota(p,e))$.  Second, observe that $\tau(\tau(p,e))$ is a simple path ending at $\tau(e)$ in either of the two cases in the definition of $\tau$.  Thus $\eta(\tau(p,e)) = \tau(\tau(p,e)) = \tau(e) = \tau \eta(p,e)$.  Thus $\eta$ is a graph morphism.  It is a directed covering since if $e\in E(\Gamma)$ is an edge with  $\iota(e)=\eta(p) = \tau(p)$, then $(p,e)$ is the unique edge at $p$ mapping to $e$ under $\eta$.

The unique lift of the empty path at $I$ to $\Gamma^{\Mac}$ is the empty path at $I$.  If $p\in \Simple(\Gamma,I)$ is non-empty, say $p=e_1\cdots e_n$, then \[\til \eta(p) = (I,e_1)(e_1,e_1e_2)\cdots (e_1\cdots e_{n-1},e_n),\] which is a simple path in $(\Gamma^{\Mac}, I)$ from $I$ to $p$.    Hence $(\Gamma^{\Mac},I)$ is rooted at $I$.  To prove that $\eta$ is simple and $(\Gamma^{\Mac},I)$ has the unique simple path property it therefore suffices to prove that the paths $\til \eta(p)$ with $p\in \Simple(\Gamma,I)$ are the only simple paths at $I$ in $\Gamma^{\Mac}$.

Let $q\in \Simple(\Gamma^{\Mac},I)$.  We prove by induction on $|q|$ that $q=\til\eta(p)$ for some $p\in \Simple(\Gamma,I)$.  If $|q|=0$, then $q=\til \eta(I)$.  Suppose now that it is true for paths of smaller length than $|q|$.  Then $q=\til \eta(p)(p,e)$ for some simple path $p\in \Simple(\Gamma,I)$ by induction.  Suppose that $pe\notin \Simple(\Gamma,I)$.  Then $\tau(p,e)=p[\tau(e)]$.  But every initial segment of $p$ is visited by $\til \eta(p)$, contradicting that $q$ is simple.  We conclude that $\tau(p,e)=pe\in \Simple(\Gamma,I)$ and so $q=\til\eta(pe)$.
\end{proof}

As immediate consequences of the above result and the universal property we obtain the following corollaries.

\begin{Cor}
$(\Gamma^{\Mac},I)$ is the unique simple cover of $(\Gamma,I)$ with the unique simple path property.
\end{Cor}

\begin{Cor}
If $\p\colon (\Gamma,I)\to (\Gamma',I')$ is a simple covering, then $(\Gamma,I)^{\Mac}\cong (\Gamma',I')^{\Mac}$.
\end{Cor}

\subsection{McCammond expansion of automata}
As mentioned earlier, the McCammond cover of an $A$-automaton is again an $A$-automaton.  In this context we call it the \emph{McCammond expansion} of the automaton.  However, the reader is cautioned that the construction is not functorial.
Let's make the McCammond expansion explicit from an automaton point-of-view.

\begin{Rmk}[Words and simple paths]
Notice that when $\mathscr A$ is a pointed $A$-automaton,
the words readable starting at a vertex $p$ are in one-to-one correspondence
with the paths starting at $p$.  Thus, if we let $\simple(\mathscr A)$ be
the words readable along simple paths starting at the root of
$\mathscr A$, then $\mathscr A$ has the unique simple path property from its
root if and only if the paths corresponding to these words all end at
distinct vertices of $\mathscr A$.
\end{Rmk}

Next, we interpret the action of $A$ on the simple paths starting at a vertex.

\begin{Def}[Acting on simple paths]\label{def:simple-act}
Let $\mathscr A$ be a pointed $A$-au\-to\-ma\-ton with root $p$.  If $u$ is a word readable as a simple path in $\mathscr A$
from $p$ to $q$ (i.e., $u\in \simple(\mathscr A)$) and $a$ is the label on
an edge in $\mathscr A$ which starts at $q$, then there is a well-defined
word in $\simple(\mathscr A)$ that we denote $u\cdot a$.  There are two
possibilities.  If the word  $ua$ is readable as a simple path
starting at $p$, then
$u\cdot a$ is defined to be $ua$;  otherwise, the path starting at
$p$ reading $ua$ ends at a vertex $q'$ already traversed by the path
reading $u$.  In that case, $u\cdot a$ is defined to be the initial
portion of $u$ which is read from $p$ to $q'$.  The two possibilities
are illustrated in Figure~\ref{fig:simple}.   When $u\cdot a = ua$ we
call this an \emph{extension}, and when $u\cdot a$ is an initial
subsegment of $u$ we call it a \emph{reduction}.    Combining these results
gives a partial function from $\simple(\mathscr A) \times A$ to
$\simple(\mathscr A)$.  Notice that this partial function is a function if
and only if $\mathscr A$ is a complete automaton.
\end{Def}

\begin{figure}[ht]
\psfrag{p}{$p$}
\psfrag{q}{$q$}
\psfrag{r}{$q'$}
\psfrag{e}{$a$}
\psfrag{e'}{$a$}
\includegraphics{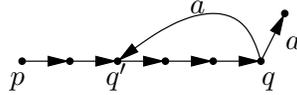}
\caption{Multiplying a simple path by an edge.\label{fig:simple}}
\end{figure}

We are now ready to rephrase the McCammond cover for automata.

\begin{Def}[Expanding automata]\label{def:mac}
Let $\mathscr A$ be an $A$-automaton rooted at
$I$. Then the $A$-automaton $\mathscr A^\Mac$ is called the \emph{McCammmond expansion} of $\mathscr A$.
\end{Def}

 Here is an example of the McCammond expansion.

\begin{figure}[ht]
\psfrag{1}{$1$}
\psfrag{x}{$a$}
\psfrag{y}{$b$}
\psfrag{xy}{$ab=ba$}
\includegraphics{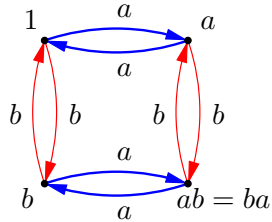}
\caption{The Cayley graph of the Klein four-group.\label{fig:klein}}
\end{figure}

\begin{exmp}[Klein four-group]\label{exmp:klein}
Let $G=\Z_2 \times \Z_2$ be the Klein four-group with the standard
presentation $\langle a,b \mid a^2=b^2=[a,b]=1 \rangle$.  The Cayley
graph of $G$ with respect to $A = \{a,b\}$ is shown in
Figure~\ref{fig:klein}.  In Figure~\ref{fig:klein-mc}, the expanded
graph $\Cay(G,A)^\Mac$ is shown on the left and its directed spanning tree is shown
on the right.  Notice that the expansion has the unique simple path
property from $1$ but not from the state $aba$ since $aba \cdot b= aba
\cdot aba$.  Notice that all edges not in the spanning tree connect
a vertex $p$ to some vertex on the unique simple path in the tree from
$1$ to $p$.
\end{exmp}

\begin{figure}[ht]
\begin{tabular}{ccc}
\psfrag{1}{$1$}
\psfrag{x}{$a$}
\psfrag{y}{$b$}
\psfrag{xy}{$ab$}
\psfrag{yx}{$ba$}
\psfrag{xyx}{$aba$}
\psfrag{yxy}{$bab$}
\includegraphics{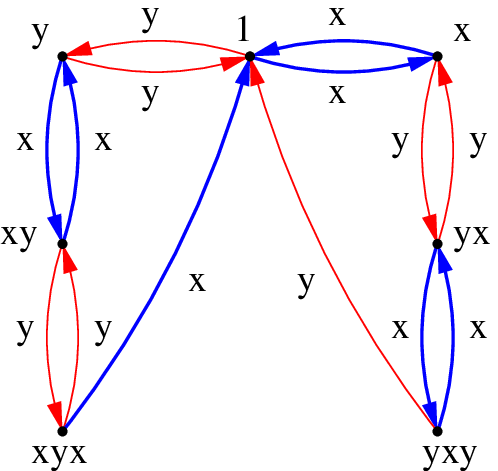} &
\hspace{10mm} &\
\psfrag{1}{$1$}
\psfrag{x}{$a$}
\psfrag{y}{$b$}
\psfrag{xy}{$ab$}
\psfrag{yx}{$ba$}
\psfrag{xyx}{$aba$}
\psfrag{yxy}{$bab$}
\includegraphics{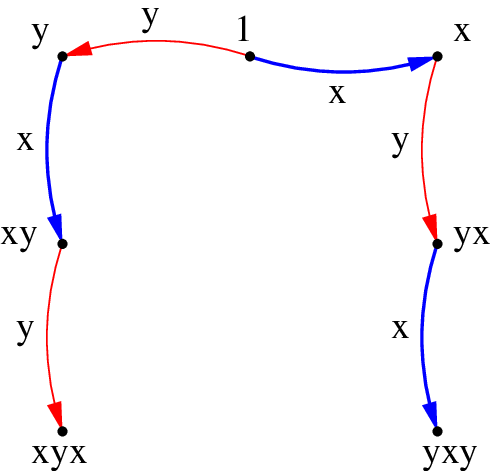}
\end{tabular}
\caption{The expansion of the graph in Figure~\ref{fig:klein} and its
  underlying tree.\label{fig:klein-mc}}
\end{figure}

Our next example shows that the McCammond expansion is not functorial.

\begin{exmp}[The McCammond expansion is not functorial]
We need to include an example.
\end{exmp}

We conclude this section with a conjecture that has a long history
dating back to the 1940s.  We mention it here since its assertion can
be reformulated in terms of depth of the tree underlying the $\Mac$
expansion to the Cayley graph of a finite group.

\begin{Conjecture}[Hamiltonian path]
If $G$ is a finite group with at least $3$ elements and it is
generated as a monoid by $A$, then there exists a simple path
in the Cayley graph of $G$ with respect to $A$ which passes through
every vertex.  In other words, there is a word $w \in A^+$ so that
every element of $G$ is represented uniquely by an initial segment of
$w$.
\end{Conjecture}

\section{Automata from a universal algebra point-of-view}
There is another viewpoint on $A$-automata that takes them to be universal algebras~\cite{universalalgebra}. In what follows we write unary operations on the right of their arguments.  First let's consider deterministic $A$-automata.  Deterministic $A$-automata are exactly universal algebras with signature (or type) consisting of a collection of unary operations indexed by $A$.  Given an $A$-automaton $(\Gamma,\ell)$ we can create an algebra with underlying set $V(\Gamma)$; the unary operation associated to $a\in A$ is $q\mapsto qa$.  Conversely, if $Q$ is a universal algebra of this type, we can form its \emph{Cayley graph} with vertex set $Q$ and set of edges $Q\times A$.  Here $\iota(q,a)=q$ and $\tau(q,a)=qa$.  The labeling sends $(q,a)\to a$.  Therefore, one has universal algebraic notions of generators, relations and homomorphisms for automata.  We shall see shortly how they relate to our topological notions.

One can similarly deal with partial deterministic automata by formally adjoining a sink.  Let us now consider universal algebras with unary operations indexed by $A$ and a distinguished constant $\square$.  We also demand that they satisfy the identity $\square a=\square$ for all $a\in A$.  Let $\mathfrak V$ be the variety of algebras of this type satisfying these identities.  Then, given an $A$-automaton $(\Gamma,\ell)$ we can define an algebra in this type by taking $V(\Gamma)\cup \{\square\}$ as the underlying set, defining $\square a=\square$ for all $a\in A$ and setting \[qa= \begin{cases} qa & qa\ \text{defined}\\ \square & \text{otherwise.}\end{cases}\] Conversely, given an algebra $Q\in \mathfrak V$, we can define its \emph{Cayley graph} by taking the vertex set to be $Q\setminus \{\square\}$ and the edge set to be $\{(q,a)\in Q\times A\mid qa\neq \square\}$.  Of course, $\iota(q,a)=q$, $\tau(q,a)=qa$ and $\ell(q,a)=a$.  Thus we can also talk about generators, relations and homomorphisms for partial automata.

Let us reinterpret these notions topologically.  We say that a morphism $\p$ of $\mathfrak V$-algebras is \emph{$\square$-restricted} if $\p\inv(\square)=\{\square\}$.

\begin{Prop}\label{monoidmorphism}
{}\
\begin{enumerate}
\item The category of algebras in $\mathfrak V$ with surjective $\square$-restricted morphisms is equivalent to the category of $A$-automata with directed covers as morphisms.
\item An $A$-automaton $(\Gamma,\ell)$ is generated by $\mathsf I$ in the sense of universal algebra if and only if $\mathsf I$ is a generating set in the sense of Definition~\ref{pointedauto}. In particular, rooted automata correspond to one-gen\-er\-a\-ted $\mathfrak V$-algebras.
\item The free algebra on a set $\mathsf I$ is $(\mathsf I\times A^*)\cup \{\square\}$ with the obvious $\mathfrak V$-algebra structure.
\end{enumerate}
\end{Prop}
\begin{proof}
To prove (1),  suppose first that $\p\colon \mathscr A_1\to \mathscr A_2$ is a directed covering of $A$-automata and  define $\psi\colon V(\mathscr A_1)\cup\{\square\}\to V(\mathscr A_2)\cup \{\square\}$ by $\psi(\square)=\square$ and  $\psi(v)=\p(v)$ for  $v\in V(\mathscr A_1)$. Clearly, $\psi$ is surjective and $\square$-restricted.  Let us check that $\psi$ is a homomorphism.  Clearly $\psi(\square\cdot a) = \psi(\square)=\square=\square\cdot a$ for all $a\in A$.  Next suppose $q\in V(\mathscr A_1)$ and $a\in A$.  Then by definition of a directed covering, we have that $\p(qa)=\p(q)a$ in the sense that either both sides are undefined or both are defined and agreed.  If both sides are undefined, then $\psi(qa)=\square=\psi(q)a$.  Otherwise, $\psi(qa)=\psi(q)a$.

Conversely, suppose that $\p\colon Q\cup \{\square\}\to Q'\cup \{\square\}$ is a surjective $\square$-restricted morphism of $\mathfrak V$-algebras.  Then we claim that $\p$ induces a directed cover defined by $\p$ on vertices and by $(q,a)\mapsto (\p(q),a)$, for $qa\neq \square$, on edges.  Notice this is well defined since $qa\neq \square$ implies $\p(q)a=\p(qa)\neq \square$ as $\p$ is $\square$-restricted.  Evidentally, $\p$ gives a morphism of $A$-automata.  Surjectivity on vertices is immediate.  Suppose $\p(q)a\neq \square$.  Then $\p(qa)=\p(q)a\neq \square$ and hence $qa\neq \square$ since $\p$ is a homomorphism.  Thus the edge $(\p(q),a)$ is a lift $(q,a)$ (and is the unique one).

The second item is trivial.  To verify (3), if $Q$ is a $\mathfrak V$-algebra generated by $\mathsf I$, then the map defined by $(p,w)\mapsto pw$ for $(p,w)\in \mathsf I\times A^*$ and $\square\mapsto \square$ is the unique morphism from $\mathsf I\times A^*$ to $Q$ extending the map $(p,1)\mapsto p$ for $p\in \mathsf I$.
\end{proof}

We leave to the reader to formulate and prove the analogous result for deterministic automata.


\subsection{Presentations and rewriting}
It now makes sense to talk about presentations of $A$-automata.  Since we are only interested in pointed automata, we restrict to the case of one-generated algebras from $\mathfrak V$ and so we need only give the relations.
It follows now that if $(\mathscr A,I)$ is a pointed automaton, then it is one-generated and can be defined by relations of the form $u=v$  and $u=\square$ where $u,v\in A^*$.  The relation $u=v$ forces $Iu=Iv$ in $\mathscr A$, while the relation $u=\square$ forces $Iu$ to be undefined.  If $\langle R\rangle$ is a presentation of $\mathscr A$, the underlying $\mathfrak V$-algebra of $\mathscr A$ is the quotient of $A^*\cup \{\square\}$ (viewed as a right $A^*$-set) by the right congruence generated by the relations in $R$ and the relations $\square a=a$.  Formally, let us define the one-step derivation relation $\Rightarrow$ by putting $\square a\Rightarrow \square$ for all $a\in A$ and $uw\Rightarrow vw$ if $u=v\in R$ where $u,v\in A^*\cup \{\square\}$, $w\in A^*$. As usual, $\Rightarrow^*$ denotes the reflexive transitive closure of the relation $\Rightarrow$.  We write $u\sim v$ for $u,v\in A^*\cup \{\square\}$ if there is a sequence $u=u_0,u_1\cdots, u_n=v$ such that,  for each $i=0,\ldots,n-1$, either $u_i\Rightarrow u_{i+1}$ or $u_{i+1}\Rightarrow u_i$.  Then the state set of $\mathscr A$ is $(A^*\cup \{\square\})/{\sim}$.  The root is the class $I$ of the empty word.  We remark that we are essentially dealing with what is known as prefix rewriting systems~\cite{ChurchRosser}: relations can only be applied to prefixes of words.

We state an obvious result.

\begin{Prop}\label{relationsforcayley}
Let $S=\langle A\mid R\rangle$ be a presentation of an $A$-semigroup.   Then the rooted Cayley graph $(\Cay(S,A),I)$ is presented as an $A$-automaton by $\langle \{ur=ur'\mid u\in A^*, r=r'\in R\}\rangle$.
\end{Prop}

\begin{Def}[Loop automaton]
A pointed automaton $(\mathscr A,I)$ is said to be a \emph{loop automaton} if it has a presentation of the form $\langle L\cup D\rangle$ where all relations in $L$ are so-called \emph{loop relations} of the form $uv=u$  and all relations in $D$ are \emph{dead-end} relations $w=\square$ where $u,v,w\in A^*$.
\end{Def}

We now prove that automata with the unique simple path property are loop automata.

\begin{Prop}\label{churchrosser}
Let $(\mathscr A,I)$ be a partial $A$-automaton with the unique simple path property.  Then
\begin{align*}
\mathscr A=\langle &\{\ell(\slp(e))= \ell([I,\tau(e)])\mid e\in \pv E(\Gamma)\} \cup {}\\ &
 \{\ell([I,v])a=\square \mid va\ \text{is undefined}, v\in V(\Gamma), a\in A\}\rangle
\end{align*}
where $\ell\colon E(\Gamma)\to A$ is the labeling function.  Thus $(\mathscr A,I)$ is a loop automaton.
\end{Prop}
\begin{proof}
Clearly $\mathscr A$ satisfies all these relations.  We show all other relations are consequences of these.  Let $\sim$ be the equivalence relation on $A^*\cup \{\square\}$ associated to the presentation above.  First we claim that if $Iu$ is defined for $u\in A^*$, then $u\sim \ell([I,Iu])$.  The proof goes by induction on the number of bold arrows in the path $p_u$ read by $u$ from $I$ to $Iu$.  If there are no bold arrows, then $p_u=[I,Iu]$ and there is nothing to prove. Else we can factor $p_u=vew$ where $e$ is the first bold arrow used by $p_u$.   If $q=\tau(v)$, then  $v=[I,q]$ and $\slp(e) = [I,q]e=ve$. See Figure~\ref{Fig:reducesloop}.
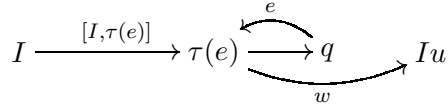
\begin{figure}[thbp]
\[\SelectTips{cm}{}
\xymatrix{I\ar[rr]^{[I,\tau(e)]}&&\tau(e)\ar[r]\ar@/_1pc/[rr]_w&q\ar@/_1pc/[l]_e&Iu}\]
\caption{Reducing to the base\label{Fig:reducesloop}}
\end{figure}
Thus $\ell(ve)\sim \ell([I,\tau(e)])$ and so $u=\ell(vew)\sim \ell([I,\tau(e)])\ell(w) = \ell([I,\tau(e)]w)$.  But $[I,\tau(e)]w$ contains strictly fewer bold arrows than $p_u$ and hence by induction $u\sim \ell([I,\tau(e)]w)\sim \ell([I,\tau(w)]) = \ell([I,Iu])$.

Suppose now that $u,v\in A^*$ and $Iu=Iv$ with both sides defined.  Then by the claim $u\sim \ell([I,Iu])= \ell([I,Iv])\sim v$.  If $u\in A^*$ and $Iu$ is undefined, then we can factor $u=vaw$ uniquely with $Iv$ defined, $a\in A$ and $Iva$ undefined.  Then by the claim $v\sim \ell([I,Iv])$.  Hence $u\sim \ell([I,Iv])aw\sim \square$ using the relation $\ell([I,Iv])a=\square$.  This completes the proof that $\mathscr A$ is given by the above presentation.
\end{proof}

The above proof provides a Church-Rosser type property~\cite{ChurchRosser} for pointed automata with the unique simple path property.  Namely, we have shown that $u\sim \square$ or $u\sim \ell([I,Iu])$ can be derived using only $\Rightarrow^*$.  Hence if we form a rewriting system with rules
\begin{alignat*}2
\ell(\slp(e))&\rightarrow \ell([I,\tau(e)])& \text{for $e\in \pv E(\Gamma)$,}\\
\ell([I,v])a&\rightarrow\square & \text{if $va$ is undefined,}\\
\square a&\rightarrow \square &
\end{alignat*}
then we obtain a confluent terminating prefix rewriting system (where der\-i\-va\-tions are defined via $\Rightarrow$).   The unique irreducible form derivable from $w\in A^*$ is $\square$ if $Iw$ is not defined and otherwise is $\ell([I,Iw])$, which from now on we denote $\red(w)$.  Here we consider $\square$ to be irreducible, whereas a word $w\in A^*$ is irreducible if it has no prefix which is the left hand side of a rule.

Let us formalize this.

\begin{Def}[Elementary loop automaton]
A loop automaton $\mathscr A$ is said to be an \emph{elementary loop automaton} if it can be given a loop automaton presentation $\langle L\cup D\rangle$ with the property that distinct irreducible elements for the prefix rewriting system
\begin{alignat*}2
uv&\rightarrow u &\quad \text{such that}\ uv=u\in L,\\u&\rightarrow \square &\quad \text{such that}\ u=\square\in D,\\ \square a&\rightarrow \square & a\in A
\end{alignat*}
represent distinct classes of $\sim$.
\end{Def}

It turns out that elementary loop automata are exactly the automata with unique simple path property.
\begin{Prop}
A pointed $A$-automaton $(\mathscr A,I)$ has the unique simple path property if and only if it is an elementary loop automaton.
\end{Prop}
\begin{proof}
The proof of Proposition~\ref{churchrosser} shows that if $(\mathscr A,I)$ has the unique simple path property, then it is an elementary loop automaton with the presentation in that proposition. We prove the converse.  Suppose $(\mathscr A,I)$ is an elementary loop automaton with loop presentation $\langle L\cup D\rangle$ and let $N$ be the set of irreducible elements of $A^*$; notice that it is prefix-closed. Since each element of $N$ gives a unique state of $\mathscr A$ and $N$ is prefix-closed, it follows that subgraph of $\mathscr A$ spanned by the paths labeled by the elements of $N$ from $I$ is a directed rooted tree isomorphic to the induced subgraph of the Cayley graph of $A^*$ with vertices the elements of $N$.  Moreover, it must be a spanning tree since if $q$ is a state of $\mathscr A$, then we can find a word $w\in A^*$ with $Iw=q$.  Replacing $w$ by the unique irreducible element to which it can be reduced shows that $T$ is a spanning tree.  We claim that each edge $e$ of $\mathscr A$ that does not belong to $T$ satisfies $\tau(e)\in [I,\iota(e)]$.  Proposition~\ref{defineuspp} will then imply that $(\mathscr A,I)$ has the unique simple path property.  So suppose $u\in N$, $ua\notin N$ with $a\in A$ and $Iua$ defined.  Since $ua$ is not irreducible and cannot be reduced to $\square$, we can write $ua=xyw$ where $xy=x\in L$.  Since $u$ is irreducible, we must have that $w$ is empty and $y$ ends in $a$.  That is, $u=xy'$ where $y=y'a$ and $ua=xy$.  But then $ua=xy\sim x$ and $x$ is a prefix of $u$.  Thus if $e$ is the edge labeled by $a$ at $Iu$, then $\tau(e)=Ix$, which is a vertex of $[I,Iu]$.  This completes the proof.
\end{proof}

\subsection{The standard Kleene expression}

We now aim to define a unionless Kleene expression for the language accepted by a finite trim pointed acceptor $(\mathscr A,I,\{q\})$ with the unique simple path property and a single terminal state $q$.  This expression turns out to be good for pumping arguments and has the advantage of being unionless.  The intuition is that for applications to complexity, one wants to associate to the automaton a unionless Kleene expression and then replace each occurrence of the Kleene star $\ast$ by $\omega$ in order to obtain an element of the free aperiodic $\omega$-algebra~\cite{Mc-nf}.

Let $\mathscr A=(\Gamma,\ell)$. We assume that a geometric rank function $r$ on $\mathscr A$ has been fixed.  It follows from Proposition~\ref{straightlineinuspp} and $(\mathscr A,I,\{q\})$ being trim that $\mathscr A$ is linear.  Clearly the vertex $q$ belongs to the bottomost strong component and $I$ belongs to the topmost one.  The Kleene expression $\mathcal K(\mathscr A,I,q,r)$ is defined via the \emph{Principal of Induction} (or PI), which is based on induction on the pair $(|\mathbf E(\Gamma)|,|V(\Gamma)|)$  where $\mathbb N\times \mathbb N$ is ordered lexicographically.  Let $C_0>C_1>\cdots> C_n$ be the chain of strong components of $\Gamma$ and let $p_i$ and $q_i$ be the respective entry and exit points of $C_i$ where we take $p_0=I$.  Let $e_i\colon q_{i-1}\to p_i$, for $i=1,\ldots,n$, be the $i^{th}$ transition edge.  See Figure~\ref{Fig:definekleene}.
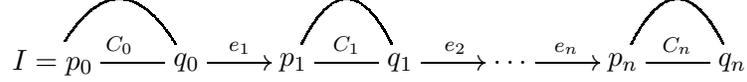
\begin{figure}[thbp]
\[\SelectTips{cm}{}
\xymatrix{I=p_0\ar@{-}@/^2pc/[r]\ar@{-}[r]^{C_0} & q_0\ar[r]^{e_1}&p_1\ar@{-}@/^2pc/[r]\ar@{-}[r]^{C_1}& q_1\ar[r]^{e_2}&\ldots\ar[r]^{e_n}&p_n\ar@{-}@/^2pc/[r]\ar@{-}[r]^{C_n}  &q_n}\]
\caption{Trim automaton with the unique simple path property\label{Fig:definekleene}}
\end{figure}
Notice that each $(C_i,p_i)$ has the unique simple path property.  It is possible that $p_i=q_i$ and some of the $C_i$ may be trivial.

\subsubsection{Base of the induction}
If $(\mathscr A,I)$ has no bold arrows, then it consists simply of the geodesic $[I,q]$ and we define
\begin{equation}\label{Kleene1}
\mathcal K(\mathscr A,I,q,r)=\ell([I,q])
\end{equation}
where $\ell$ is the function labeling the edges of $\mathscr A$.

\subsubsection{Elementary induction}\label{elementaryinduction}
If $n\geq 1$, then we define
\begin{equation}\label{Kleene2}
\mathcal K(\mathscr A,I,q,r) = \mathcal K\left(\mathscr A(q_{n-1}),I,q_{n-1},r|_{\mathscr A(q_{n-1})}\right)\ell(e_n)\mathcal K\left(C_n,p_n,q,r|_{C_n}\right)
\end{equation}
where $\mathscr A(q_{n-1})$ is the full subautomaton of $\mathscr A$ whose vertices are greater than or equal to $q_{n-1}$ in the accessibility order (cf.~Proposition~\ref{straightlineinuspp}).  This make sense since $(\mathscr A(q_{n-1}),I)$ and $(C_n,p_n)$ have the unique simple path property and each have strictly fewer vertices than $\mathscr A$ and no more bold arrows.

\subsubsection{Loop case}\label{ss:loopcase}
Suppose that $n=0$ (so there is only one strong component) and that $\mathscr A$ has at least one bold arrow. Let $e$ be the bold arrow ending at $I$ with largest geometric rank.  Let $\mathscr B$ be the automaton obtained from $\mathscr A$ by removing the edge $e$ and taking the strong component of $I$.  Then $(\mathscr B,I)$ has the unique simple path property and has fewer bold arrows then $\mathscr A$.  Also $\cut(e)$ has the unique simple path property from $I$ by Propositions~\ref{cutkeepsuspp} and is linear by Proposition~\ref{cutislinear}.  Finally, we observe that $\ov{([I,q],I)}$ has the unique simple path property.  Moreover, every vertex of $\ov{([I,q],I)}$ is on a path from $I$ to $q$ by Proposition~\ref{closurepreservesfilters}.  Thus $\ov{([I,q],I)}$ is trim.  It also has fewer bold arrows than $\mathscr A$ since the graph obtained by removing the bold arrows ending at $I$ is closed with respect to $I$ and contains $[I,q]$.  Thus we may define
\begin{equation}\label{Kleene3}
\begin{split}
\mathcal K(\mathscr A,I,q,r) = \mathcal K\left(\mathscr B,I,I,r|_{\mathscr B}\right)&\left[\mathcal K\left(\cut(e),I,I',r|_{\cut(e)}\right) \mathcal K\left(\mathscr B,I,I,r|_{\mathscr B}\right)\right]^*\\ &\cdot \mathcal K\left(\ov{([I,q],I)},I,q,r|_{\ov{([I,q],I)}}\right)
\end{split}
\end{equation}
where $r|_{\cut(e)}$ makes sense as each bold arrow of $\cut(e)$ is a bold arrow of $\mathscr A$.

Before verifying that the Kleene expression yields the correct language, we compute some examples.  In the first few examples there will be no multiple bold arrows ending at the same vertex and so we do not need to give the geometric rank function.

\begin{exmp}
Suppose the automaton $\mathscr A$ is given by \[\SelectTips{cm}{}
\xymatrix{I\ar[r]^\Delta &q\ar@(r,u)_{\alpha}}\]  Then $\mathcal K(\mathscr A,I,q,r)= \Delta\alpha^*$.
\end{exmp}

\begin{exmp}
Suppose the automaton $\mathscr A$ is given by
\[\SelectTips{cm}{}
\xymatrix{I\ar[r]^\Delta& p\ar@/_1pc/[r]_{\alpha_1}&q\ar@/_1pc/[l]_{\alpha_2}}\] Then $\mathcal K(\mathscr A,I,q,r)= \Delta(\alpha_1\alpha_2)^*\alpha_1$.
\end{exmp}

\begin{exmp}
Let the automaton $\mathscr A$ be given by
\[\SelectTips{cm}{}
\xymatrix{I\ar[r]^\Delta& q\ar[r]^{a}&\bullet\ar[r]^b&\bullet\ar@/^1pc/[l]^{d}\ar@/_2pc/[ll]_c}\]
Then \[\cut(c)= \SelectTips{cm}{}
\xymatrix{q\ar[r]^{a}&\bullet\ar[r]^b&\bullet\ar@/^1pc/[l]^{d}\ar[r]^c &q'}\]
A computation similar to the previous example shows that \[\mathcal K(\cut(c), q,q',r) =  a(bd)^*bc.\]
Thus $\mathcal K(\mathscr A,I,q,r)= \Delta(a(bd)^*bc)^*$.
\end{exmp}

Our next example modifies the previous one by changing the terminal state $q$.

\begin{exmp}
Suppose the automaton $\mathscr A$ is given by
\[\SelectTips{cm}{}
\xymatrix{I\ar[r]^\Delta& p\ar[r]^{a}&\bullet\ar[r]^b&q\ar@/^1pc/[l]^{d}\ar@/_2pc/[ll]_c}\]
The only new ingredient is to compute $\ov{([p,q],p)}$, which is \[\SelectTips{cm}{}
\xymatrix{p\ar[r]^{a}&\bullet\ar[r]^b&\bullet\ar@/^1pc/[l]^{d}}\] and so has Kleene expression \[\mathcal K\left((\ov{[p,q],p}),p,q,r\right)= a(bd)^*b.\]  Thus  $\mathcal K(\mathscr A,I,q,r)= \Delta(a(bd)^*bc)^*a(bd)^*b$.
\end{exmp}

In our next example, we will make use of the geometric rank function.
\begin{exmp}
Let the automaton  $\mathscr A$ be given by
\[\SelectTips{cm}{}
\xymatrix{I\ar[r]^\Delta &q\ar@(r,u)_{\alpha}\ar@(r,d)^{\beta}}\]
and let us suppose that $\beta$ has greater geometric rank than $\alpha$.  Then  the Kleene expression is $K(\mathscr A,I,q,r)=\Delta\alpha^*(\beta\alpha^*)^*$.
\end{exmp}

A slight variation on the previous example is the following.
\begin{exmp}
Let the automaton  $\mathscr A$ be given by
\[\SelectTips{cm}{}
\xymatrix{     &                                     &\bullet\ar@(r,u)_{\gamma}\ar@/_1pc/[ld]_{\alpha_2}\\
I\ar[r]^\Delta &q\ar@/_1pc/[ru]_{\alpha_1}\ar@(r,d)^{\beta}&}\]
and let us suppose that $\beta$ has greater geometric rank than $\alpha_2$.  Then
\[\cut(\alpha_2) = \SelectTips{cm}{}
\xymatrix{q\ar[r]^{\alpha_1}&\bullet\ar@(ur,ul)_{\gamma}\ar[r]^{\alpha_2}&q'}\]
and so $K(\mathscr A,I,q,r)=\Delta(\alpha_1\gamma^*\alpha_2)^*\left[\beta(\alpha_1\gamma^*\alpha_2)^*\right]^*$.
\end{exmp}

Our final example mixes several ingredients.

\begin{exmp}
Suppose the automaton  $\mathscr A$ is given by
\[\SelectTips{cm}{}
\xymatrix{     &   q\ar@/_/[d]_{\alpha_3}                                  &\bullet\ar@(r,u)_{\gamma}\ar@/_/[l]_{\alpha_2}\\
I\ar[r]^\Delta &p\ar@/_1.5pc/[ru]_{\alpha_1}\ar@(r,d)^{\beta}&}\]
and let us suppose that $\alpha_3$ has greater geometric rank than $\beta$.  Then
\[\cut(\alpha_3) = \SelectTips{cm}{}
\xymatrix{p\ar[r]^{\alpha_1}&\bullet\ar@(ur,ul)_{\gamma}\ar[r]^{\alpha_2}&q\ar[r]^{\alpha_3}&p'}\] and
\[\ov{([p,q],p)} = \SelectTips{cm}{}
\xymatrix{p\ar[r]^{\alpha_1}&\bullet\ar@(ur,ul)_{\gamma}\ar[r]^{\alpha_2}&q}\]
Thus $K(\mathscr A,I,q,r)=\Delta\beta^*\left[(\alpha_1\gamma^*\alpha_2\alpha_3)\beta^*\right]^*\alpha_1\gamma^*\alpha_2$.
\end{exmp}

We now prove that the Kleene expression accepts the language recognized by the automaton.

\begin{Thm}
Let $(\mathscr A,I,\{q\})$ be a finite trim pointed $A$-acceptor with the unique simple path property and with a single terminal state $q$ where $\mathscr A=(\Gamma,\ell)$.   Fix a geometric rank function $r$ on $(\mathscr A,I)$.  Then $\mathcal K(\mathscr A,I,q,r)$ is a Kleene expression for the language of words accepted by $\mathscr A$ with initial state $I$ and terminal state $q$ that does not use union.
\end{Thm}
\begin{proof}
The proof goes by induction on the pair $(|\pv E(\Gamma)|,|V(\Gamma)|)$ where we order $\mathbb N\times \mathbb N$ lexicographically.  It is clear that if $\mathscr A$ has no bold arrows, then $\mathcal K(\mathscr A,I,q,r)$ is unionless and accepts the same language as $\mathscr A$.  Similarly, in the case of the elementary induction~\ref{elementaryinduction}, one easily verifies that $\mathcal K(\mathscr A,I,q,r)$ is a unionless Kleene expression accepting the language of $\mathscr A$ using the induction hypothesis.   The only difficult case is the loop case~\ref{ss:loopcase}.

Let us retain the notation of \ref{ss:loopcase}.  By induction we have $\mathscr B$ with initial and terminal state both set to $I$ accepts precisely $\mathcal K(\mathscr B,I,I,r|_{\mathcal B})$, whereas $\mathcal K\left(\cut(e),I,I',r|_{\cut(e)}\right)$ accepts all strings reading from $I$ to $I'$ in $\cut(e)$ and $\mathcal K\big(\ov{([I,q],I)},I,q,r|_{\ov{([I,q],q)}}\big)$ accepts all words reading in  $\ov{([I,q],I)}$ from $I$ to $q$   Then trivially, $\mathcal K(\mathscr A,I,q,r)$ is accepted by $\mathscr A$ with initial state $I$ and final state $q$ and is a unionless Kleene expression using the induction and \eqref{Kleene3}.  The converse direction requires some proof.

Suppose $w$ labels a path from $I$ to $q$.    We can factor $w=uv$ with $u,v\in A^*$ so that $Iu=I$ and the path labeled by $v$ from $I$ to $q$ does not revisit $I$. First we claim that $v$ is accepted by $\ov{([I,q],I)}$ with initial state $I$ and terminal state $q$.  If $q=I$, then $v$ is empty and there is nothing to prove.  Otherwise, we proceed by induction on the number of applications of $\Rightarrow$ needed to reduce $v$ to $\ell([I,q])$ in the rewriting system appearing after Proposition~\ref{churchrosser}.  If $v=\ell([I,q])$, then there is nothing to prove.  Suppose $v\Rightarrow x$.  Then we can write $v=\ell(\slp(f))z$ for some bold arrow $f$ with $x=\ell([I,\tau(f)])z$.  By choice of $v$, the edge $f$ does not end at $I$.  By induction, $x$ is accepted by  $\ov{([I,q],I)}$.  Since $f$ is a bold arrow with $\tau(f)$ visited by $x$ on its run from $I$, it follows from the definition of being closed with respect to $I$ that $f\in  \ov{([I,q],I)}$.  Thus $\lp(f)\subseteq \ov{([I,q],I)}$ by Proposition~\ref{loopin}.  Therefore, $\slp(f) = [I,\tau(f)]\lp(f)\subseteq  \ov{([I,q],I)}$ and so $v$ is accepted by  $\ov{([I,q],I)}$ with terminal state $q$.  Hence $v\in \mathcal K\left(\ov{([I,q],I)},I,q,r|_{\ov{([I,q],q)}}\right)$.

It thus remains to show that
\begin{equation*}\label{bigmesstoshow}
u\in  \mathcal K\left(\mathscr B,I,I,r|_{\mathscr B}\right)\left[\mathcal K\left(\cut(e),I,I',r|_{\cut(e)}\right) \mathcal K\left(\mathscr B,I,I,r|_{\mathscr B}\right)\right]^*.
\end{equation*}
Let $s$ be the path read by $u$ from $I$ to $I$ and factor $p=s_0es_1e\cdots s_{m-1}es_m$ where no $s_i$ uses the edge $e$ and where we interpret the case $m=0$ as $p=s_0$.  Then $\ell(s_m)$ is accepted by $\mathscr B$ with initial and terminal state $I$.  We may further factor each $s_i$ with $0\leq i\leq m-1$ as $s_i=t_iu_i$ so that $t_i\colon I\to I$ and so that $u_i\colon I\to \iota(e)$ does not revisit $I$.  Then $\ell(t_i)$ is accepted by $\mathscr B$ with initial and terminal state $I$ for all $0\leq i\leq m-1$.  We claim that $\ell(u_i)a$ is accepted by $\cut(e)$ with intial state $I$ and terminal state $I'$.  If $I=\iota(e)$ (that is, $e$ is a loop edge), then each $u_i$ is empty and $\cut(e) = I\xrightarrow{\, a\,} I'$ and so there is nothing to prove.  So assume $I\neq \iota(e)$.
We show by induction on the number of applications of $\Rightarrow$ needed to reduce $\ell(u_i)$ to $\ell([I,\iota(e)])$ in the rewriting system after Proposition~\ref{churchrosser} that $\ell(u_i)$ reads in $\cut(e)$ from $I$ to $\iota(e)$.  Since $a$ labels an edge from $\iota(e)$ to $I'$ in $\cut(e)$, this will establish the claim.  Recall that $\cut(e)$ is obtained by cutting $\ov{(\lp(e),I)}$ at $I$.  If $u_i=[I,\iota(e)]$ there is nothing to prove since $[I,\iota(e)]\subseteq \ov{\lp(e)}$ and doesn't use the edge $e$, and hence ``is'' the geodesic from $I$ to $\iota(e)$ in $\cut(e)$.  Thus $\ell(u_i)$ reads from $I$ to $\iota(e)$ in $\cut(e)$.  Next suppose that $\ell(u_i)\Rightarrow x$ where $x$ reads a path from $I$ to $\iota(e)$ in $\cut(e)$.   By definition of $\Rightarrow$, we can write $\ell(u_i) = \ell(\slp(f))z$ for some bold arrow $f$ so that $x=\ell([I,\tau(f)])z$.  Since $u_i$ does not revisit $I$, it follows that $f$ does not end at $I$.  Then since $x$ labels a path in $\ov{(\lp(e),I)}$ from $I$ to $\iota(e)$, it follows by the definition of a closed subgraph at $I$ that $f\in \ov{(\lp(e),I)}$ and hence, by Proposition~\ref{loopin}, we must have $\lp(f)\subseteq \ov{(\lp(e),I)}$.  Consequently, $\slp(f)=[I,\tau(f)]\lp(f)$ is contained in $\ov{(\lp(e),I)}$.  Because $u_i$ does not revisit $I$, it follows that $\ell(u_i)=\ell(\slp(f))z$ is readable from $I$ to $\iota(e)$ in $\cut(e)$.

Thus we now have (using the inductive hypothesis)
\begin{align*}
u&=[\ell(t_0)(\ell(u_0)a)]\cdots [\ell(t_{m-1})(\ell(u_{m-1})a)]\ell(s_m)\\ &\in \left[\mathcal K\left(\mathscr B,I,I,r|_{\mathscr B}\right)\mathcal K\left(\cut(e),I,I',r|_{\cut(e)}\right)\right]^{m-1}\mathcal K\left(\mathscr B,I,I,r|_{\mathscr B}\right)\\
&\subseteq \mathcal K\left(\mathscr B,I,I,r|_{\mathscr B}\right)\left[\mathcal K\left(\cut(e),I,I',r|_{\cut(e)}\right) \mathcal K\left(\mathscr B,I,I,r|_{\mathscr B}\right)\right]^*
\end{align*}
as required.  This completes the proof of the theorem.
\end{proof}

\section{Semigroup expansions}

\subsection{Semigroups}\label{sec:semigroups}

Once again we fix an alphabet $A$ (usually assumed finite).
In later sections we often require semigroups to have certain
finiteness properties such as being finite, or at least finite
$\J$-above.  A semigroup $S$ is \emph{finite $\J$-above} if there are
only finitely many elements in $S$ which are $\J$-above (i.e., $\geq_{\J}$) any particular
element of $S$.  We also need the following equivalent
characterization.  In this paper we allow ideals to be empty.

\begin{lem}\label{lem:fja}
A semigroup $S$ is finite $\J$-above if and only if there exists a
family of co-finite ideals whose intersection is empty.
\end{lem}
\begin{proof}
If such a family of ideals exists, then for each $s\in S$ there exists
an ideal $I$ in the family with $s\notin I$. Since $S\setminus I$ is finite
and contains all the elements of $S$ that are $\J$-above $s$, it follows that $S$ is
finite $\J$-above.  Conversely, if $S$ is finite then the empty ideal
is co-finite and the result is immediate. If $S$ is infinite and
finite $\J$-above, then for each $\J$-class $J$ let $I_J$ denote the
elements in $S$ which are not $\J$-above $J$.  Clearly these ideals
are co-finite and we will show that their intersection is empty.
Since $S$ is infinite, for each $s\in S$ there is an element $t$ which
is not $\J$-above $s$.  Because $s$ is not in the ideal corresponding
to the $\J$-class of $st$, evidentally $s$ is not in the intersection of all such
ideals and thus the intersection of these ideals is empty.
\end{proof}

\begin{Def}[Categories]
In Table~\ref{tab:cat} we list the categories commonly used throughout
the article and the notations used to describe them.  When the objects
are $A$-semigroups, the morphisms are restricted to $A$-morphisms.
\end{Def}

\begin{table}
\begin{tabular}{|c|c|}
\hline {\bf Notation} & {\bf Description}\\
\hline $\SG$ & semigroups\\
\hline $\FS$ & finite semigroups\\
\hline $\FJ$ & finite $\J$-above semigroups\\
\hline $\SGA$ & $A$-semigroups\\
\hline $\FSA$ & finite $A$-semigroups\\
\hline $\FJA$ & finite $\J$-above $A$-semigroups\\
\hline
\end{tabular}
\vspace*{1em}
\caption{Commonly used categories\label{tab:cat}}
\end{table}

\subsection{Straightline automata}\label{sec:automata}
In this subsection we define two finite-state automata for each finite
$\J$-above $A$-semigroup $S$ and each word $w \in A^+$, one which
depends on the word $w$ (the straightline automaton), and a second
which only depends on the element in $S$ which $w$ represents (the
Cayley automaton). These two types of automata will be crucial to all
of the results that follow.

Straightline automata play a critical role in the solution to the word problem for Burnside semigroups~\cite{Mc91}.

\begin{Def}[Straightline automata]\label{def:str}
Let $(\mathscr A,I)$ be a pointed $A$-au\-to\-ma\-ton and $w$ a word in $A^+$. The
\emph{straightline automaton of $w$ with respect to $\mathscr A$}, denoted
$\str^\mathscr A(w)$, has an underlying directed graph which is a subgraph of $\mathscr A$.  The directed
graph underlying the automaton $\str^\mathscr A(w)$ is the union of the path
from $I$ labeled by $w$, and the strong components of each of the vertices
that this path passes through. This labeled directed graph is made into
an acceptor by specifying the vertex $I$ as its initial state and $Iw$ as its terminal state.

If $S$ is an $A$-semigroup and $w$ is a word in $A^+$,  then the
\emph{straightline automaton of $w$ with respect to $S$} is by definition
$\str^S(w)=\str^{\Cay(S,A)}(w)$. Observe that its
vertex set consists of all of the vertices that represent elements of
$S^I$ which are $\R$-equivalent to elements of $S^I$ represented
by initial segments of $w$.
\end{Def}

\begin{Rmk}
If $(\mathscr A,I)$ is a finite pointed linear $A$-automaton and $w\in A^+$ reads from $I$ to a vertex in the bottommost strong component of $\mathscr A$, then $\mathscr A=\str^\mathscr A(w)$.
\end{Rmk}

The following properties of $\str^S(w)$ follow easily from the
definition.

\begin{lem}[$\str^S(w)$]\label{lem:str-props}
If $S$ is an arbitrary finite $\J$-above $A$-semigroup and $w$ is word
in $A^+$, then $\str^S(w)$ (with initial state $I$ and terminal state $[w]_S$) is a trim, deterministic finite-state
acceptor which is linear and its strong components are
Sch\"utzenberger graphs of $\Cay(S,A)$.  In addition, a word $x\in
A^+$ is readable on $\str^S(w)$ starting at the initial state if and
only if $x$ is $\R$-equivalent to an initial segment of $w$.
\end{lem}

The transition edges of the linear graph $\str^S(w)$ are
exactly those edges whose endpoints represent elements of $S$ which
are not $\R$-equivalent to each other, in other words, edges which
represent descents in the $\R$-order.  Thus every edge in $\str^S(w)$
either is a transition edge or it lies in a unique Sch\"utzenberger
graph.  Let the vertices $p_i$ and $q_i$ be defined as in
Definition~\ref{def:exits} and let $s_i$ and $t_i$ denote the
elements of $S$ which corresponding to these vertices.  In addition, if
$\str^S(w)$ has exactly $k+1$ strong components, we define $p_0$ to be
the initial state of $\str^S(w)$ and $q_k$ to be its terminal state.  For
$i=0,1,\ldots,k$, the vertex $p_i$ will be the first vertex in the
path $w$ which lies in the $i^{th}$ strong component and $q_i$ will be
the last such vertex.  Thus, $p_i$ and $q_i$ can also be identified
with specific initial segments of $w$.  Under this scheme $p_0=q_0$
corresponds to the empty prefix and $q_k$ corresponds to word $w$
itself.  Notice that the vertices $p_0$ and $q_0$ will always
correspond since the Cayley graph used to define $\str^S(w)$ is that
of $S^I$.

\begin{exmp}
Figure~\ref{fig:str} illustrates the structure of a typical
straightline automaton.  In this example there are $k+1$ strong
components.  The $0^{th}$ strong component is trivial, as it always
is, and the entry and exit points are equal in the third strong
component.
\end{exmp}

\psfrag{u0}{$p_0=q_0$}
\psfrag{u1}{$p_1$}
\psfrag{u2}{$p_2$}
\psfrag{u3}{$p_3=q_3$}
\psfrag{uk}{$p_k$}
\psfrag{v1}{$q_1$}
\psfrag{v2}{$q_2$}
\psfrag{vk}{$q_k$}

\psfrag{schu1}{$\sch^T(s_1)$}
\psfrag{schu2}{$\sch^T(s_2)$}
\psfrag{schu3}{$\sch^T(s_3)$}
\psfrag{schuk}{$\sch^T(s_k)$}
\psfrag{schun}{$\sch^T(s_n)$}
\begin{figure}[ht]
\includegraphics{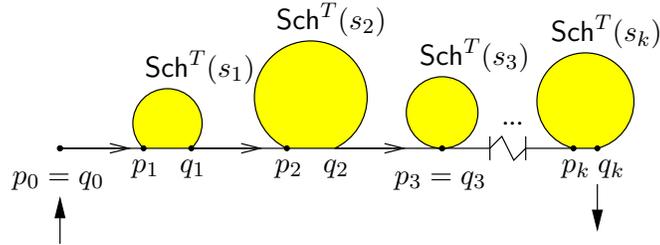}
\caption{A straightline automaton}\label{fig:str}
\end{figure}

Of course, there is an obvious analogue of Lemma~\ref{lem:str-props} for straightline automata associated to finite pointed automata.

\begin{Def}[Cayley automata]\label{def:cayley}
Let $S$ be an $A$-semigroup and let $w$ be a word in $A^+$.  The
\emph{Cayley automaton of $w$ with respect to $S$} is the full
subgraph of $\Cay(S,A)$ on the set of vertices which represent
elements of $S^I$ $\R$-above $[w]_S$.  The initial state and terminal state
are defined as the vertices $I$ and $[w]_S$, as before.  The Cayley
automaton of $w$ will be denoted $\Cay^S(w)$.  Notice that $\Cay^S(w)$
actually only depends on the element of $S$ that $w$ represents and
not on the word itself.  Thus, following the same convention as for
Sch\"utzenberger graphs, if $[w]_S = s$, we might write $\Cay^S(s)$
instead.
\end{Def}

The following lemmas records some elementary properties of $\Cay^S(w)$.

\begin{lem}[$\Cay^S(w)$]\label{lem:cay-props}
If $S$ is a finite $\J$-above $A$-semigroup and $w$ is a word in $A^+$,
then $\Cay^S(w)$ (with initial state $I$ and terminal state $[w]_S$) is a trim, deterministic, finite-state automaton
which accepts the language of words equivalent to $w$ in $S$.  In
addition, the finite-state automaton $\str^S(w)$ is a subautomaton of
$\Cay^S(w)$.  Moreover, a word $x\in A^+$ is readable on
$\Cay^S(w)$ if and only if $[x]_S \geq_{\mathscr J} [w]_S$, it is readable on
$\Cay^S(w)$ starting at the initial state if and only if $[x]_S \geq_{\mathscr R}
[w]_S$, and it is readable on $\Cay^S(w)$ ending at $[w]_S$ if and
only if $[x]_S \geq_{\eL} [w]_S$.
\end{lem}

\begin{exmp}
Figure~\ref{fig:cay} illustrates the structure of a possible Cayley
automaton.  Notice that the induced order on strong components is this
example is not linear, and that transition edges between two
particular strong components are not always unique.  On the other
hand, the automaton $\str^S(w)$ shown in Figure~\ref{fig:str} is
visible as a subautomaton of $\Cay^S(w)$.
\end{exmp}

\begin{figure}[ht]
\includegraphics{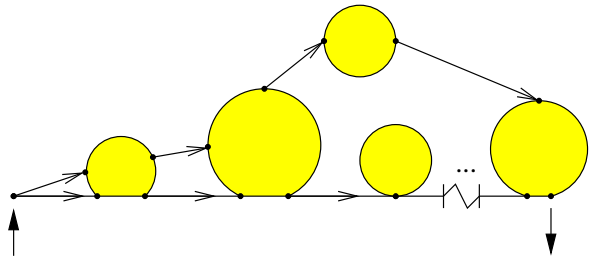}
\caption{A Cayley automaton}\label{fig:cay}
\end{figure}

The situation where $\str^S(w)$ and $\Cay^S(w)$ coincide is
particularly nice.

\begin{lem}[Equal]\label{lem:str=cay}
Let $S$ be a finite $\J$-above $A$-semigroup and let $w$ be a word in
$A^+$.  If $\str^S(w) = \Cay^S(w)$, then this trim, deterministic
finite state acceptor --- which depends only on the element $w$
represents --- both is linear and accepts the language of words
equivalent to $w$ in $S$.  In addition, every element $\R$-above
$[w]_S$ is $\R$-equivalent to an initial segment of $w$.
Similarly, if the left-handed versions of $\str^S(w)$ and $\Cay^S(w)$
are identical, then every element $\eL$-above $[w]_S$ is
$\eL$-equivalent to a final segment of $w$.
\end{lem}

Finally, we record a fact about Sch\"utzenberger graphs for later use.  Recall that an element $s$ of a semigroup $S$ is \emph{regular} if $s=sts$ for some $t\in S$.  For instance, idempotents are regular.

\begin{lem}[$\sch^S(w)$]
Suppose $S$ is an $A$-semigroup, $w$ and $x$ are words in $A^+$, and
$w$ represents a regular element of $S$.  Then $x$ is readable on
$\sch^S(w)$ if and only if $[x]_S \geq_{\mathscr J} [w]_S$.
\end{lem}

\begin{proof}
Consider the automaton $\Cay^S(w)$ and recall that $\sch^S(w)$ is the
strong component of the vertex $[w]_S$ in $\Cay^S(w)$.  If $x$ is
readable on $\sch^S(w)$, then since the graph is strongly connected,
there exist words $y, z\in A^*$ so that $yxz$ is readable as a
loop starting and ending at $[w]_S$.  But this means that $wyxz$ is
accepted by $\Cay^S(w)$ and hence $[wyxz]_S=[w]_S$.  In particular,
$[x]_S \geq_{\mathscr J} [w]_S$.  Conversely, if $[x]_S$ is $\J$-above $[w]_S$
then there exist $y,z\in A^*$ such that $[yxz]_S=[w]_S$. Suppose that $[wrw]_S=[w]_S$ with $r\in A^+$, using regularity of $[w]_S$.  Then this
means that $[wryxz]_S=[wrw]_S=[w]_S$.  Since $\Cay^S(w)$ is partial
deterministic this forces $ryxz$ to be readable as a loop in
$\sch^S(w)$ based at $[w]_S$ and, in particular, $x$ is readable in
$\sch^S(w)$, as $[wry]_S\R [w]_S$.
\end{proof}

\subsection{Expansions}\label{sec:expansions}
In this subsection we review the general notion of an expansion and
record some of their chief properties.

\begin{Def}[Expansions]\label{def:exps}
Let $\category{C}$ be a subcategory of $\SG$.  An \emph{expansion on
$\category{C}$} is a functor $\exa\colon \category{C}\to \category{C}$
together with a surjective natural transformation $\eta$ from $\exa$ to the
identity functor $1_\category{C}$ on $\category{C}$; see~\cite{Mac-CWM} for more on functors and natural transformations.  More explicitly, an expansion assigns to every
semigroup $S$ in $\category{C}$ a semigroup $S^\exa$ and a surjective morphism
$\eta_S\colon S^\exa\twoheadrightarrow S$, and to every morphism $f\colon S\to T$ in
$\category{C}$ it assigns a morphism $f^\exa\colon S^\exa\to T^\exa$ in
$\category{C}$ so that the diagram
\[\xymatrix{S^\exa\ar[r]^{f^\exa}\ar@{->>}[d]_{\eta_S}&T^\exa\ar@{->>}[d]^{\eta_T}\\ S\ar[r]^f& T}\]
commutes. The semigroups and morphisms assigned must remain in
$\category{C}$.  One must assign identity morphisms to identity morphisms and composition must be preserved.
If $S^\exa$ remains finite whenever $S$ is finite, we
say that $\exa$ \emph{preserves finiteness}.   Suppose that $\exa$ and $\exb$ are expansions.
The semigroup $(S^\exa)^\exb$ is denoted $S^{\exa.\exb}$.

Most of the expansions defined in this article are expansions on a
category of $A$-semigroups.  Notice that such an expansion is uniquely
determined by the semigroups assigned since $A$-morphisms such as
$f^\exa\colon S^\exa \to T^\exa$ and $\eta_S\colon S^\exa\twoheadrightarrow S$ are uniquely
determined by their domain and range whenever they exist.  Thus only
the existence of morphisms such as $f^\exa$ and $\eta_S$ need to be
checked: all of the remaining functoriality requirements follow automatically.  We
also encounter situations where the objects can be expanded but the
morphisms cannot.  In other words, for every $S$, there is an object
$S^\exa$ and a map $S^\exa\twoheadrightarrow S$, but for maps $S\to T$ there may or
may not exist a map $S^\exa\to T^\exa$.  In this situation we call
$S\mapsto S^\exa$ an \emph{object expansion}.
\end{Def}

The next lemma essentially shows that an expansion which preserves finiteness also
preserves the property of being finite $\J$-above.

\begin{lem}[Preserving finiteness]\label{lem:exp-fja}
If $\exa$ is an expansion on $\SGA$ which preserves finiteness and for
all semigroups $S \in \SGA$ and for all ideals $I \subseteq S$, there
is a map $(S/I)^\exa \twoheadrightarrow (S^\exa/J)$ where $J =
\eta_S^{-1}(I)$, then $\exa$ applied to a finite $\J$-above semigroup
remains finite $\J$-above.
\end{lem}

\begin{proof}
Let $S$ be a finite $\J$-above $A$-semigroup and let $\script{F}$ be a
family of co-finite ideals in $S$ whose intersection is empty
(Lemma~\ref{lem:fja}).  Consider the inverse images of these ideals
under $\eta_S\colon S^\exa\to S$.  The intersection of these ideals in
$S^\exa$ is clearly empty, and they are also co-finite since for each
$I \in \script{F}$, $S^\exa/J$ is the image of the semigroup
$(S/I)^\exa$ which is finite by hypothesis.
\end{proof}

For later use we briefly review some of the most common expansions.

\begin{Def}[Common expansions]\label{def:common-exps}
Let $S$ be an $A$-semigroup.  Roughly speaking, the \emph{reverse or right
  Rhodes expansion} remembers the exit points as the path $w$ travels
through $\Cay^S(w)$, while the \emph{reverse or right Karnofsky-Rhodes}
variation remembers the actual transition edges used (i.e., exit point,
entry point and particular edge traversed).  Note that the reverse
Karnofsky-Rhodes expansion of $S$ clearly maps onto the reverse Rhodes
expansion of $S$.  The reverse Rhodes expansion is thus a fairly small
expansion in which $\Cay^T(w)$ is quasilinear for all $w\in A^+$, and
the \emph{reverse Karnofsky-Rhodes expansion}, is the smallest
expansion such that $\Cay^T(w)$ is linear for all $w\in A^+$.
We denote the reverse Rhodes expansion by $\rRh$ and the reverse
Karnofsky-Rhodes by $\rKR$.

The original versions of these two expansions were left-handed
versions which we denote $\Rh$ and $\KR$.  In other words, the
statements above hold when we define left-handed versions of
$\str^S(w)$ and $\Cay^S(w)$ using the left Cayley graph and we read
words from right-to-left instead of left-to-right. The
\emph{Birget-Rhodes expansion}, denoted $\BR$, is the limit
$A$-semigroup obtained by iterating the Rhodes and the reverse Rhodes
expansions until the result stabilizes.  All five of these expansions
preserve finiteness.  See~\cite{BR--exp} for a precise definition of
the Birget-Rhodes expansion, and~\cite{Eilenberg},~\cite{Elston} or~\cite{RhodesWeil} for precise definitions of the Rhodes expansion, the
Karnofsky-Rhodes expansion, and their reverses.
\end{Def}

For the reader's convenience we include here the definition of the right Karnofsky-Rhodes expansion (cut to the generating set $A$) so that the reader gets the idea.

\begin{Def}[Karnofsky-Rhodes expansion]
If $S$ is an $A$-semigroup, define a congruence on $A^+$ by putting $u\equiv v$ if:
\begin{enumerate}
\item $[u]_S=[v]_S$;
\item The paths read by $u$ and $v$ from $I$ to $[u]_S=[v]_S$ in $\Cay(S,A)$ use the same transition edges.
\end{enumerate}
The quotient semigroup $A^+/{\equiv}$ is denoted $S^{\rKR}$.  It is immediate from (1) that there is a surjective morphism $\eta_S\colon S^{\rKR}\to S$.  If $f\colon S\to T$ is a homomorphism of $A$-semigroups, then $f(s)\xrightarrow{a}f(s)a$ a transition edge of $\Cay(T,A)$ forces $s\xrightarrow{a}sa$ to be a transition edge of $\Cay(S,A)$.  It follows that there exists a map $f^{\rKR}\colon S^{\rKR}\to T^{\rKR}$ and so we have defined an expansion on the category of $A$-semigroup.  It is easy to see from (1) and (2) that $\rKR$ preserves finiteness.
\end{Def}

\begin{Rmk}
It is known that the projection  $\eta_S\colon S^{\rKR}\to S$ is universal amongst maps from $A$-semigroups to $S$ with derived semigroupoid dividing a locally trivial category~\cite{Elston,RhodesStein}.
\end{Rmk}

When $S=S^\ex$ we say that $S$ is \emph{stable under the expansion
  $\ex$}.  The following result is an immediate consequence of
Definition~\ref{def:common-exps}.

\begin{lem}\label{lem:KR-closed}
Let $S$ be a finite $\J$-above $A$-semigroup.  If $S$ is stable under
the reverse Rhodes expansion $\rRh$, then $\str^S(w)$ and $\Cay^S(w)$
have the same vertex set for all words $w\in A^+$ and if $S$ is stable
under the reverse Karnofsky-Rhodes expansion, $\rKR$, then $\str^S(w)
= \Cay^S(w)$ for all words $w\in A^+$ and all of the properties listed
in Lemma~\ref{lem:str=cay} must hold.
Similar results hold for the left-handed versions when $S$ is stable
under $\Rh$ or $\KR$.
\end{lem}
\begin{proof}
We just handle the case of the reverse Karnosfky-Rhodes expansion as the other case is similar.   Since $\str^S(w)$ is always a subautomaton of $\Cay^S(w)$, it suffices to show the reverse inclusion.  To do this, it suffices by the definitions to show that each $\R$-class of elements $\R$-above $[w]_S$ contains an initial segment of $w$.  Suppose that $[u]_S\geq_{\R}[w]_S$.  Then $[w]_S=[uv]_S$ for some $v\in A^*$.  We can factor $u=u'u''$ so that reading $u''$ from $[u']_S$ stays in a Sch\"utzenberger graph, i.e., the last letter of $u'$ labels the last transition edge $e$ read by $u$ from $I$ in $\Cay(S,A)$.  The since $S=S^{\rKR}$, it follows that $w$ uses the same transition edges as $uv$ and so, in particular, uses $e$.  Thus $w=w'w''$ where the last edge read by $w'$ from $I$ is $e$.  Then $[w']_S=[u']_S\R [u'u'']_S=[u]_S$.  This establishes that each $\R$-class of elements $\R$-above $[w]_S$ contains an initial segment of $w$, as was required.
\end{proof}

For expansions closed under iteration, there is a general stability
result.

\begin{lem}[Stability]\label{lem:stable}
Let $\category{C}$ be a subcategory of $\SGA$ and let $\exa$ and
$\exb$ be expansions on $\category{C}$.  If, for every $A$-semigroup
$S$ in $\category{C}$, $S^{\exa.\exa} = S^\exa$ and there exists a map
$S^\exa \to S^\exb$, then for every $S$ in $\category{C}$, $S^\exa =
S^{\exa.\exb} = S^{\exb.\exa}$.  In other words, $S^\exa$ is stable
under $\exb$.  Conversely, if $S^\exa = S^{\exa.\exb}$ or $S^\exa =
S^{\exb.\exa}$, then there is a map from $S^\exa\to S^\exb$.
\end{lem}

\begin{proof}
If we apply the $\exa$-expansion to the maps $S^\exa\to S^\exb\to S$
we get maps in each direction between $S^\exa$ and $S^{\exb.\exa}$ as
a result of the hypothesis that $S^{\exa.\exa} = S^\exa$.  Within the
category of $A$-semigroups this implies they are isomorphic.
Similarly there is a map $S^{\exa.\exb}\to S^\exa$ since $\exb$ is an
expansion and a map $S^{\exa.\exa}\to S^{\exa.\exb}$ since the
hypothesis of the theorem can be applied to the $S^\exa$ instead of
$S$.  Once again, maps in both directions implies they are isomorphic.
The final assertion is easy since $S^{\exa.\exb}\to S^\exb$ is the
$\exb$-expansion applied to map $S^\exa\to S$ and $S^{\exb.\exa}\to
S^\exb$ is the $\exa$-expansion of semigroup $S^\exb$ projected back
to $S^\exb$.
\end{proof}

We conclude this section by reviewing the notion of a relational
morphism, its relationship with expansions, and a remark on the most
common ways to construct expansions.

\begin{Def}[Relational morphisms]
Let $S$ and $T$ be semigroups and recall that a \emph{relational
  morphism from $S$ to $T$} is a subsemigroup \mbox{$R \subseteq S\times T$}
such that the natural projection from $R$ to $S$ is onto.  If
$\p\colon R\to S$ and $\psi\colon R\to T$ denote the natural projections
restricted to $R$, then the ``relation'' alluded to in the name
relational morphism is the relation $\psi \circ \p^{-1}$.  A
relation-based definition of a relational morphism also exists, but
the equivalent definition given above is easier to work with in
practice.  We thus refer to the subsemigroup $R \subseteq S\times T$ as
the relational morphism from $S$ to $T$ even though technically, it
merely encodes the corresponding relation.

In the category of $A$-semigroups, the concept is further simplified
since there is a \emph{canonical relational morphism} between any two $A$-semigroups, which automatically projects onto both the domain and the range.  It is defined as follows. Let $S$ and $T$ be $A$-semigroups with functions
$f\colon A\to S$ and $g\colon A\to T$.  Their \emph{product} in the category $\SGA$ (sometimes called their \emph{product cut-to-generators}) is
the subsemigroup of $S\times T$ generated by the
ordered pairs $(f(a),g(a))$, $a \in A$.  Call this $A$-subsemigroup
$R$ and let $\p\colon R\twoheadrightarrow S$ and $\psi\colon R\twoheadrightarrow T$ be the $A$-morphisms
derived from the projection maps.  The canonical relational morphism
between $S$ and $T$ is the relation $\psi \p^{-1}$. The reader should be aware that the composition of two canonical relational morphisms need not be canonical.
\end{Def}

The concept of a relational morphism is critically important in finite
semigroup theory (see, for example,~\cite{Eilenberg,Kernel,Slice,Tilson,Cats2,Pinbook}).  Moreover, given a semigroup variety
$\V$ and a finite $A$-semigroup $S$, it is also important to have
detailed information about the relational morphisms from $S$ to the
$A$-semigroups $V\in \V$, particularly when $V \in \V$ is itself
finite.

\begin{Rmk}[Relational morphisms and expansions]
The notion of a relational morphism is very closely tied to that of an
expansion and the problem of constructing relational morphisms can
often be reduced to the problem of constructing finite expansions of
finite $A$-semigroups.  Let $\exa$ be an expansion on $\FSA$.
Starting from an $A$-semigroup $S$ and a variety $\V$, the semigroup
$S^\exa$ can be sent to its maximal image $V\in\V$ (which exists since
$\V$ is a variety).  Moreover, since $S^\exa$ is finite, $V$ is
finite.  Then, \[\xymatrix{S^\exa\ar@{>>}[d]\ar[r]&V\\ S&}\]
is a relational morphism.  Conversely, given a systematic method for constructing
relational morphisms
 \[\xymatrix{R\ar@{>>}[d]\ar[r]&V\\ S&}\]
 where $S$ is an arbitrary
finite semigroup, $V$ is a finite member of $\V$ (depending on $S$),
and all semigroups and morphisms are $A$-semigroups and $A$-morphisms,
the assignment $S \mapsto R$ is often an expansion on $\SGA$.  See
\cite{BR--exp} and~\cite{Eilenberg} for a more detailed discussion.
\end{Rmk}

\begin{Rmk}[Constructing expansions]
Four of the most common ways to construct expansions on $\FSA$ are:
\begin{enumerate}
\item Ramsey Theory;
\item Zimin words and uniformly recurrent sequences;
\item Semidirect product expansions;
\item Mal'cev expansions.
\end{enumerate}

The first method was applied brilliantly by Ash in~\cite{Ash} to
solve the Type-II Conjecture (see~\cite{GrRoSp80} for a good general
reference on Ramsey theory and see~\cite{ash1,Ash,BMR,HMPR,flows} for other applications to
semigroup theory).  The second method relies heavily on classical
universal algebra and semigroup theory, mixed with combinatorics
(the article~\cite{KhSa95} contains an excellent survey of this approach). The third method, which includes the Karnofsky-Rhodes expansion, is exposited in~\cite{Elston} (see also~\cite{RhodesStein}). The
fourth method is the one used in this paper and is described in the
next section; see also~\cite{Elston,RhodesStein}.

The last two methods of constructing expansions are encompassed by the notion of relatively free relational morphisms with respect to a variety of relational morphisms~\cite[Chapter 3]{qtheor}.
\end{Rmk}

\subsection{Mal'cev expansions}\label{sec:malcev}
In this subsection we define the Mal'cev expansion of a semigroup with
respect to a variety.  The results in this section are in general
well-known but they provide a language and a context for the newer
results contained in the later sections.

\begin{Def}[Mal'cev kernels]\label{def:m-ker}
The \emph{Mal'cev kernel} of a semigroup morphism $\p\colon S\to T$ is the
collection of inverse images of idempotents. In other words, it is the
set of subsemigroups $\{\p^{-1}(e)\}$ where $e$ is an idempotent in
$T$.  If each semigroup in the Mal'cev kernel is contained in a
variety $\V$ we say the \emph{Mal'cev kernel lies in $\V$}.
\end{Def}

The main result we need involving Mal'cev kernels is Brown's theorem.
Recall that a variety is called \emph{locally finite} if its finitely
generated free objects remain finite (or equivalently its finitely generated members are all finite).

\begin{thm}[Brown~\cite{Brown}]\label{thm:brown}
Let $\p\colon S\to T$ be a map between two $A$-semigroups where $A$ is
finite.  If $T$ is finite and the Mal'cev kernel of $\p$ lies in a
locally finite variety $\V$, then $S$ is also finite.
\end{thm}

Algebraic proofs of Brown's theorem can be found in~\cite{idempotentstabilizers} and~\cite[Chapter 4]{qtheor}; see also~\cite{SimBr2}.  Mal'cev expansions were considered in~\cite{Elston} (a profinite analogue for pseudovarieties was introduced in~\cite{RhodesStein}).

\begin{Def}[Mal'cev expansions]\label{def:malcev}
Let $S$ be an $A$-semigroup and let $\V$ be a variety. The
\emph{Mal'cev expansion of $S$ by $\V$} is the largest $A$-semigroup
which maps to $S$ with Mal'cev kernel in $\V$.  More precisely, for
each $A$-morphism $T\to S$ whose Mal'cev kernel lies in $\V$, define
$\sim_T$ as the congruence on the free semigroup $A^+$ which produces
$T$, i.e., the congruence which corresponds to the $A$-morphism $A^+
\to T$.  The Mal'cev expansion of $S$ by $\V$ is then defined as the
semigroup $A^+/{\sim}$ where $\sim$ is the intersection of all of these
congruences $\sim_T$ on $A^+$.  The Mal'cev expansion of $S$ by $\V$
will be denoted $S^\V$.   See~\cite{BR--exp,Elston} for a more detailed
discussion of Mal'cev expansions.
\end{Def}

Some properties of Mal'cev expansions are immediate from the
definition.

\begin{lem}\label{lem:malcev}
Let $S$ be an $A$-semigroup and let $\V$ be a variety.  The semigroup
$S^\V$ is the largest $A$-semigroup which maps to $S$ with Mal'cev
kernel in $\V$. More precisely the canonical $A$-morphism $\eta_S\colon S^\V \to
S$ has Mal'cev kernel lying in $\V$ and if $\psi\colon T\to S$ is another such
$A$-morphism, then there is a (unique) $A$-morphism $\Psi\colon S^\V\to T$ so that
\[\xymatrix{S^\V\ar[rr]^{\Psi}\ar@{->>}[rd]_{\eta_{S}}&&T\ar[ld]^{\psi}\\ &S&}\]
commutes.
\end{lem}

\begin{Rmk}[A presentation of $S^\V$]\label{rmk:apresentation}
One can alternatively construct $S^\V$ as follows.  Let $E$ be a basis of identities for the variety $\V$.  Define an $A$-semigroup $S'$ with presentation consisting of all relations of the form $u(w_1,\ldots,w_n)=v(w_1,\ldots, w_n)$ where $u=v$ is an identity from $E$ in $n$-variables and $w_1,\ldots, w_n\in A^+$ all map to the same idempotent of $S$.  By construction, $S'\to S$ is a well-defined $A$-morphism with Mal'cev kernel in $\V$.   To verify the universal property of Lemma~\ref{lem:malcev}, notice that if $T$ is any $A$-semigroup with $T\to S$ having Mal'cev kernel in $\V$ and $u,v,w_1,\ldots,w_n$ are as above, then since $[w_1]_T,\ldots, [w_n]_T$ generate a semigroup in $\V$ they satisfy $[u(w_1,\ldots,w_n)]_T=[v(w_1,\ldots, w_n)]_T$ and so there is a morphism $S'\to T$ of $A$-semigroups.  Thus $S'=S^\V$.

Notice that the functoriality of the assignment $S\mapsto S^\V$ is clear from this presentation.  Indeed, if $f\colon S\to T$ is a homomorphism of $A$-semigroups and $w_1,\ldots, w_n\in A^+$ map to the same idempotent of $S$, then also $w_1,\ldots,w_n$ map to the same idempotent of $T$.  Thus $T^\V$ satisfies all the defining relations of $S^\V$ and so there is a morphism $f^{\V}\colon S^\V\to T^\V$.
\end{Rmk}

One of the main goals of geometric semigroup theory is to determine
the structure of $S^{\V}$ for various locally finite varieties $\V$.
Note, however, that even taking a Mal'cev expansion with respect to the trivial
variety is a very non-trivial operation.  If $S$ is a finite semigroup
and $A$ is the set of all elements in $S$, then Chris Ash studied, in
essence, the Mal'cev expansion $S^\Triv$ with respect to the
generating set $A$ in his solution to the Type II conjecture.  See
\cite{ash1,Ash,BMR} for details and
\cite{BR--exp} for an early version of this approach.

The following theorem establishes the key properties of Mal'cev
expansions, namely that they are indeed expansion in the sense of
Definition~\ref{def:exps}, and that when the variety $\V$ is locally
finite, they preserve finiteness properties.

\begin{thm}\label{thm:malcev}
For each variety $\V$, the assignment $S \mapsto S^\V$ defines an
expansion on the category $\SGA$.  Moreover, if $\V$ is locally
finite, then it preserves finiteness.  In particular, it restricts to
an expansion on $\FSA$ or $\FJA$.
\end{thm}

\begin{proof}
Remark~\ref{rmk:apresentation} shows that $S\mapsto S^\V$ is an expansion.

To show the second assertion we need to show that if $S$ is finite,
then $S^\V$ is finite, and if $S$ is finite $\J$-above, then $S^\V$ is
finite $\J$-above.  When $S$ is finite, the assertion follows
immediately from Brown's theorem (Theorem~\ref{thm:brown}). When $S$
is finite $\J$-above we will use Lemma~\ref{lem:exp-fja}.  Let $I$ be
an ideal in $S$ and let $J$ be the inverse image of $I$ under the
$A$-morphism $S^\V\to S$.  Notice that the Mal'cev kernel of $S^\V/J
\to S/I$ is a subfamily of the Mal'cev kernel of $S^\V\to S$ with the
possible addition of the trivial semigroup, and thus lies in $\V$.  It
now follows from Lemma~\ref{lem:malcev} that there is an $A$-morphism
from $(S/I)^\V$ to $S^\V/J$ and by Lemma~\ref{lem:exp-fja} the proof
is complete.
\end{proof}

A list of some locally finite varieties is given in
Table~\ref{tab:locally-finite}.

\begin{table}[thbp]
\begin{tabular}{|c|c|c|}
\hline {\bf Notation} & {\bf Description} & {\bf Equations} \\
\hline  \Triv & trivial & $x=1$\\
\hline \SL & semilattices & $xy=yx, x^2=x$ \\
\hline \RZ & right-zero semigroups & $xy=y$ \\
\hline \LZ & left-zero semigroups & $xy=x$ \\
\hline \Ba & bands & $x^2=x$ \\
\hline \RB & rectangular bands & $x^2=x, xzy=xwy$ \\
\hline \LC & left constants & $xy=xz$\\
\hline \RC & right constants & $yx=zx$ \\
\hline \Co & $2$-sided constants & $xzy=xwy$ \\
\hline \Dk & $k$-delay & $x_1\cdots x_k = x_0x_1\cdots x_k$ \\
\hline \Null & null semigroups & $xy=0$\\
\hline \Nilk & nilpotent of class $k$ & $x_1\cdots x_k = 0$ \\
\hline \Commn &
  \begin{tabular}{l}
  commutative semigroups\\
  satisfying $x^m=x^{m+n}$
  \end{tabular} &
  $xy=yx, x^m=x^{m+n}$   \\
\hline $\langle S \rangle$ &
  \begin{tabular}{l}
  variety generated by a\\
  finite semigroup $S$
  \end{tabular} &
  $\textrm{Eqs}(S)$ \\
\hline \Zp &
  \begin{tabular}{l}
  vector spaces over $\mathbb{Z}_p$,\\
  $p$ prime
  \end{tabular} &
  \begin{tabular}{l}$xy=yx, x^p=1$
  \end{tabular} \\
\hline
\end{tabular}
\vspace*{1em}
\caption{Locally finite varieties\label{tab:locally-finite}}
\end{table}

The lattice of locally finite varieties is closed under taking joins, subvarieties and semidirect products.
The concept of a Mal'cev product can also be extended to
form a (nonassociative) product on the set of varieties.

\begin{Def}[Mal'cev products of varieties]
If $\U$ and $\V$ are varieties of semigroups, the \emph{Mal'cev
  product} of $\U$ and $\V$ (denoted $\U \malce  \V$) is the variety
generated by all semigroups $S$ admitting a homomorphism $\p\colon S\to V$ with $V\in \V$ so that the Mal'cev kernel of $\p$ belongs to $\U$.  It follows from Brown's theorem that the Mal'cev product of two locally
finite varieties is also a locally finite variety, and, as a result,
this table can be extended by taking Mal'cev products.  Some care must
be exercised with parentheses, since Mal'cev products are not
necessarily associative.  The inclusion \[\mathcal U\malce (\V\malce \mathcal W)\subseteq (\mathcal U\malce \mathcal V)\malce \mathcal W\] always
holds, but a strict inclusion is possible.
\end{Def}

Numerous relations exist between the locally finite varieties listed
in the table with respect to the Mal'cev product. For example  $\Ba = \Ba
\malce  \Ba$, as is immediate from the definition.  It is a well-known consequence of Green-Rees structure theory that $\Ba=\RB\malce \SL$.    In
contrast, if we replace $\Ba$ by $\bar\Ba=\script{C} \malce  \Ba$,
then finite iterated Mal'cev products of $\bar\Ba$, bracketed in the larger way,
contain all finite aperiodic semigroups by the two-sided Prime Decomposition
Theorem~\cite{Kernel,qtheor}.

As an illustration of the power of Lemma~\ref{lem:malcev} we prove that
the Mal'cev expansion with respect to the variety of rectangular bands
applied to the free semilattice over $A$ is the free band over $A$.
Actually, with additional work, one could prove that the Birget-Rhodes
expansion of the free semilattice is already equal to free band, but
that fact is harder to derive using these techniques.

\begin{lem}\label{lem:fb-sl}
Let $\V$ and $\mathcal W$ be varieties.  Then the free semigroup on $A$ in $\V\malce \mathcal W$ is $F_{\mathcal W}(A)^{\V}$ where $F_{\mathcal W}(A)$ is the free semigroup on $A$ in $\mathcal W$.  In particular, the rectangular band expansion of the free semilattice is the free band.
\end{lem}
\begin{proof}
Suppose that $\sigma\colon A\to T$ is a map with $T\in \V\malce \mathcal W$.  Then we can write $T$ as a quotient $\pi\colon S\to T$ where $S$ admits a morphism $\theta\colon S\to W$ with the Mal'cev kernel of $\theta$ in $\V$.  Let $\tau\colon A\to S$ be a map with $\sigma=\pi\tau$.  Clearly to extend $\sigma$ to $F_{\mathcal W}(A)^{\V}$ it suffices to extend $\tau$.  Without loss of generality we may assume that $\theta$ is onto and that $\tau(A)$ generates $S$.  Then there is an $A$-morphism $W^{\V}\to S$ by the universal property in Lemma~\ref{lem:malcev}.  But there is an $A$-morphism $F_{\mathcal W}(A)\to W$ since $W\in \mathcal W$, and so by functoriality we have an $A$-morphism $F_{\mathcal W}(A)^{\V}\to W^{\V}\to S$, completing the proof.
\end{proof}

\subsection{Rectangular bands}\label{sec:bands}
In this subsection we investigate the Mal'cev expansion with respect to
the variety of rectangular bands and the closely related varieties of
left zero/right zero semigroups.  Let $S$ be an $A$-semigroup and
consider the semigroups $S^\RB$, $S^\LZ$, and $S^\RZ$.  Since left and
right zero semigroups are examples of rectangular bands, it is
immediate from the definition that there are maps $S^\RB\to S^\LZ$ and
$S^\RB\to S^\RZ$.  We begin by establishing properties of the
projection $S^\RB\to S$.  Recall that if $\mathscr K$ is any of Green's relations, then a morphism $\p\colon S\to T$ is a \emph{$\mathscr K'$-map} if $s,s'\in S$ regular and $\p(s)=\p(s')$ implies that $s\mathrel{\mathscr K} s'$.  See~\cite[Chapter 8]{Arbib} or~\cite[Chapter 4]{qtheor}.



\begin{lem}\label{lem:rb-map}
If $S$ is an $A$-semigroup then the map $\p\colon S^\RB\to S$ is
one-to-one on each subgroup of $S^\RB$.  Moreover, if $S$ is finite $\J$-above, then the inverse image of each
regular $\J$-class of $S$ contains a single regular $\J$-class in
$S^\RB$.  In particular, $\p$ is a $\J'$-map.
\end{lem}
\begin{proof}
Let $G$ be a subgroup of $S^\RB$ and let $e$ be its identity.  Since
$\p$ restricted to $G$ is a group homomorphism, it is injective on
$G$ if and only if $\ker \p|_G=\{e\}$.  On the
other hand, $e$ and $\p(e)$ are idempotents, so the inverse image
$\p^{-1}(\p(e))$ is a rectangular band by definition.  Since the
only idempotent in a group is its identity, we can conclude the
restriction is indeed one-to-one.
Next, let $J$ be a regular $\J$-class of $S$ and let $x$ and $y$ be
idempotents of $\p^{-1}(J)$.  Since $\p(x)$ and $\p(y)$ are
$\J$-equivalent (and $\p$ is onto) there exists $u$ and $v$ in
$S^\RB$ such that $\p(uxv)=\p(y)=e$, where $e$ is an idempotent of
$J$.  On the other hand, $\p^{-1}(e)$ is a rectangular band by
definition, so $uxv$ and $y$ are $\J$-equivalent in $S^\RB$.  Thus
$x\geq_{\J} y$ and a similar argument shows $y\geq_{\J} x$.  It
follows that $\p^{-1}(J)$ contains only a single regular
${\J}$-class.
\end{proof}

\begin{Rmk}
One can similarly show that $S^{\LZ}\to S$ is an $\eL'$-map and $S^{\RZ}\to S$ is an
$\R'$-map~\cite[Chapter 4]{qtheor}.
\end{Rmk}

Because $S^\RB\to S$ is a $\J'$-map which is injective on subgroups, when $S$ is finite the techniques and results from~\cite{RhodesWeil,qtheor} can be applied. In particular, $S^\RB$ and $S$ have the same complexity~\cite{Eilenberg,qtheor}.  In the
next lemma, we show that all three expansions are stable under
iteration.


\begin{lem}[Stable under iteration]\label{lem:iteration}
A map $\theta\colon S\to T$ has Mal'cev kernel in $\LZ$ ($\RZ$) if and only if the
inverse image of each left zero (right zero) subsemigroup in $T$ is a left zero (right zero)
subsemigroup in $S$.  Similarly, $\theta$  has Mal'cev kernel in
$\RB$ if and only if the inverse image of each rectangular band in $T$
is a rectangular band in $S$.  As a consequence, $S^{\LZ.\LZ} = S^\LZ$,
$S^{\RZ.\RZ} = S^\RZ$, and $S^{\RB.\RB} = S^\RB$.
\end{lem}
\begin{proof}
We just handle the case of rectangular bands, as the other cases are similar (and this result is any event well known).  Clearly, it suffices to handle the case of a surjective morphism $\theta\colon S\to T$ of bands with $T$ a rectangular band and $\theta$ a morphism with Mal'cev kernel in $\RB$.  By Lemma~\ref{lem:rb-map}, $S$ has a single $\J$-class.  It follows that $S$ is a rectangular band.
Finally, to see the last assertion, let $\exa$ denote either $\LZ$,
$\RZ$, or $\RB$ and consider the composition of maps $S^{\exa.\exa}\to
S^\exa\to S$.  An idempotent in $S$ pulls back to a left zero
semigroup/right zero semigroup/rectangular band in $S^\exa$ which
we now know pulls back to a semigroup of the same type in
$S^{\exa.\exa}$.  Thus, by Lemma~\ref{lem:malcev} there exists a map
$S^\exa\to S^{\exa.\exa}$. As usual, maps in both directions implies
they are isomorphic.
\end{proof}

One consequence of being stable under the $\RB$-expansion is that the
semigroup remembers the first and last letters of the words
representing an element in the following sense.

\begin{lem}[Vertex labels]\label{lem:rb-labels}
If $S$ is a non-trivial finite $\J$-above $A$-se\-mi\-group which is
stable under the $\RB$-expansion and $u$ and $v$ are words in $A^+$
which represent the same element of $S$ then they have the same first
letter and the same last letter.  As a consequence, there is a unique
edge starting at the initial state of $\str^S(v)$ and for every vertex
$v$ in $\Cay(S,A)$, all of the edges ending at $v$ have identical
labels.
\end{lem}

\begin{proof}
Let $S\to \{1\}$ be the collapsing morphism.  Viewing $\{1\}$ as an $A$-generated semigroup via this map, functoriality yields a surjective morphism $S^\RB\to \{1\}^{\RB}$.  Since $\{1\}^{\RB}$ is the free rectangular band on $A$ by Lemma~\ref{lem:fb-sl}, and hence two words are equal in it if and only if they have the same first and last letter,  the result follows.
\end{proof}

A similar statement and proof shows that $S^\LZ$ remembers the
first letter and $S^\RZ$ remembers the last letter.  The $\LZ$, $\RZ$ and $\RB$-expansions are also
stable under many of the expansions defined in
Definition~\ref{def:common-exps}.

\begin{figure}[htbp]
\[\xymatrix{
S^\LZ\ar[d] & S^\RB\ar[l]\ar[r]\ar[dd] & S^\RZ\ar[d]\\
S^\rKR\ar[d]      &                           & S^\KR\ar[d]\\
S^\rRh & S^\BR\ar[l]\ar[r] & S^\Rh}\]
\caption{Relations between various expansions.\label{fig:expansions}}
\end{figure}

\begin{thm}[Stable under other expansions]\label{thm:stable}
For each $A$-semigroup $S$, there exist $A$-morphisms as shown in
Figure~\ref{fig:expansions}.  As a consequence, $S^\LZ$ is stable
under $\rKR$ and $\rRh$, $S^\RZ$ is stable under $\KR$ and $\Rh$, and
$S^\RB$ is stable under all of the expansions shown.
\end{thm}

\begin{proof}
The maps that are not obvious are consequences of well-known properties of the Rhodes, Karnofsky-Rhodes and Birget-Rhodes expansions~\cite{BR--exp,Elston,RhodesStein}.  For instance, we verify the projection $\eta_S\colon S^{\rKR}\to S$ has Mal'cev kernel in $\LZ$.  Suppose $v,w\in A^+$ map to an idempotent $e$ in $S$.  Then since $[vw]_S=e^2=e=[v]_S$, it follows that $v$ and $vw$ use the same transition edges from $I$ in $\Cay(S,A)$ in the paths they read to $e$.  Thus $vw=v$ in $S^{\rKR}$.  It follows that $\eta_S\inv(e)$ is a left zero semigroup.

By Lemma~\ref{lem:iteration} the expansions $\LZ$, $\RZ$ and $\RB$ are
closed under iteration, so the stability properties of $S^\LZ$ and
$S^\RZ$ follow by Lemma~\ref{lem:stable}.  Similarly, $S^\RB$ is stable
under all of these with the possible exception of $\BR$.  However,
once $S^\RB$ is stable under $\Rh$ and $\rRh$, it is also stable under
$\BR$ and by Lemma~\ref{lem:stable}, there is a map $S^\RB\to S^\BR$.
\end{proof}

Combining Lemma~\ref{lem:KR-closed} and Lemma~\ref{lem:iteration}
proves the following.

\begin{cor}\label{cor:rb-str}
If $S$ is a finite $\J$-above $A$-semigroup and $T$ is either $S^\LZ$
or $S^\RB$, then $\str^T(w) = \Cay^T(w)$ is a trim, deterministic
finite state acceptor which depends only on the element $w$
represents, it is linear and it accepts the language of words
equivalent to $w$ in $T$.  In addition, every element $\R$-above
$[w]_T$ is $\R$-equivalent to an initial segment of $w$.
Similar properties hold for the left-handed versions when $T$ is
either $S^\RZ$ or $S^\RB$.
\end{cor}


%
%
%
%

We thus have
the following result, which shows that we can use straightline automata
(and in particular the factors of its base) to investigate Green's
relations in $S^\RB$.  This is a first small step toward converting
algebraic information into geometric form.

\begin{cor}[Green's relations]\label{cor:subwords}
Let $S$ be an $A$-semigroup, let $T=S^\RB$ and let $u$ and $v$ be
words in $A^+$.  Then $[u]_T \geq_{\mathscr J} [v]_T$ if and only if $u$ is
readable in $\str^T(v)$.
Moreover, $[u]_T \geq_{\mathscr R}
[v]_T$ if and only if $u$ is readable in $\str^T(v)$ starting at the
initial state if and only if $[u]_T$ is $\R$-equivalent to an element
represented by an initial segment of $v$, and $[u]_T \geq_{\mathscr L} [v]_T$
if and only if $u$ is readable in $\str^T(v)$ ending at the terminal state
if and only if $[u]_T$ is $\eL$-equivalent to an element represented
by a final segment of $v$.
\end{cor}


Another immediate corollary of these results is that loops in
straightline automata can be ``read backwards'' in the following
sense.

\begin{cor}[Reading backwards]\label{cor:read-back}
Let $S$ be an $A$-semigroup stable under the $\RB$-expansion and let
$w$ and $u$ be words in $A^+$.  If $wu$ and $w$ represent the same
element of $S$, then $w$ is readable in $\str^S(w)$ starting at $I$
and $u$ is readable as a loop starting and ending at the vertex
$[w]_S$.  Moreover, there is a partition of $w = w_1 w_2$, with $w_i
\in A^*$ such that $w_2$ is $\eL$-equivalent to $u$.  In particular,
there exists words $v$ and $v'$ such that $[v'u]_S = [w_2]_S$ and
$[vw_2]_S = [u]_S$ and the entire configuration of paths shown in
Figure~\ref{fig:loop} can be read in $\str^S(w)$ starting at $I$.
\end{cor}
\begin{proof}
The first statement is clear.  By Corollary~\ref{cor:subwords}, since $[u]_S\geq_{\eL} [w]_S$, we can write $w=w_1w_2$ with $[u]_S\eL [w_2]_S$.  The remaining statements of the corollary are straightforward.
\end{proof}

\begin{figure}[ht]
\begin{tabular}{cc}
\begin{tabular}{c}
\psfrag{1}{$1$}
\psfrag{alpha}{$w$}
\psfrag{beta}{$u$}
\includegraphics[scale=.8]{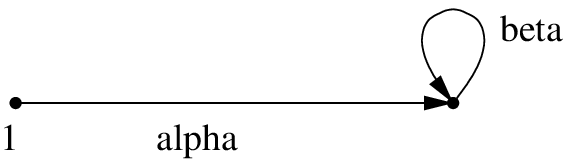}
\end{tabular}
&
\begin{tabular}{c}
\psfrag{alpha1}{$w_1$}
\psfrag{alpha2}{$w_2$}
\psfrag{beta}{$u$}
\psfrag{x}{$v$}
\psfrag{y}{$v'$}
\psfrag{1}{$1$}
\includegraphics[scale=.8]{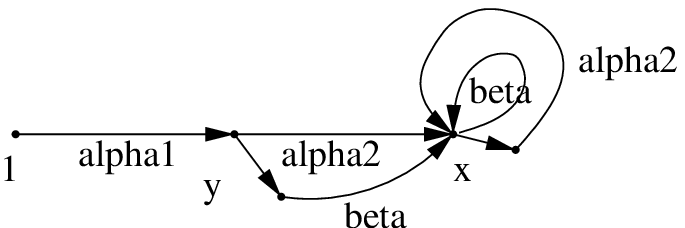}
\end{tabular}
\end{tabular}
\caption{A standard loop before and after factoring.\label{fig:loop}}
\end{figure}

Corollary~\ref{cor:read-back} takes on added significance once we
prove that $S^\RB$ can be presented using only relations of the form
$uv=u$.

\begin{Def}[Loop relations]\label{def:loops}
A \emph{loop relation} is one of the form $uv = u$ for words $u,v \in
A^*$.  An $A$-semigroup has a \emph{loop presentation} if it has
presentation in which each relation is a loop relation.  When focusing
on the $\eL$-order instead of the $\R$-order, we consider
\emph{reverse loop relations} $uv=v$, which lead to \emph{reverse loop
presentations}.  Notice that Proposition~\ref{relationsforcayley} immediately implies that the Cayley graph of a semigroup with a loop presentation is a loop automaton.
\end{Def}

\begin{lem}\label{lem:rb-loop-pres}
For each $A$-semigroup $S$, the semigroup $S^\RB$ can be given a loop
presentation.  In particular, $S^\RB$ can be defined by adding the
relation $uv=u$ whenever the words $u,v \in A^+$ represent the same
idempotent $e$ in $S^\RB$.  The analogous reverse loop presentation is
another presentation of $S^\RB$.
\end{lem}
\begin{proof}
This is immediate from Remark~\ref{rmk:apresentation}.
\end{proof}

The fact that $S^\RB$ has a loop presentation has important
consequences for its straightline automata.  In particular, they are
all loop automata.

\begin{lem}\label{lem:loop-pres->aut}
If $S$ is a finite $\J$-above $A$-semigroup with a loop presentation,
then, for all words $w\in A^+$, $\Cay^S(w)$ is a loop automaton.
\end{lem}
\begin{proof}
We already observed that $\Cay(S,A)$ is a loop automaton.  Clearly a loop presentation of $\Cay^S(w)$ can be obtained by using the loop relations from the presentation of $\Cay(S,A)$ and by adding the dead-end relations $ua=\square$ where $[u]_S\geq_{\R}[w]_S$, $a\in A$ and $[ua]_S\ngeq_{\R}[w]_S$.
\end{proof}

Actually, more is true.  For each specific word $w$ only a finite
number of the loop relations are needed in order to build its
straightline automaton.

\begin{lem}\label{lem:rb-loop-aut}
If $T$ is a finite $\J$-above $A$-semigroup with  $T=T^\RB$, and $w$ is a
word in $A^+$, then the automaton $\str^T(w)$ is a loop automaton
defined by a finite number of loop relations.
\end{lem}

\begin{proof}
By Corollary~\ref{cor:rb-str}, Lemma~\ref{lem:rb-loop-pres} and Lemma~\ref{lem:loop-pres->aut},
$\str^T(w)$ is a loop automaton, but we must prove that it requires
only a finite number of loop relations.  There is a cofinite ideal $I$ of $T$ so that $\str^T(w)=\str^{T/I}(w)$ so we may assume without loss of generality that $T$ is finite.

By Ramsey theory~\cite{GrRoSp80}, if a word $W\in A^+$ is sufficiently long relative to
the size of $T$, then $W$ can be partitioned into factors so that $W
= w_1 u v w_2$ with $w_i\in A^*$ and where $u$, $v$ and $uv$ all represent
the same element of $T$ (which is necessarily an idempotent).  In the
notation from~\cite{GrRoSp80} we need $|W| > [3]^2_m$ where $m=|T|$; call this lower bound $n$.
Indeed, if $W=a_1\cdots a_n$ with the $a_i$ letters, then consider the complete graph $\Gamma$ on $\{0,\ldots, n\}$.  We color the edges of $\Gamma$ by elements of $T$.  More precisely, if $i<j$ are elements of $\{0,\ldots,n\}$, then the edge between $i$ and $j$ is colored by $[a_{i+1}\cdots a_j]_T$.  Ramsey's theorem~\cite{GrRoSp80} provides a monochromatic triangle.  Such a triangle corresponds to $i<j<k$ with \[[a_{i+1}\cdots a_j]_T = [a_{j+1}\cdots a_k]_T = [a_{i+1}\cdots a_k]_T\] and so taking $u=a_{i+1}\cdots a_j$, $v=a_{j+1}\cdots a_k$ gives the required partition.

Thus by applying a loop relation of the form $uv=u$, with $|u|, |v|
\leq n$, and $u$ and $v$ representing the same idempotent in $T$, all words
are equivalent to words of length at most $n$.  For this finite number
of short words, a finite number of loop relations of the form
$u_iv_i=u_i$ with $u_i,v_i$ equal to the same idempotent in $T$ are needed to derive whatever further equalities are
needed.  This gives a \emph{finite} list of defining loop relations
for $\str^T(w)$.
\end{proof}

Next, we show that $S^\RB$ is almost never a monoid, but before doing
so, we note that when $S$ is the trivial semigroup generated by $A =
\{a\}$, then $S^\RB$ is still trivial.  This is because the entire
semigroup must be a rectangular band and a one-generated rectangular
band is trivial.  On the other hand, this turns out to be the only
situation where $S^\RB$ is a monoid.

\begin{lem}[Almost never a monoid]\label{lem:not-monoid}
If $S$ is an $A$-semigroup then $S^\RB$ is not a monoid unless $|A|=1$ and $S=\{1\}$.
\end{lem}
\begin{proof}
If $|A|\geq 2$, then the collapsing morphism $S\to \{1\}$ yields a surjective morphism $S^{\RB}\to\{1\}^{\RB}$ (where $\{1\}$ is viewed as $A$-generated).  Since $\{1\}^{\RB}$ is the free rectangular band on $A$, and hence not a monoid, it follows that $S^{\RB}$ is not a monoid.  So we may suppose that $|A|=1$, i.e., $S$ is cyclic.  If $S$ is not a monoid, then clearly $S^{\RB}$ is also not a monoid.  So we may assume that $S$ is a finite cyclic group of order $n\geq 2$ generated by $a$.  Let $T$ be the semigroup given by the presentation $\langle a\mid a^2=a^{2+n}\rangle$.  Then the canonical morphism $T\to S$ has trivial Mal'cev kernel.  Thus $T$ is a non-monoid homomorphic image of $S^{\RB}$ and so $S^{\RB}$ is not a monoid.
\end{proof}

As a consequence of Lemma~\ref{lem:not-monoid}, our decision to use
$S^I$ rather than only adding an identity when one is not already
present is irrelevant once $S$ is stable under the $\RB$-expansion.
%

\subsection{Improving stabilizers}\label{sec:stab}
In this short subsection we first consider the effect of the rectangular
band expansion on stabilizers, and then we show how an expansion due
to Le Saec, Pin and Weil can further simplify their structure.

\begin{Def}[Stabilizers]\label{def:stab}
Let $S$ be a semigroup.  The \emph{right stabilizer} of $s\in S$, is
the subsemigroup $S_s=\{t\mid st=s\}$.  Similarly, there are
\emph{left stabilizers} ${}_sS = \{t\mid ts=s\}$ and \emph{double
  stabilizers} \[{}_{s'}S_s = {}_{s'}S \cap S_s = \{t\mid ts=s \textrm{
  and } s't=s'\}.\]
\end{Def}

The next lemma allows us to equate right stabilizers with loops
accepted by Sch\"utzenberger graphs, thus turning the study of right
stabilizers into an investigation into the properties of various
automata.

\begin{lem}[Loops in $\sch^S(s)$]\label{lem:stab=loops}
Let $S$ be an $A$-semigroup and let $s$ be an element of $S$.  The
elements of $S$ represented by words accepted by $\sch^S(s)$ are
precisely the right stabilizers of $s$, i.e., the elements in $S_s$.
\end{lem}
\begin{proof}
If $u$ is a word representing $s$ and $v$ is a word accepted by
$\sch^S(s)$, then $uv$ is accepted by $\Cay^S(u)$ so that $uv$ and $u$
both represent $s$, whence $s\cdot [v]_S = [uv]_S = [u]_S = s$, and $[v]_S$
is a right stabilizer of $s$.  Conversely, if $t$ is a right
stabilizer of $s$ and $u$ and $v$ represent $s$ and $t$ respectively, then $[uv]_S
= [u]_S$ and $uv$ is accepted by $\Cay^S(u)$.  Moreover, since
$\Cay^S(u)$ is deterministic, $v$ is read as a loop based at $[u]_S$
in the strong component of $[u]_S$, but this is the Sch\"utzenberger
graph $\sch^S(s)$ and $v$ is reading an accepted loop.
\end{proof}

In a left cancellative semigroup, a semigroup $S$ acts on the left of its Cayley graph by injective functions.  In the next proposition, we show that if $S$ is stable under the $\RZ$-expansion, then the left action of $S$ is injective when restricted to Sch\"utzenberger graphs.  Recall the a partial order is said to be \emph{unambiguous} if $x\leq y,z$ implies $y,z$ are comparable.

\begin{thm}\label{localcancellativity}
Let $S$ be a finite $\J$-above $A$-semigroup stable under the $\RZ$-expansion.  Then the $\eL$-order on $S$ is unambiguous.    Moreover, if $r,s,t\in S$ with $r\R t$ and $sr=st$, then $r=t$.  Consequently, if $s,t\in S$, then the natural map $\sch^S(t)\to \sch^S(st)$ induced by left multiplication by $s$ is an injective morphism of $A$-automata.
\end{thm}
\begin{proof}
It is conceptually easier to establish the dual result when $S$ is stable under the $\LZ$-expansion, which we proceed to do.  We use that $S=S^{\wedge\ell\pv 1}$.  Suppose that $r,t$ have a common lower bound in the $\R$-order on $S$.  Let $u,v,w,z\in A^+$ satisfy $[u]_S=r$, $[v]_S=t$ and $[uw]_S=[vz]_S$.  We prove that $r$ and $t$ are comparable in the $\R$-order.  By hypothesis, the paths read by $uw$ and $vz$ from $I$ in $\Cay(S,A)$ use the same set $\{e_1,\ldots,e_m\}$ of transition edges.  Since $\Cay(S,A)$ is linear, we may assume the indices are chosen so that $e_1\succ e_2\cdots \succ e_m$.   Hence the set of transition edges used by $u$  in its run from $I$ is of the form $\{e_1,\ldots, e_k\}$ for some $k\leq m$ and similarly the set used by $v$ is of the form $\{e_1,\ldots, e_r\}$ for some $r\leq m$.  Without loss of generality we may assume that $k\leq r$. It then follows that $[u]_S\geq_{\R} [v]_S$.

Next suppose that $r,t\in S$ satisfy $r\eL t$ and $rs=ts$.  Then since $rs=ts$ is a common lower bound to $r$ and $t$ in the $\R$-order, they are comparable  by the previous paragraph.   It follows that $[u]_S\HH [v]_S$ in $S$.    Choose $u,v,w\in A^+$ such that $[u]_S=r$, $[v]_S=t$ and $[w]_S=s$. There are two cases.  Suppose first $[uw]_S\R [u]_S$.  Then since $[uw]_S=[vw]_S$, Green's Lemma implies $[u]_S=[v]_S$, as required.  Next suppose that $[u]_S>_{\R} [uw]_S$.  As $[uw]_S=[vw]_S$, both $uw$ and $vw$ use the same transition edge $e$ when exiting the Sch\"utzenberger graph of $[u]_S$. Also note that if $w'$ is an initial segment of $w$, then $[uw']_S\eL [vw']_S$ and hence \[[uw']_S\R [u]_S\iff [uw']_S\J [u]_S\iff [vw']_S\J [u]_S\iff [vw']_S\R [u]_S\] as $[v]_S\HH [u]_S$.  It now follows that if $w'$ is the longest initial segment of $w$ so that $[uw']_S\R [u]_S$, then it is also the longest initial segment so that $[vw']_S\R [u]_S$ and moreover $[uw']_S=[vw']_S$ is the initial vertex of the transition edge $e$ (and hence belongs to the Sch\"utz\-en\-ber\-ger graph of $[u]_S$).  Then again Green's Lemma implies $[u]_S=[v]_S$.
\end{proof}

As a corollary, we obtain a structural result due to the second author
\cite{IIM2}.

\begin{thm}[Right stabilizers]\label{thm:right-stab}
If $S$ is a finite $\J$-above $A$-semigroup stable under the $\RZ$-expansion, then each
right stabilizer $S_s$ is finite $\R$-trivial and an $\eL$-chain in $S$, but not necessarily within the
subsemigroup $S_s$.
\end{thm}
\begin{proof}
Since each element of $S_s$ is $\eL$-above $s$, it is finite by assumption on $S$ and is an $\eL$-chain in $S$ by unambiguity of the $\eL$-order (Theorem~\ref{localcancellativity}).  Moreover, if $r,t$ are right stabilizers of $s$ with $r\R t$, then $sr=s=st$ implies $r=t$ by Theorem~\ref{localcancellativity}. Thus ${}_sS$ is $\R$-trivial, completing the proof.
\end{proof}

If $S=S^{\RB}$, then $S$ has unambiguous $\eL$- and $\R$-orders, a property already acheived by $S^{\BR}$~\cite{BirgetSyn2}.

\begin{Cor}
If $S$ is a finite $\J$-above $A$-semigroup stable under the $\RB$-expansion, then the $\eL$- and $\R$-orders on $S$ are unambiguous.
\end{Cor}



Although it is not needed later, we record a two-sided version of
Theorem~\ref{thm:right-stab} (also from~\cite{IIM2}) which is
available whenever $S$ is stable under the $\RB$-expansion.

\begin{thm}[Double stabilizers]\label{thm:double-stab}
If $S$ is a finite $\J$-above $A$-semigroup stable under the $\RB$-expansion, then each
double stabilizer $U={}_{s}S_{s'}$ is finite $\J$-trivial and an $\eH$-chain in $S$ but not necessarily within
$U$.  
\end{thm}

\begin{proof}
Since $S$ is stable under $\RB$, it is also stable under $\LZ$ and
$\RZ$ (Theorem~\ref{thm:stable}). Furthermore, ${}_{s}S_{s'}$ is a
subsemigroup of ${}_{ss'}S_{ss'}$ so without loss of generality we can
take $s=s'$ and apply Theorem~\ref{thm:right-stab} and its dual.
\end{proof}

\begin{Rmk}[Fixing stabilizers]\label{rem:fix-stab}
If $S$ is a finite $A$-semigroup stable under the $\RZ$-expansion and
the right stabilizers are not $\eL$-chains within themselves, then one
way to remedy this is by repeatedly applying the Henckell-Sch\"utz\-en\-ber\-ger
expansion to $S$.  In fact, Henckell, in his thesis, rediscovered the
Sch\"utzenberger product in order to prove this result.  The precise
statement goes as follows.  If $S$ is a finite $A$-semigroup, then the Henckell-Sch\"utzenberger expansion
$S^{\wedge_k}$ of $S$ with $k = 2^{|S|}$ has the property that right
stabilizers of $S^{\wedge_k}$ mapped into $S$ are $\eL$-chains within
themselves.  This was exploited by the second and third authors to show that stabilizers in certain relatively free profinite semigroups are $\eL$-chains within themselves. See~\cite{BR--exp,Henckellstable} for details.
\end{Rmk}

Another way to improve the properties of the right stabilizers is to
use the Mal'cev expansion with respect to the locally finite variety
$\Zp$, first studied by Le Saec, Pin and Weil in~\cite{idempotentstabilizers}. See
also~\cite{Elston}.

\begin{thm}[Idempotent stabilizers]\label{thm:zp-stab}
Let $S$ be a finite $A$-semigroup stable under the $\RZ$-expansion,
let $p$ be a prime, and let $T= S^\Zp$.  If $p$ is sufficiently large
relative to the size of $S$, then each right stabilizer $T_t$ is an
$\R$-trivial band.  In other words, each subsemigroup $T_t$ satisfies
the identities $xyx=xy$ and $x^2=x$.
\end{thm}

Combining Theorem~\ref{thm:zp-stab} with Theorem~\ref{thm:right-stab},
we have the following,  the major improvements being that all of the
right stabilizers are now idempotents, the stabilizers form an
$\eL$-chain within themselves, and they satisfy an additional equation
which is useful later in the article.

\begin{thm}[Improved stabilizers]\label{thm:improved-stab}
If $S$ is a finite $A$-semigroup, $p$ is a sufficiently large prime relative
to the size of $S^\RB$ and $T=S^{\RB.\Zp.\RB}$, then each right
stabilizer subsemigroup $T_t$ is an $\R$-trivial (also called left regular) band, i.e., its
satisfies the identities $xyx=xy$ and $x^2=x$.  Moreover, $T_t$ forms
an $\eL$-chain within itself so that for any two elements $r,s \in
T_t$, one is $\eL$-above the other, say $r\geq_{\mathscr L} s$, where
$\geq_{\mathscr L}$ now represents the $\eL$-order in $T_t$ rather than $T$.
In addition, whenever $r$ and $s$ are right $t$ stabilizers and
$r\geq_{\mathscr L} s$, the equation $sr = s$ also holds.  As a consequence,
the $\eL$-classes of $T_t$ are left zero semigroups.
\end{thm}

\begin{proof}
Since $S^\RB$ is stable under the $\RZ$ expansion, the right
stabilizer subsemigroups of $S^{\RB.\Zp}$ are $\R$-trivial bands.  Let
$T'=S^{\RB.\Zp}$ and let $t'$ denote the image of $t\in T$ under the
natural map $T\to T'$. Once it is noticed that $T_t$ is contained in the inverse image
of $T'_{t'}$ under this map and that the Mal'cev kernel of this map
lies in $\RB$ by definition, it is clear that $T_t$ consists solely of
idempotents.  Moreover, it is $\R$-trivial and an $\eL$-chain in $T$
by Theorem~\ref{thm:right-stab}.  Since the $\eL$-relation restricted to idempotents is independent of the ambient semigroup, this completes the proof.
\end{proof}

\section{The McCammond expansion revisited}
Let us investigate what happens when we apply the McCammond expansion to straightline automata and Cayley graphs.

\begin{Thm}
Let $(\mathscr A,I)$ be a pointed $A$-automaton.  Then, for $w\in A^+$, the pointed $A$-automaton $(\str^{\mathscr A^\Mac}(w),I)$ has the unique simple path property and is linear.
\end{Thm}
\begin{proof}
This follows immediately from Proposition~\ref{straightlineinuspp}.
\end{proof}

\begin{exmp}[Expanding straightline au\-to\-ma\-ta]
Suppose that $S$ is an $A$-semi\-group and $w$ is a word in
$A^+$.  If $\str^S(w)$ is the au\-to\-ma\-ton on the left in
Figure~\ref{fig:mc} (with the edge labels suppressed) and $w$ is the
path which passes through vertices $1,p,q,r,s,t$, and $u$ in that
order, then $\str^{(\str^S(w))^{\Mac}}(w)$ is the automaton on the right, where we have
simplified the vertex labels as well as leaving out the edge labels.
In particular, the labels on the right simply indicate the vertex in
$\str^S(w)$ to which it is sent under the projection map.  Similarly,
if $w'$ is the path which passes through vertices $1,p,q,r',s,t$, and
$u$ in that order, then $\str^S(w') = \str^S(w)$, but $\str^{(\str^S(w'))^{\Mac}}(w')\neq
\str^{(\str^S(w))^{\Mac}}(w)$.  Specifically, the automaton $\str^{(\str^S(w))^{\Mac}}(w)$ would be altered
by having the edge connecting $t$ to $u$ starting instead at the other
vertex labeled $t$ to obtain $\str^{(\str^S(w'))^{\Mac}}(w')$.
\end{exmp}

\begin{figure}[ht]
\psfrag{1}{$1$}
\psfrag{p}{$p$}
\psfrag{q}{$q$}
\psfrag{r}{$r$}
\psfrag{r'}{$r'$}
\psfrag{s}{$s$}
\psfrag{t}{$t$}
\psfrag{u}{$u$}
\begin{tabular}{cc}
\begin{tabular}{c}
\includegraphics{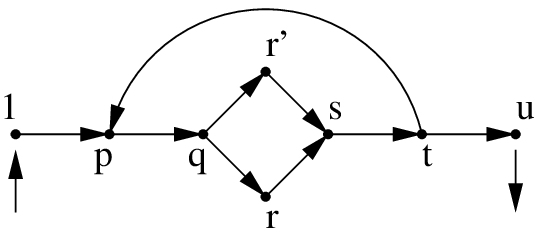}
\end{tabular}
&
\begin{tabular}{c}
\includegraphics{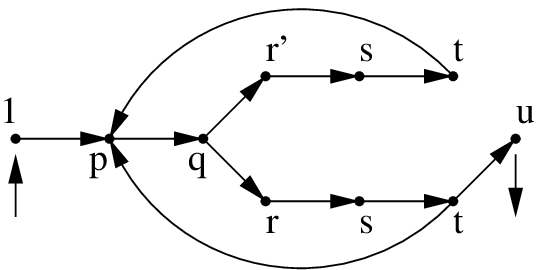}
\end{tabular}
\end{tabular}
\caption{Expansion of a straightline automaton.\label{fig:mc}}
\end{figure}

\begin{Def}[The $\Mac$-expansion of a semigroup]
If $S$ is a $A$-semigroup, then $S^{\Mac}$ will denote the transition semigroup of $\Cay(S,A)^\Mac$.  This is an object expansion on $A$-generated semigroups which preserves finiteness.  Note that $\Cay(S^{\Mac},A)$ is not $\Cay(S,A)^\Mac$.  In general, we will never be interested in $\Cay(S^{\Mac},A)$ but rather in the action of $S^\Mac$ on $\Cay(S,A)^\Mac$.
\end{Def}

\begin{exmp}[Not a Cayley graph]
Add example that $\Cay(S^{\Mac},A)$ is not $\Cay(S,A)^\Mac$.
\end{exmp}


\subsection{Properties of the $\Mac$-expansion}\label{sec:mc2}
In this section we prove that for each $A$-semigroup $S$, the $\Mac$-expansion of $S$ has Mal'cev kernel contained in a locally finite variety of aperiodic semigroups and so in particular preserves being finite $\J$-above (Theorem~\ref{thm:mc-divides}).  The key
step will be to show that the $\Mac$-expansion of the right zero
$A$-semigroup is a finite band.

\begin{Def}[Right zero semigroup]\label{const:constant-maps}
For each finite set $A$, let $A^r$ denote the right zero semigroup on
$A$ (i.e., $a\cdot b = b$ for all $a,b\in A$) and let $\Gamma_{A^r} =
\Cay(A^r,A)$.  Recall that by convention $\Cay(S,A)$ is the
Cayley graph for $S^I$ rather than $S$ itself, so that $\Gamma_{A^r}$ has
$|A|+1$ vertices.  If we use $1$ to denote the root vertex of
$\Gamma_V$ and define $A^1 = A\cup \{1\}$, then the edges of
$\Gamma_{A^r}$ can be described as follows.  For each $a\in A^1$ and for
each $b\in A$ there is an edge labeled $b$ from $a$ to $b$.  Viewed as
a function, $\cdot b$ is the constant function from $A^1$ to $A^1$
sending everything to $b$.  See Figure~\ref{fig:rightzero}.
\end{Def}

\begin{figure}[ht]
\psfrag{1}{$1$}
\psfrag{a}{$a$}
\psfrag{b}{$b$}
\psfrag{c}{$c$}
\includegraphics{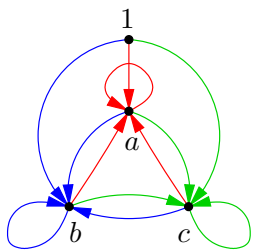}
\caption{The automaton $\Gamma_{A^r} = \Cay(A^r,A)$ when
  $A=\{a,b,c\}$.\label{fig:rightzero}}
\end{figure}

The simple paths starting at $1$ in $\Gamma_{A^r}$ have a particularly
easy description: a word $w\in A^*$ can be read as a simple path
starting at $1$ if and only if no letter occurs more than once in $w$.
We call these words \emph{distinct letter words}.  Thus the vertices
in $(\Gamma_{A^r})^\Mac$ are in one-to-one correspondence with the
distinct letter words.

\begin{figure}[ht]
\psfrag{1}{$1$}
\psfrag{a}{$a$}
\psfrag{b}{$b$}
\psfrag{c}{$c$}
\includegraphics{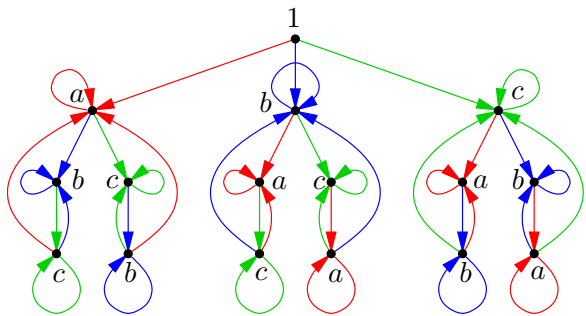}
\caption{The automaton $(\Gamma_{A^r})^\Mac$ when
  $A=\{a,b,c\}$.\label{fig:mc-rightzero}}
\end{figure}

\begin{Def}[Reduction]
Every word $w\in A^*$ can be read in $(\Gamma_{A^r})^\Mac$ starting at $1$
and ending at some vertex, say $v$.  By construction there is a unique
simple path in this graph from $1$ to $v$.  The word read by this
simple path is called the \emph{reduction of $w$} and denoted
$\red(w)$.
\end{Def}

For example, if $A =\{a,b,c,d,e\}$ then $\red(aba)=a$, $\red(abc) =
abc$ and $\red(abcdce)=abce$.  Here are two easy observations about
reductions.

\begin{lem}[Recurrence]\label{lem:recur}
Suppose $w\in A^+$ is a word with first letter $a$.  If $a$ recurs in
$w$, that is, $w = auav$ for possibly empty words $u, v \in A^*$,
then $\red(w) = \red(av)$.
\end{lem}

\begin{proof}
Reading one letter at a time it is easy to see that the path
corresponding to $w$ only passes through vertices whose label starts
with $a$.  Moreover, when $a$ recurs in $w$, the path $w$ returns to
the vertex labeled $a$.  Thus the paths corresponding to $w$ and to
$av$ end at the same vertex and hence have the same reduction.
\end{proof}

The most efficient use of Lemma~\ref{lem:recur} would, of course,
focus on the last occurence of $a$ in $w$ so that $a$ does not occur
in $v$.

\begin{lem}[No recurrence]\label{lem:no-recur}
Suppose $w\in A^+$ is a word with first letter $a$.  If $a$ does not
recur in $w$, then $w = av$ for some $v\in (A\setminus \{a\})^*$ and
$\red(w) = a\red(v)$.
\end{lem}

\begin{proof}
Let $B$ denote $A\setminus \{a\}$ and notice that there is an
label-preserving embedding of $(\Gamma_{B^r})^\Mac$ into $(\Gamma_{A^r})^\Mac$
which sends the root in the domain to the vertex labeled $a$ in the
range.  The result is an immediate consequence of this observation.
\end{proof}

\begin{exmp}
Using Lemma~\ref{lem:recur} and Lemma~\ref{lem:no-recur} it is easy to
calculate the reduction of any word.  For example, if $w =
abacdabdbccebgfdf$ then
\[\begin{array}{ccll}
\red(abacdabdbccebgfdf) &=& \red(abdbccebgfdf) & \textrm{Lemma~\ref{lem:recur}}\\
&=& a\red(bdbccebgfdf) & \textrm{Lemma~\ref{lem:no-recur}}\\
&=& a\red(bgfdf) & \textrm{Lemma~\ref{lem:recur}}\\
&=& ab\red(gfdf) & \textrm{Lemma~\ref{lem:no-recur}}\\
&=& abg\red(fdf) & \textrm{Lemma~\ref{lem:no-recur}}\\
&=& abg\red(f)  & \textrm{Lemma~\ref{lem:recur}}\\
&=& abgf & \textrm{Lemma~\ref{lem:no-recur}}
\end{array}\]
\end{exmp}

We can now show that $(A^r)^\Mac$ is a band.

\begin{lem}[Band]\label{lem:band}
If $A$ is a finite set, then the $\Mac$-expansion of the right zero
semigroup $A^r$ is a finite band.
\end{lem}

\begin{proof}
Note that it suffices to check that for each distinct letter word
$w\in A^*$ and for each word $u\in A^+$, $\red(wu) = \red(wuu)$ since
this shows that the functions $\cdot u$ and $\cdot uu$ agrees on each
vertex of $(\Gamma_{A^r})^\Mac$.  If $w$ is non-empty and the first
letter of $w$ does not occur in $u$, then applying
Lemma~\ref{lem:no-recur} to both $wu$ and $wuu$ reduces the problem to
a similar problem with a shorter word $w$.  Continuing in this way, we
may assume that either $w$ is the empty word or that the first letter
of $w$ occurs in $u$.  When $w$ is empty, a single application of
Lemma~\ref{lem:recur} shows that $\red(uu) = \red(u)$ and we are done.
Similarly, when $w$ is not empty and the first letter of $w$, say $a$,
occurs in $u$ then applying Lemma~\ref{lem:recur} to the last
occurence of $a$ in $wu$ and in $wuu$ produces the same result in each
case.  Thus they have the same reduction, which completes the proof.
\end{proof}

The following example shows that, in contrast with the Cayley graphs
of right zero semigroups, the Mal'cev kernel of $S(\mathscr A^\Mac)\to
S(\mathscr A)$, does not, in general, consist of bands.

\begin{exmp}[Not always a band]\label{exmp:not-band}
Let $\mathscr A$ denote the $A$-automaton shown
on the left-hand side of Figure~\ref{fig:band} and root the graph at
the vertex labeled $1$.  It is easy to check that $abc$ and $(abc)^2$ act the same on each vertex.  In particular,
they both send $1$ and $r$ to $r$ and they both fail when starting at
$p$ or $q$.  Thus $abc$ is an idempotent in $S(\mathscr A)$.
On the other hand, in $\mathscr A^\Mac$, shown on the right-hand side of
Figure~\ref{fig:band}, $abc$ and $(abc)^2$ act differently
since $1\cdot abc = r'$ and $1 \cdot (abc)^2 = r$.  Thus $abc$ no
longer represents an idempotent in $S(\mathscr A^\Mac)$.
\end{exmp}

\begin{figure}[ht]
\begin{tabular}{ccc}
\psfrag{1}{$1$}
\psfrag{a}{$a$}
\psfrag{b}{$b$}
\psfrag{c}{$c$}
\psfrag{p}{$p$}
\psfrag{q}{$q$}
\psfrag{r}{$r$}
\begin{tabular}{c}
\includegraphics{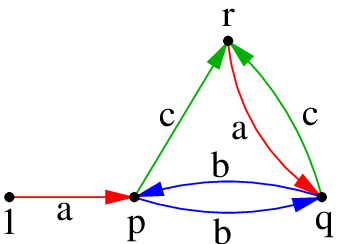}
\end{tabular}
& \hspace{1cm} &
\psfrag{1}{$1$}
\psfrag{a}{$a$}
\psfrag{b}{$b$}
\psfrag{c}{$c$}
\psfrag{p}{$p$}
\psfrag{q}{$q$}
\psfrag{q'}{$q'$}
\psfrag{r}{$r$}
\psfrag{r'}{$r'$}
\begin{tabular}{c}
\includegraphics{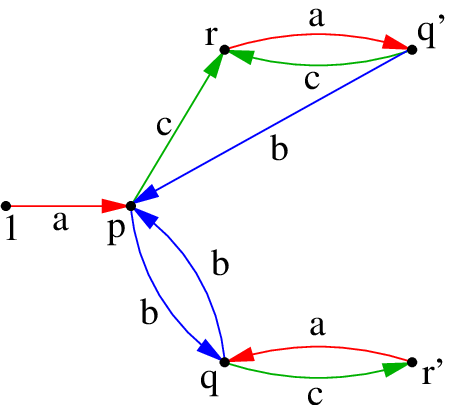}
\end{tabular}
\end{tabular}
\caption{An $\A$-automaton $\mathscr A$ and its expansion
  $\mathscr A^\Mac$.}\label{fig:band}
\end{figure}

Despite the fact that the Mal'cev kernel does not always lie in $\Ba$,
it is quite restricted.  Since we only intend to apply this expansion
to $A$-semigroups stable under the $\RB$-expansion, nothing is lost if
we restrict our attention to automata in which the label on an edge is a
function of its terminal vertex (Lemma~\ref{lem:rb-labels}).

\begin{thm}[Bounding the Mal'cev kernel]\label{thm:mc-divides}
If $\mathscr A=(V,A)$ is a pointed $A$-au\-to\-ma\-ton with root $I$
in which the label on an edge is a function of its terminal vertex,
then there is a finite band $B$ acting faithfully on a set $X$ such that $S(\mathscr A^\Mac)$ divides
$(X,B) \wr (V,S(\mathscr A))$ where $\wr$ denotes the wreath product of partial transformation semigroups.
\end{thm}
\begin{proof}
Let $V$ denote the vertex set of $\mathscr A$ and let $B$ be the $\Mac$-expansion of the $V$-semigroup $V^r$.   By Lemma~\ref{lem:band}, $B$ is
a finite band.  Put $X=\simple(V^r)$.  At this point, we can define a function $f\colon \simple(\mathscr A)\to X\times V$ as follows.   Associate to each word $u=a_1\cdots a_n$ of
$\simple(\mathscr A)$ the pair $f(u)=((I,Ia_1,\ldots,Ia_1\cdots a_{n-1}),Iu)$ where elements of $V^*$ are written as strings.  The map $f$ is injective
precisely because the edge labels are determined by their terminal
vertices, so that the word $u$ can be completely recovered from the
sequence of vertices it visits.   To each $a\in A$, we associate the pair $\wh a=(g_a,[a]_{S(\mathscr A)})\in  (X,B) \wr (V,S(\mathscr A))$ where $g_a\colon Q\to B$ is given by $qg_a = \cdot (qa)$ if $qa$ is defined and is arbitrary otherwise.  If $u\in \simple(\mathscr A)$ is as above and $u\cdot a$ is defined in $\mathscr A^\Mac$, then there are two cases.  If $ua$ is simple, then
\begin{align*}
f(u)\wh a&=((I,Ia_1,\ldots,Ia_1\cdots a_{n-1}),Iu)\wh a\\ &= ((I,Ia_1,\ldots,Ia_1\cdots a_{n-1},Iu),Iua)= f(u\cdot a).
\end{align*}
Otherwise, $u\cdot a = I\cdot a_1\cdots a_j$ for the unique $1\leq j\leq n-1$ so that one has $Iua=I\cdot a_1\cdots a_j$.  But then $\red(I,Ia_1,\ldots,Ia_1\cdots a_{n-1},Iu) = (I,Ia_1,\ldots,Ia_1\cdots a_j)$.  Thus
\begin{align*}
f(u)\wh a&=((I,Ia_1,\ldots,Ia_1\cdots a_{n-1}),Iu)\wh a\\ &= (\red(I,Ia_1,\ldots,Ia_1\cdots a_{n-1},Iu),Iua)\\ &= ((I,Ia_1,\ldots,Ia_1\cdots a_j),Iua)=f(u\cdot a).
\end{align*}

It now follows that $S(\mathscr A^\Mac)\prec (X,B) \wr (V,S(\mathscr A))$ (c.f.~\cite{Eilenberg}).
\end{proof}

Recall that if $\mathcal V$ and $\mathcal W$ are varieties, then the semidirect products $\mathcal V\ast \mathcal W$ is contained in the Mal'cev product $(\LC\malce \mathcal V)\malce \mathcal W$ where we recall that
 $\LC$ is the locally finite variety defined by
the identity $xy=xz$~\cite[Chapter 2]{qtheor}. Thus by Theorem~\ref{thm:mc-divides}, the projection
$S^\Mac\to S$ has a Mal'cev kernel which lies in the locally finite aperiodic variety $\LC\malce \mathcal B$ if $S$ is stable under the $\RB$-expansion.  Finally,
Lemma~\ref{lem:malcev} and Lemma~\ref{lem:rb-labels} complete the
proof of the following corollary of Theorem~\ref{thm:mc-divides}.

\begin{cor}[Bounding $S^\Mac$]\label{cor:mc-divides}
If $S$ is an $A$-semigroup stable under the $\RB$-expansion, then the
Mal'cev kernel of $S^\Mac\to S$ lies in the locally finite aperiodic variety $\LC\malce \mathcal B$. As a result, there
is an $A$-morphism \mbox{$S^{\LC\malce \mathcal B}\to S^\Mac$}.
\end{cor}

Finally, we note that even though the inverse images of idempotents is
not always a band (Example~\ref{exmp:not-band}), the torsion in the
Mal'cev kernel is quite controlled.

\begin{lem}[Bounding torsion]
If $S$ is an $A$-semigroup stable under the $\RB$-expansion, then the
semigroups in the Mal'cev kernel of $S^\Mac\to S$ satisfy the identity
$x^2=x^3$.  As a result, if $S$ satisfies the identity $x^m=x^{m+n}$
for constants $m$ and $n$, then $S^\Mac$ satisfies the identity
$x^{m+1}=x^{(m+1)+n}$ for the same constants $m$ and $n$.
\end{lem}

\begin{proof}
It suffices to observe that the variety $\LC\malce \mathcal B$ clearly satisfies the identity $x^2=x^3$ since in any band $x$ and $x^2$ map to the same idempotent and so by definition of $\LC$ we have $xx=xx^2$.
\end{proof}
\section{Algebraic rank function}\label{sec:rank}
In this section we begin by defining an algebraic rank function.  For
Kleene expressions the rank is the nested star-height, for loop
automata it is the maximum number of loops within loops. In the case
of the Burnside semigroups, $\burnside(m,n)$ for $m\geq 6$ and $n\geq
1$, the first author proved that the automata, Kleene and algebraic
definition all agree for $\str(w)$ in~\cite{Mc91}.  He also proved
that these semigroups are finite $\J$-above. One main idea of
geometric semigroup theory is to run this procedure backwards.  That
is, to start with a finite $\J$-above semigroup, define the rank
algebraically and then prove it is the automata rank in many cases.
This also has strong ties with the holonomy theorem for semigroups
\cite{Rh91}.

\begin{Def}[Algebraic rank function]
Let $S$ be a finite $\J$-above $A$-semigroup and let $s$ be
one of its element.  Since $S$ is torsion, there is a unique
idempotent in the cyclic subsemigroup generated by $s$ which we denote
$s^\omega$.  Moreover, since $S$ is finite $\J$-above, there are only
a finite number of idempotents in $S$ which are $\J$-above $s^\omega$.
We define the \emph{algebraic rank of $s$} to be the length of the
longest strictly increasing $\J$-chain of idempotents starting at
$s^\omega$.  This defines a map $\rank_S\colon S\to \{0,1,2,\cdots\} = \N$.
More specifically, \[\rank_S(s) = \max\{r \mid s^\omega = e_0 <_{\mathscr J} e_1
<_{\mathscr J} \cdots <_{\mathscr J} e_r\}\] where all $e_i$ are idempotents in $S$.
\end{Def}

When defining a geometric rank function on $\Cay(S,A)^\Mac$, we will always assume that the ordering on the loops at a vertex refines the algebraic rank of the images in $S$ of the words labeling the loops.

The following lemma records some elementary properties of the
algebraic rank function.

\begin{lem}[Elementary properties]\label{lem:alg-rank}
Let $S$ be a finite $\J$-above $A$-semi\-group and let
$\rank_S\colon S\to \N$ denote its algebraic rank function.  If $s$ and $t$
are elements of $S$ and $e$ and $f$ are idempotents in $S$, then:
\begin{itemize}
\item $\rank_S(s)= \rank_S(s^k)$ for all $k\geq 1$;
\item $\rank_S(e)\geq \rank_S(f)$ whenever $e$ is $\J$-above $f$;
\item $\rank_S(e)=\rank_S(f)$ wherever $e$ and $f$ are $\J$-equivalent;
\item $\rank_S(e)=\rank_S(f)$ and $e \geq_{\mathscr J} f$ implies $e =_{\mathscr J} f$;
\item $\rank_S(st)=\rank_S(ts)$ for all $s,t\in S$.
\end{itemize}
\end{lem}

\begin{proof}
The only part which is not immediate is the last one, but this follows
from the first statement once it is noticed that $(st)^\omega$ is always
$\J$-equivalent to $(ts)^\omega$.
\end{proof}


The next lemma studies the effect of $\J'$-maps on the rank.

\begin{lem}[Rank and $\J'$-maps]
If $S\to T$ is an onto $\J'$-map between finite semigroups, then
$\rank_S(s)=\rank_T(\p(s))$ for all $s\in S$.
\end{lem}

\begin{proof}
By assumption the map $\p$ has the property that $\p^{-1}$ of a
regular $\J$-class consists of one regular $\J$-class plus maybe some
null elements above it.  Hence the restriction of the $\J$-order to
the regular $\J$-classes results in the same poset for both $S$ and
$T$.  The fact that $\rank_S=\rank_T$ now follows immediately.
\end{proof}

On the other hand, the rank function can fall an arbitrarily large
amount under other types of maximal proper surmorphisms~\cite{qtheor,Kernel,RhodesWeil} that are not $\J'$-maps.

\begin{Rmk}
For example, consider the semilattice $S=\{0,1,\cdots ,n\}$ with $\max$. Then
$\rank_S(i)=i$.  Now put $S' = S\cup \{n'\}$ where $n'j=jn'=n$ for $0\leq j\leq n$ and $n'n'=n'$; so $S'$ is obtained from $S$ by inflating $n$ to $\{n,n'\}$.  Then $\rank_{S'}(n')=0$, but its image under the natural map $S'\to S$ that identifies $n$ with $n'$ has rank $n$.  This map is a maximal proper surmorphism of class $III_{R >R}$ in the terminology of~\cite{RhodesWeil} and the Mal'cev kernel of the surmorphism is in
$\script{SL}$.
\end{Rmk}

\begin{lem}[Lifting Lemma]\label{lem:lifting}
Let $S$ be finite $\J$-above and let $\p\colon S\to T$ be an
onto map.  For each $t\in T$ there exists an $s \in S$ such that
$\p(s^\omega)=t^\omega$ and $\rank_S(s)\geq \rank_T(t)$.
\end{lem}

\begin{proof}
Suppose $t^\omega=e_0 <_{\mathscr J} e_1 <_{\mathscr J} \cdots <_{\mathscr J} e_r$, where $<_{\mathscr J}$
denotes the $\J$-order in $T$. First choose an idempotent $f_r$ in $S$
which maps to $e_r$, then in $Sf_rS$ choose an idempotent $f_{r-1}$
which maps to $e_{r-1}$, etc.  This yields idempotents
$f_0<\cdots< f_r$ in $S$ with $\p(f_i)=e_i$.  Setting
$s=s^\omega=f_0$ establishes the result.
\end{proof}

Recall that idempotents in a semigroup are partially ordered by putting $e\leq f$ if $ef=e=fe$, or equivalently, $e\leq_{\eH}f$.

\begin{lem}[Alternate definition]
Let $S$ be a finite $\J$-above $A$-semigroup.  The function $f\colon S\to
\N$ defined by \[f(t) = \max\{r\mid t^\omega\J e_0 < e_1 <\cdots <
e_r\},\] where each $e_i$ is an idempotent, is the same as the rank
function $\rank_S(t)$.
\end{lem}

\begin{proof}
By an elementary lemma from~\cite{resultsonfinite}, if $J$ and $J'$ are regular
$\J$-classes in a finite semigroup with $J > J'$, then for each
idempotent $e \in J$, there exists an idempotent $f \in J'$ so that $e
>f$.  Indeed, if $e'=uev\in J'$ is an idempotent set $f=eve'ue$. Then $f^2=evuevueevuevue=eve'e'e'ue=eve'ue=f$ and $ufv= ueve'uev=e'e'e'=e'$; so $f\in J'$ is an idempotent with $f<e$.   This plus the fact that $\J$-equivalent idempotents have the
same rank proves the result.
\end{proof}

The following corollary is immediate.

\begin{Cor}[Subsemigroups]
If $S$ is finite $\J$-above  and $T$ is a subsemigroup of
$S$, then $T$ is finite $\J$-above and, for all $t\in T$, the inequality
$\rank_T(t)\leq \rank_S(t)$ holds.
\end{Cor}

The importance of the algebraic rank function is as follows.  Suppose that $S$ is an $A$-semigroup and let $\eta\colon \Cay(S,A)^{\Mac}\to \Cay(S,A)$ be the corresponding simple directed cover.  Let $v$ be a vertex of $\Cay(S,A)^{\Mac}$.  Then $[w]_S$ stabilizes $\eta(v)$.  Often $S$ will already have been expanded to have idempotent stabilizers.  Thus the loops at $v$, when mapped into $S$, will be an $\eL$-chain of idempotents and we want to choose a geometric rank function accordingly.  More formally, we shall always choose a geometric rank function $r$ for $\Cay(S,A)^{\Mac}$ so that the following occurs.  If $e,f$ are bold arrows ending at $v$ and $y,z$ are the labels of $\lp(e),\lp(f)$, respectively, then we shall assume that if $\rank_S(y)< \rank_S(z)$ then $r(e)<r(f)$.  In other words, the order on bold arrows at $v$ given by the geometric rank function will be a topological sorting of the algebraic rank of the labels of the loops (taken in $S$).

\bibliographystyle{abbrv}
\bibliography{standard2}
\end{document}